\documentclass{amsart}
\usepackage{graphicx}
\usepackage{esint}
\usepackage{xcolor}
\usepackage{amsmath,amsfonts,amssymb,amsthm,enumerate,ulem}
\usepackage{enumitem}
\usepackage{float}
\usepackage{color}
\usepackage[mathscr]{eucal}
\usepackage[sectionbib]{chapterbib}
\usepackage{comment}
\usepackage{hyperref}
\usepackage{mathtools}
\usepackage{stmaryrd}

\newenvironment{changemargin}[2]{\begin{list}{}{%
\setlength{\topsep}{0pt}%
\setlength{\leftmargin}{0pt}%
\setlength{\rightmargin}{0pt}%
\setlength{\listparindent}{\parindent}%
\setlength{\itemindent}{\parindent}%
\setlength{\parsep}{0pt plus 1pt}%
\addtolength{\leftmargin}{#1}%
\addtolength{\rightmargin}{#2}%
}\item }{\end{list}}

\let\oldtocsection=\tocsection
\let\oldtocsubsection=\tocsubsection
\let\oldtocsubsubsection=\tocsubsubsection

\renewcommand{\tocsection}[2]{\hspace{0em}\oldtocsection{#1}{#2}}
\renewcommand{\tocsubsection}[2]{\hspace{1em}\oldtocsubsection{#1}{#2}}
\renewcommand{\tocsubsubsection}[2]{\hspace{2em}\oldtocsubsubsection{#1}{#2}}

\providecommand{\R}{\mathbb{R}}

\providecommand{\N}{\mathbb{N}}
\providecommand{\Z}{\mathbb{Z}}

\providecommand{\Id}{\mathrm{Id}}
\providecommand{\eps}{\varepsilon}
\providecommand{\de}{\partial}
\providecommand{\dd}{\,\mathrm{d}}
\providecommand{\dt}{\,\mathrm{d}t}
\providecommand{\dx}{\,\mathrm{d}x}
\providecommand{\dy}{\,\mathrm{d}y}
\providecommand{\ds}{\,\mathrm{d}s}
\providecommand{\dz}{\,\mathrm{d}z}
\renewcommand{\leq}{\leqslant}
\renewcommand{\geq}{\geqslant}

\renewcommand{\div}{\operatorname{div}}
\newcommand{\curl}{\operatorname{curl}}
\newcommand{\supp}{\operatorname{Supp}}
\newcommand{\const}{\operatorname{const.}}
\newcommand{\dist}{\operatorname{dist}}
\newcommand{\norm}[1]{\left\|#1\right\|}
\newcommand{\scalar}[2]{\left\langle#1,#2\right\rangle}
\newcommand{\seps}{\tilde{\mathcal{S}}^\eps}
\newcommand{\uref}{u_{\mathrm{ref}}}
\newcommand{\cals}{\mathcal{S}_0}
\newcommand{\Supp}{\mbox{Supp\,}}

\definecolor{green2}{RGB}{0,173,58}
\definecolor{FranckiesPurple}{RGB}{147,101,214}
\definecolor{ogbrown}{RGB}{237, 137, 49}
\newcommand{\David}[1]{{\ \par \noindent \textcolor{green2}{David: #1}}}
\newcommand{\Franck}[1]{{\ \par \noindent \textcolor{FranckiesPurple}{Franck: #1}}}

\newcommand{\mres}{\mathbin{\vrule height 1.6ex depth 0pt width
0.13ex\vrule height 0.13ex depth 0pt width 1.3ex}}

\newtheorem{thm}{Theorem}[section]
\newtheorem{cor}[thm]{Corollary}
\newtheorem{lem}[thm]{Lemma}
\newtheorem{prop}[thm]{Proposition}
\newtheorem{defn}[thm]{Definition}
\newtheorem{rem}[thm]{Remark}

\numberwithin{equation}{section}

\title{Motion of a massive rigid loop in a 3D perfect incompressible flow} 
\author{Olivier Glass, David Meyer, Franck Sueur} 
\date{\today}
\begin{document}
%
%
%
%
%
%
%
%

\begin{abstract}
  We consider the motion of a rigid body immersed in an inviscid incompressible fluid. In 2D, an important physical effect associated with this system is the famous Kutta-Joukowski effect.  
 In the present paper, we identify a similar effect in the 3D case.
 For this, we first recast the  Newtonian dynamics of the rigid body as a first-order nonlinear ODE for the $6$-component body velocity, in the body frame. 
Then, we focus on the particular case where the rigid body occupies a slender tubular domain with a smooth closed curve as the centerline and a circular cross-section, in the limit where the radius goes to zero, with fixed inertia and circulation around the curve. 
 We establish that the dynamics of the limit massive rigid loop are given by a first-order nonlinear ODE with coefficients that depend only on the 
 inertia, on the fluid vorticity, and on the limit curve through two $3$D vectors, which are involved in a  {skewsymmetric} $6 \times 6$ matrix that appears in the limit force and torque, a structure which is reminiscent of the 2D Kutta-Joukowski effect.
We also identify the limit fluid dynamics as, where, as in the case of the Euler equation alone, the vorticity evolves according to the usual transport equation with stretching, but with a velocity field that is due not only to the fluid vorticity but also to a vorticity filament associated with the circulation around the limit rigid loop. 
 This result is in stark contrast with the case where the filament is made of fluid, with non-zero circulation, since in the latter, the filament velocity becomes infinite in the zero-radius limit.
However,  considering the inertia scaling that corresponds to a fixed density, we prove that there are solutions for which the solid velocity and its displacement tend to infinity over a time interval of size $\mathcal{O}(1)$.
 \end{abstract}

\maketitle

\begin{changemargin}{-1cm}{-1cm}
\tableofcontents
\end{changemargin}


\section{Introduction}

In this paper, we consider the \textbf{motion of a rigid body immersed in a 3D perfect incompressible flow}. 
The motion of this rigid body is given by the Newton equations, which involve its inertia, its translation and rotation acceleration, together with force and torque due to the fluid pressure on the solid boundary. On the other hand, the fluid dynamics are given by the incompressible Euler equations.
To simplify, we consider the case of a single rigid body, and we assume that the fluid around it occupies the rest of the Euclidean space. 
We assume the non-penetration condition on the boundary, that is, that at the interface between the fluid and the rigid body, there is continuity of the normal component of the fluid and rigid velocities. 
This setting provides a system coupling an ODE for the $6$ degrees of freedom of the rigid body and a PDE for the fluid. 
This system is precisely described in Section~\ref{initeq}.
It has been the subject of many mathematical studies in recent years.
Let us already say here that it is a conservative system: at least formally, the total kinetic energy of the system is preserved in time, this is recalled in Section~\ref{subsec-en}; and that the existence and uniqueness of classical solutions to this system, for short times, is well understood, this is recalled in Section \ref{Subsec-Cauchy}.
\smallskip

A classical topic, relevant e.g.\ for aeronautics, is the computation of circulation-induced lift forces. In $2$D, such a force is the so-called Kutta-Joukowski force, which was first discovered by  Kutta and   Joukowski at the beginning of the 20th century. In particular, the force is responsible for the lift of an airfoil translating in an inviscid fluid at a constant speed. In this analysis, one considers the section of the airfoil as a two-dimensional body within a two-dimensional fluid.
Another important simplification in this theory is that one ignores the fluid viscosity, as well as the effect of the fluid vorticity.
 Despite these rough assumptions, this analysis turned out to be relevant, to some extent, for applications to aerodynamics. 
We refer here to \cite{Gu,Howe,Lamb,Milne} for more. 
Recently, this analysis has been extended to the case of small rigid bodies moving in a 2D perfect incompressible fluid, see for instance \cite{GMS,Munnier,Sueur17,Sueur-china,GS}. The outcome is that a force similar to the  Kutta-Joukowski force appears in the zero radius limit, that is, the limit dynamics of the point particles involve a force reminiscent of the Kutta-Joukowski force. 
Our goal here is to investigate the role of such a \textbf{Kutta-Joukowski effect in the 3D case}. 
\smallskip

Toward that goal, an intuitive approach is to try to reformulate the Newton equations to decouple the six degrees of freedom of the rigid body as much as possible from the fluid influence. 
 In the present case of a single rigid body, it is convenient to consider the Newton equations \textbf{in the body frame}. It is interesting to recall first what is known in the historical case; it is the case where the fluid is assumed to be potential.  Then, the fluid velocity is only due to the motion of the rigid body and can be decomposed thanks to the so-called Kirchhoff potentials, which only depend on the geometry, with coefficients depending only on the rigid velocity. The latter is given by  ${p}  \in \R^{6}$ (the first three coordinates corresponding to the translation velocity and the three other ones to the rotation velocity) and then satisfies an ordinary differential equation of the form
\begin{equation} \label{anoter1-intro}
({\mathcal M}_{g} + {\mathcal M}_{a} ) p' + \langle \Gamma_g, p , p \rangle + \langle  \Gamma_a ,  p ,  p \rangle 
=  0 ,
\end{equation}
where  $\mathcal M_g $ is a $6 \times 6$ symmetric positive definite matrix encoding the genuine inertia of the rigid body, 
  $\mathcal M_{a}$  is a $6 \times 6$  symmetric and positive-semidefinite 
  encoding the added inertia of the rigid body, 
   $\Gamma_{g}: \mathbb{R}^{6} \times \mathbb{R}^{6} \rightarrow \mathbb{R}^{6}$ is a bilinear symmetric mapping, encoding the Coriolis effect due to the change of frame, and 
 $\Gamma_{a}: \mathbb{R}^{6} \times \mathbb{R}^{6} \rightarrow \mathbb{R}^{6}$ is 
another  bilinear symmetric mapping, which encodes the variation of the added inertia in the original frame, see Section \ref{Subsec:COF}, Section \ref{Subsec:COF2} and \eqref{DefGammaa}.   Thus, a full decoupling of the rigid body dynamics occurs in the case where the fluid is assumed to be potential.
Such a reformulation is well-known, see for example \cite{Munnier}. 
Let us point out that this equation is obtained under the assumption that the circulations around the rigid body are zero and that the fluid vorticity vanishes. 
The impact of nonzero circulation and nonzero fluid vorticity is precisely the subject of Section \ref{Subsec:COF};  the main result is given in Proposition~\ref{reform-macro} where we obtain the counterpart of the equation \eqref{anoter1-intro} around the body: 
\begin{equation} \label{anoter2-intro}
({\mathcal M}_{g} + {\mathcal M}_{a} ) p' + \langle \Gamma_g, p , p \rangle 
=  {\mathcal B}[ u ] p + {\mathcal D} [ p, u, \omega ]   .
\end{equation}
where $\omega := \curl u$ is the fluid vorticity, 
  $u$ is the fluid velocity, 
${\mathcal B}[u]$  is a \textbf{skewsymmetric} $6 \times 6$ matrix, which depends on the trace of $u$ on the rigid body boundary  and 
 ${\mathcal D} \lbrack p, u, \omega \rbrack$ is a vector  in $\R^6$, which  linearly depends on $\omega$. Let us point out that the term $ {\mathcal B}[ u ] p$ reduces to 
  $-\langle  \Gamma_a ,  p ,  p \rangle$ in the potential case. 
The skew-symmetry of this term
 is reminiscent of the Kutta-Joukowski effect. Moreover, this term is linear with respect to the trace of the fluid velocity at the body surface, which is close, yet different, from the circulations, for which, as recalled above, only the tangential component is involved. 
\smallskip

To go further, in Section~\ref{Subsec:toro}, we focus on the case of a simple geometry with a circulation: the case where the rigid body is a \textbf{thick rigid closed simple filament}. In this case, only one circulation, which we call $\mu$, comes into play. 
In particular, we will distinguish various contributions in the fluid velocity $u$: a potential part due to the motion of the rigid body, a part due to the circulation around the filament, and a last part due to the vorticity. 
This allows us to split the term $ {\mathcal B}[ u ] p$ into three parts and  
 to establish in Proposition~\ref{reup} that the dynamics are given by an equation of the form: 
\begin{equation} \label{eq-intro}
\big(  {\mathcal M}_{g} + {\mathcal M}_{a} \big) p' 
  + \langle  \Gamma_g ,  p ,  p \rangle + \langle  \Gamma_a ,  p ,  p \rangle
=  \mu B p + {\mathcal B}  \left[ K_{ \mathcal F_0} [\omega] \right] p
  + {\mathcal D} \lbrack  p, u, \omega \rbrack ,
\end{equation}
where $B$ is a $6 \times 6$ matrix depending only on the geometry, 
  and $K_{ \mathcal F_0}$ is the Biot-Savart law in the fluid domain. This extends to the $3$D case the earlier results obtained in the 2D case in \cite{GMS,Munnier,Sueur17,Sueur-china}.
\smallskip

To simplify yet further the geometry, in Section~\ref{sec-z}, we investigate the \textbf{zero-radius limit} where the cross-section of the domain occupied by the rigid body is assumed to shrink, so that it converges to a \textbf{$3D$ closed simple curve} $\mathcal{C}_0$. To simplify, we consider the case where the cross-section is circular. Regarding the inertia coefficients, we will consider the case of a \textbf{massive filament}, that is, the case where the mass and the rotational inertia matrix  are assumed to be independent of 
  $\varepsilon$. The circulation $\mu$ around the body is considered fixed independently of $\eps$.

\begin{figure}[h]
    \begin{minipage}{0.2\linewidth}
    \centering
    \includegraphics[width=\linewidth]{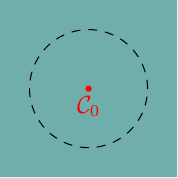}
\end{minipage}
\hfill $\longrightarrow$ \hfill
\begin{minipage}{0.50\linewidth}
    \centering
    \includegraphics[width=\linewidth]{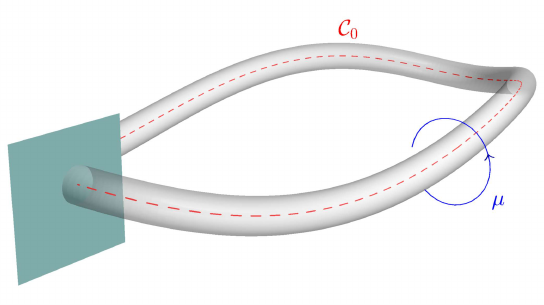}
\end{minipage}
\label{fig:solid}
\caption{Thick rigid filament with circular cross section and non-zero circulation}
\end{figure}
In the case where the fluid is irrotational, the main outcome is to identify the limit dynamic, which is given by the following first-order ordinary differential equation (cf.\ Thm.\ \ref{pasdenom-VO}):
\begin{equation} \label{eql1-intro}
{\mathcal M}_{g}  p' + \langle  \Gamma_g ,  p ,  p \rangle
 = \mu B^* p .
\end{equation}
where $B^*$ is the $6 \times 6$ matrix: 
\begin{equation} \label{defBstar-intro}
B^* :=  
\begin{pmatrix}
   0  &   \mathcal A_0 \wedge \cdot\\ 
   -\mathcal A_0 \wedge \cdot  &  \mathcal V_0 \wedge \cdot
\end{pmatrix} ,
\end{equation}
where $\mathcal A_0$ and $\mathcal V_0$ are respectively the vector area and the volume vector, defined as 
\begin{equation}
\mathcal A_0 := \frac12 \int_{\mathcal C_0} x \wedge \dd\kappa_{\mathcal{C}_0} (x) \in \R^3 
\quad \text{ and } \quad 
\mathcal V_0 :=-\frac{1}{2} \int_{\mathcal C_0} \vert x \vert^2 \dd\kappa_{\mathcal{C}_0} (x)   \in \R^3 , \label{def av}
\end{equation}
where $\kappa_{\mathcal{C}_0}$ is the  vector-valued measure $\kappa_{\mathcal{C}_0}$ 
obtained from the restriction of the one-dimensional Hausdorff measure on ${\mathcal C}_0$: 
\begin{equation} \nonumber
\kappa_{\mathcal{C}_0} := \tau \, \mathcal{H}^1 \, \mres {\mathcal C_0}, 
\end{equation}
where $\tau$ is the unit tangent vector along the curve ${\mathcal C}_0$. 
The structure of the term in the r.h.s. of \eqref{eql1-intro} is very much reminiscent of the one identified in 2D by Kutta and Joukowski, in the sense that it agrees with the force one would obtain from the 2D Kutta-Joukowski effect if one integrates it over the curve $\mathcal{C}_0$.
However, the structure of the force seems to be new, even compared to the physics literature.

Going back to the original laboratory frame, the equations take the following form 
\begin{equation*}
 m  (h^* )''  =  \mu  \mathcal A \wedge R^* , \quad \text{and} \quad 
(\mathcal{J} R^* )' = -\mu  \mathcal A \wedge (h^* )'  + \mu \mathcal V  \wedge R^* ,
\end{equation*}
where $h^*(t)$ and $R^*(t)$ are the position of the center and  the angular velocity of the limiting infinitesimal body, see Section \ref{initeq} and, similarly to \eqref{def av}, we associate with the time-dependent curve $\mathcal C(t)$
a unique vector-valued measure $\kappa_{\mathcal C} $, as well as (time-dependent) area and volume vectors: 
\begin{equation*}
\mathcal A := \frac12 \int_{\mathcal C} (x-h^*(t))\wedge \dd\kappa_{\mathcal C} (x) \in \R^3 
\quad \text{ and } \quad 
\mathcal V :=-\frac{1}{2} \int_{\mathcal C} \vert x-h^*(t) \vert^2 \dd\kappa_{\mathcal C} (x) \in \R^3 .
\end{equation*}

\smallskip 
We then extend the analysis to the case where some vorticity is present in the fluid away from the filament. In this case, in Theorem~\ref{pasdenom-VO2} we identify an extra term in the limit dynamics of the limit loop due to the fluid vorticity. 
When vorticity is present, we also consider the case where the inertia follows a scaling corresponding to a fixed density; in particular, the solid's mass converges to zero. In this case, in Theorem~\ref{thm:div} we show that the solid's velocity and its displacement can diverge to $+\infty$ over a time interval of size $\mathcal{O}(1)$, which is in striking contrast to the $2$D case \cite{GMS,GS}, where such a system converges to the point vortex system.

\smallskip 
It is interesting to compare the results above to the case where the ambient fluid perturbation is driven by the \textbf{steady Stokes system}. 
In particular, in \cite{HPS}, the authors study the motion of several slender rigid filaments in a fluid 
whose velocity and pressure are given as the sum of a background flow and a perturbation part 
given by the steady Stokes system.
The rigid filaments occupy disjoint, well-separated, closed, curved tubes of fixed lengths 
and possibly non-circular cross-sections and of radius converging to zero,  
 while the volumetric mass density is fixed in  \cite{HPS},
so that the limiting filament has zero inertia. 
 The main result of \cite{HPS} establishes that the limit dynamics of the limit filaments are given by decoupled first-order ODEs involving renormalized Stokes resistance tensors and renormalized Faxén-type forces and torques, associated with the limit curves. These effects are of a very different nature compared to the 3D  Kutta-Joukowski effects identified in this paper.

\smallskip 
Our results are also in stark contrast with the \textbf{vortex filament case}, where instead of a body, the thick filament is made of fluid with a concentration of vorticity; 
 oriented in the direction of the tangent
(which generates a non-zero circulation around $\mathcal{C}$), that is,
the case of the so-called vortex filaments. 
The most classical vortex filament is an axially symmetric solution of the 3D incompressible Euler which does not change shape in time and whose vorticity is concentrated inside a solid torus.
These objects were first described by Helmholtz in his celebrated work \cite{He1,He2}, in particular the case where the vorticity field is concentrated in a circular vortex-filament of very small cross section. 
He observed that such a ring moves with constant speed along the symmetry axis with a velocity that becomes infinite in the zero-radius limit.  
Formal asymptotics of thin vortex tubes around curves were first obtained by Da Rios in 1906 \cite{rios}, and suggest that they evolve by their binormal flow (after time-rescaling). 
Recently, many rigorous mathematical investigations have been done; see, for instance, \cite{jss,JS,Da}, and references therein. 

\smallskip 

Our analysis relies on a refined \textbf{asymptotic analysis of the fluid velocity}, which is of interest by itself. 
More precisely, we establish in Section~\ref{sec-Asymptotic} some rigorous estimates of different parts of the velocity field: the Kirchhoff potentials, the Biot-Savart field, and the harmonic field.
These estimates also apply in the case where the fluid occupies the exterior of a filament which is fixed in the laboratory frame, and are already new in this case. Actually, such a case formally corresponds to the limit case where the filament has an infinite inertia. 
Comparable analysis of the asymptotic behaviour in the exterior of a filament is well-known for the harmonic equation, with Dirichlet or Neumann condition, was performed in particular in \cite{Maz}. However, the scale techniques that were applied there do not directly apply to the vector setting here. 
  Let us also mention the recent work \cite{Ohm} for a different approach.
On the other hand, in the case considered in  \cite{HPS},  where the ambient fluid perturbation is driven by the steady Stokes  system,  the authors make use of the Bogovskii operator and of Helmholtz’ minimum dissipation principle 
to establish that the main part of the perturbation flow due to the filaments satisfies, up to a logarithmic renormalization, a modified Stokes system in the full space,  with a source term corresponding to Dirac masses along the limit curves with some amplitudes depending on the limit rigid velocities and on the background velocity. 

\smallskip

In general, the analysis of the vanishing body limit in fluid-solid systems is a very natural problem, but also a very difficult problem, and completely open for 3D Euler aside from the results of this paper. In 2D, it is known that the system converges to the vortex wave system, see, e.g.\ \cite{GMS,GS}. For viscous fluids, on the other hand, there have been many recent results showing the convergence to the classical \textbf{Navier-Stokes equations} in the limit, see for instance 
\cite{BFRZ,FRZ,LT,HI,EMT}.


\section{Motion of a rigid body immersed in a perfect incompressible flow}

%
%

\subsection{Equations for a rigid body immersed in a perfect incompressible flow}
\label{initeq}

We consider the motion of a rigid body immersed in a perfect incompressible flow. 
We denote by $h(t)$ the position of its center of mass at time $t \geq 0$, and we choose the origin of the frame of reference of $\R^3$ such that $h (0) = 0$.
The body rigidly moves so that at time $t \geq 0$ it occupies a domain $\mathcal{S}(t)$, which is obtained 
by means of a rigid movement with respect to its initial location $\mathcal{S}_0$,
which is supposed to be a connected (but not necessarily simply connected) smooth compact subset of $\mathbb{R}^{3}$.
Correspondingly, there exists a rotation matrix $Q (t) \in SO(3)$, with $Q(0) = \text{Id}$, 
such that the position of a point attached to the body with an initial position $x$ is moved to 
$h (t) + Q (t)x $ at the time $t \geq 0$.
Moreover, for any $t \geq 0$, there exists a unique vector $\Omega(t)$ in $\R^3$ such that for any $x \in \R^3$, 
\begin{equation} \label{Qom}
Q^{T} Q' (t)  x = \Omega(t) \wedge x ,
\end{equation}
where $Q^{T}$ denotes the transpose matrix of $Q$.
Accordingly, the solid velocity is given by
\begin{equation} \label{vit-sol}
v_{{\mathcal S}}(t,x) := h'(t) + R(t) \wedge (x-h(t)),   \text{ with } R(t) := Q(t)\Omega(t) .
\end{equation}
We denote by $m >0$ the mass of the body and by ${\mathcal J} (t)$ its rotational inertia matrix at time $t \geq 0$, which evolves in time according to Sylvester's law:
\begin{equation} \label{syl}
\mathcal{J} (t) = Q(t)  \mathcal{J}_0 Q^{T} (t),
\end{equation}
where $\mathcal{J}_0$ is the initial inertia matrix.

We assume that, for any $t\geq0$, the open set  $\mathcal{F}(t) := \R^3 \setminus {\mathcal S} (t) $
is occupied by a fluid driven by the incompressible Euler equations 
so that the fluid velocity vector field $v=(v_1,v_2 ,v_3)$ 
and the fluid pressure scalar field $q$ satisfy in $\cup_{t \geq 0} \, \mathcal{F}(t) $ 
the following PDEs: 
\begin{equation} \label{intrau1}
\displaystyle \frac{\partial v }{\partial t}+(v  \cdot\nabla)v   + \nabla q =0 \quad  
\text{ and } \quad \div v  = 0.
\end{equation}
Above, it is understood that the fluid is supposed to be homogeneous, of density $1$; this does not change the mathematical analysis but simplifies the notations.
We denote by 
${\mathcal F}_{0}:= \R^{3} \setminus {\mathcal S}_{0} $ the initial fluid domain. 
The rigid body  is assumed to be only accelerated, for any $t\geq0$, by the force exerted by the fluid pressure $q$ 
 on its boundary $\partial  \mathcal{S} (t)$,  following  the Newton equations, according to the following ODEs: 
\begin{gather} \label{intrau2}
 m h'' (t) =  \int_{\partial  \mathcal{S} (t)} q n  \dd\sigma 
 \quad \text{ and } \quad 
 (\mathcal{J} R)' (t) =  \int_{\partial  \mathcal{S} (t)} (x-h) \wedge q n  \dd\sigma ,
\end{gather}
where $\dd\sigma$ is the two-dimensional Hausdorff measure. 
Above $n$ denotes the unit normal vector on $\partial \mathcal{S} (t)$ 
pointing outside of the fluid domain $\mathcal{F}(t) $. 
We assume that the boundary of the solid is impermeable so that the 
fluid cannot penetrate into the solid, and we assume that there is no cavitation as well. 
The natural boundary condition at the fluid-solid interface is therefore 
\begin{equation} \label{bc-intro}
v \cdot n = v_{{\mathcal S}}   \cdot n   \quad \text{for}\ \  x\in \partial \mathcal{S}  (t) . 
\end{equation}

\subsection{Reformulation of the Newton equations in the body frame}
\label{Subsec:COF}  
In order to transfer the equations in the body frame, in which the solid is stationary, we apply the following isometric change of variables:
\begin{gather} 
\label{chgedeb1} 
\ell(t) :=  Q(t)^T \,  h' (t),  \\
\label{u_S}
u_{\mathcal{S}} (t,x) := \ell(t) + \Omega(t) \wedge x , \\
\label{chgedeb2}
u(t,x) := Q(t)^T \,  v(t, Q(t) x+h(t)) \ \text{ and } \ \pi(t,x) := q(t, Q(t) x+ h(t)) .
\end{gather}
with 
\begin{equation} \label{defp}
{p} := ( \ell , \Omega ) \in \mathbb{R}^{6}.
\end{equation}

By the change of variables \eqref{chgedeb1}-\eqref{chgedeb2}
the equations \eqref{intrau1}--\eqref{bc-intro} become
\begin{gather}
\label{chE1}
\partial_t u+ (  u - u_\mathcal{S} ) \cdot \nabla u +  \Omega  \wedge u  + \nabla \pi  =  0 
  \ \text{ and } \ \div u = 0   \ \text{ for } \ x \in \mathcal{F}_0 , \\
\label{chE3}
u  \cdot n = u_\mathcal{S}  \cdot n  \ \text{ for } \ x\in \partial \mathcal{S}_0,  \\
\label{chSolide1} 
m  \ell'  =  \int_{ \partial \mathcal{S}_0 } \pi n \dd\sigma  + m \ell \wedge   \Omega ,  \\
\label{chSolide2} 
\mathcal{J}_0 \Omega'  =    \int_{ \partial   \mathcal{S}_0}  \pi ( x \wedge  n) \dd\sigma  
  + ( \mathcal{J}_0   \Omega ) \wedge  \Omega .
\end{gather}
\subsection{Conservation of energy}
\label{subsec-en}
One easily checks that, formally, the total kinetic energy, given by
\begin{equation} \label{NRJ-1}
\mathcal E := \frac12 m  \vert  \ell \vert^2 + \frac12 \mathcal{J}_0 \Omega \cdot \Omega 
+ \frac12 \int_{ \mathcal{F}_0} \vert u \vert^2 \dd x ,
\end{equation}
is conserved. 
In fact, the time derivative of $\mathcal E$ satisfies
\begin{equation*}
\mathcal E' = m  \ell'  \cdot \ell +  \mathcal{J}_0 \Omega' \cdot \Omega 
+ \int_{ \mathcal{F}_0}     u  \cdot  \frac{\partial u }{\partial t} \dd x .
\end{equation*}
Then, using \eqref{u_S}, \eqref{chSolide1} and \eqref{chSolide2} we deduce
\begin{equation*}
m   \ell'  \cdot \ell +  \mathcal{J}_0 \Omega' \cdot \Omega 
= \int_{\partial  \mathcal{S}_0} \pi  u_{{\mathcal S}}   \cdot n  \dd\sigma ,
\end{equation*}
while from \eqref{chE1}, \eqref{chE3} and an integration by parts, we find that
\begin{equation*}
\int_{ \mathcal{F}_0}   
u\cdot\big(-(u-u_\mathcal{S})\cdot\nabla u-\Omega\wedge u-\nabla\pi\big)\dx
= - \int_{\partial  \mathcal{S}_0} \pi u \cdot n  \dd\sigma .
\end{equation*}
Then it follows from \eqref{chE3} that $ \mathcal E' = 0$.

\subsection{Wellposedness} \label{Subsec-Cauchy}
The existence and uniqueness of classical solutions for short times 
is now well-understood thanks to the works \cite{GST,ORT1,ORT2,RR,HSMT,Sueur12}, 
see in particular \cite[Theorem 5]{Sueur12} for a result in the setting considered here
(that is, solutions in Hölder
spaces where the whole fluid-body domain is $\R^3$). 
More precisely, the following general result holds. 
\begin{prop}[\cite{Sueur12}, \cite{GST}] 
\label{loc wellp}
Assume that the initial data $\mathcal{S}(0)$, $h^0, \dot{h}^0$,  
$v^0\in C^{\lambda+1, r}(\mathcal{F}(0))$ with $\lambda \in \N_{>0}$ and $r\in (0,1)$ are given such that $\curl v^0$ has bounded support. 
Then there is a $T>0$ and unique local strong solution on $[0,T)$ to 
the system \eqref{chE1}-\eqref{chSolide2} such that 
\begin{align}
  & u,\pi\in L_{loc}^\infty([0,T),C^{\lambda+1,r}(\mathcal{F}(\cdot)))\\
  &u\in L_{loc}^\infty([0,T), L^2(\mathcal{F}(\cdot)))\\
  & \de_t u \in L_{loc}^\infty([0,T),(C^{\lambda,r}\cap L^2)(\mathcal{F}(\cdot)))\\
  & \ell\in C^2([0,T),\R^3),\quad \Omega\in C^1([0,T),\R^3)
\end{align} 
\end{prop}
Moreover, one may easily adapt the Beale-Kato-Majda \cite{BKM1984} necessary condition on the $L^1_t L^\infty_x$ norm of the vorticity for a blow-up in finite time from the case of a fluid alone to this system. In particular, this yields the global in-time existence of classical solutions in the case of a potential flow. We remark that the formal energy conservation from \ref{subsec-en} above does indeed hold for strong solutions; all the calculations there can easily be made rigorous using the regularity and the decay estimates in the Appendix \ref{app-IPP}. \par
In the special case of a thick, rigid, closed, simple filament considered in this paper
(see Subsection~\ref{Subsec:toro}), we will use a standalone, more precise result for the Cauchy theory,
see Theorem~\ref{start3}.

\subsection{Reformulation of the Newton equations as an ODE}
\label{Subsec:COF2} 

\noindent To reformulate the Newton equations, we are going to use the following Kirchhoff potentials. 
First, we introduce the elementary rigid velocities: 
\begin{equation*} 
\zeta_{i}  (x) :=   \left\{\begin{array}{ll} 
e_i & \text{if} \ i=1,2,3 ,\\ \relax
 e_{i-3} \wedge x& \text{if} \ i=4,5,6 ,
\end{array}\right.
\end{equation*}
and their normal traces, for $i=1,\ldots,6$, 
\begin{equation} \label{t1.6}
K_i (x) := \zeta_{i} \cdot n ,
\end{equation}
where $e_i $, for $i=1,2,3$, denotes the $i$-th unit vector of the canonical basis of $ \R^{3}$.
The Kirchhoff potentials $\nabla \Phi_{i} $, for $i=1,\ldots,6$,  are then defined as  the unique solutions to
\begin{align}
\label{yi0}
\Delta  \Phi_{i}  &= 0 ,   \ \text{ for } \   x \in \mathcal F_0  ,  \\
\label{yi1}
\partial_n \Phi_{i}  &=  K_{i}  ,  \ \text{ for } x \in \partial \mathcal S_0  ,  \\
\label{yi2}
\lim_{|x|\to \infty} \Phi_{i} (x)  &=0   
\end{align}
for $1 \leqslant i \leqslant 6$ (see e.g.\ \cite[Thm.\ 3.1]{Amrouche} for existence and uniqueness). 
The Kirchhoff potentials $\nabla \Phi_{i} $ are smooth and decay as $1 / \vert x \vert^3$ at infinity, 
since they satisfy the following compatibility condition for the Neumann Problem:
\begin{equation} \label{Stokes0}
\int_{\partial  \mathcal{S}_{0}  } K_i \dd\sigma = 0.
\end{equation}
(This easily follows from the properties of the fundamental solution of the Laplace operator in $3$D, see the Appendix \ref{app-IPP} for more details.)
Note in particular that $\nabla \Phi_i$ is in $L^2 (\mathcal{F}_{0})$.

Let $\mathcal M_g $ be the symmetric positive definite matrix 
\begin{equation*}
\mathcal M_g  :=
\begin{pmatrix}
m \, \Id_3  &  0\\
0  &  \mathcal J_0
\end{pmatrix}  ,
\end{equation*}
where we recall that $m$ and $\mathcal{J}_0$ are, respectively, the mass and
the rotational inertia matrix, while $\Id_3 $ is the identity matrix $3 \times 3$. 
Let $\Gamma_{g}: \mathbb{R}^{6} \times \mathbb{R}^{6} \rightarrow \mathbb{R}^{6}$ 
be the bilinear symmetric mapping defined, for all ${p} = ( \ell , \Omega ) \in \R^{6}$, by
\begin{equation} \label{DefGammag}
\langle \Gamma_{g}  , p, p \rangle := -
  \begin{pmatrix}
    m \ell \wedge \Omega   \\
    (\mathcal{J}_0 \Omega) \wedge \Omega 
  \end{pmatrix} .
\end{equation}
Note that
\begin{equation} \label{AnnulationGamma}
\forall {p} \in \mathbb{R}^{6}, \ \ \ \langle \Gamma_{g}, {p}, p \rangle \cdot p =0   .
\end{equation}
Let $\mathcal M_a $ be the following $6 \times 6$ matrix:
\begin{gather} \label{def M}
\mathcal M_a 
:= \Big(  \int_{\mathcal F_0} \nabla  \Phi_{i}  \cdot \nabla \Phi_{j}\dx  \Big)_{1 \leqslant i,j \leqslant 6}
= \Big(  \int_{\partial \mathcal S_0} \Phi_{i} \, \partial_n \Phi_{j}\dd\sigma \Big)_{1 \leqslant i,j \leqslant 6},
\end{gather}
where the second equality is obtained by integration by parts.
The matrix $\mathcal M_{a} $ encodes the added inertia of the rigid body. 
It is symmetric and positive-semidefinite.

We will use the following notation for the triple scalar product, or mixed product, of three vectors
$a,b,c$ in $\R^3$,
\begin{equation*}
\lbrack  a,b,c \rbrack := \det ( a,b,c) = a \cdot (b \wedge c) 
 = (a  \wedge b ) \cdot  c .
\end{equation*}
For a smooth vector field $u$ defined on $\partial \mathcal{S}_{0}$, let ${\mathcal B}[u]$ be the (antisymmetric) $6 \times 6$ matrix whose coefficients are 
\begin{equation} \label{calBdef}
{\mathcal B} [u]_{i,j} := 
 \int_{\partial  \mathcal{S}_{0}  } \, [ \zeta_{j} , \zeta_i , u\wedge n ] \dd\sigma .
 \end{equation}
 Finally, for any $p \in \R^6$ and for two smooth vector fields $u$ and $\omega$ in $\mathcal{F}_{0}$,
 we consider the vector ${\mathcal D} \lbrack p, u, \omega \rbrack$ in $\R^6$ given by: 
\begin{equation} \label{defD}
{\mathcal D} [ p, u, \omega ]
:= \left(  
\int_{\mathcal  F_0}  [ \zeta_{i}, \omega ,u ]  \dx 
- \int_{ \mathcal{F}_0}   [ \omega, u -  u_{\mathcal{S}}, \nabla \Phi_{i} ] \dx
\right)_{1 \leqslant i \leqslant 6}, 
\end{equation}
where then the dependence on $p$ is through the $u_{\mathcal{S}} := \ell  (t)  + \Omega(t) \wedge x$ term.
With these notations, the system can be recast as follows.
\begin{prop} \label{reform-macro}
For a strong solution in the sense of Prop.\ \ref{loc wellp}, the Equations \eqref{Qom}--\eqref{bc-intro}
are equivalent to \eqref{Qom}--\eqref{intrau1}, \eqref{bc-intro}
and
\begin{equation} \label{anoter2}
({\mathcal M}_{g} + {\mathcal M}_{a} ) p' + \langle \Gamma_g, p , p \rangle 
=  {\mathcal B}[ u ] p + {\mathcal D} [ p, u, \omega ]   .
\end{equation}
where $\omega := \curl u$ denotes the fluid vorticity. 
\end{prop}
This extends the earlier results obtained in the 2D case in \cite{GMS,Munnier,Sueur17,Sueur-china} 
and in the 3D irrotational case in \cite{Munnier}.
In the irrotational case, we may observe the following rephrasing of the energy conservation: 
\begin{equation} \label{Conserv-set}
\Big( \frac12 ({\mathcal M}_{g} + {\mathcal M}_{a} ) p \cdot p \Big)' = 0. 
\end{equation}
Furthermore, multiplying \eqref{anoter2} by $p$, and using \eqref{AnnulationGamma} and ${\mathcal B} [ u ] p \cdot p = 0$, 
we obtain:  
\begin{gather*}
\Big(  \frac12 ({\mathcal M}_{g} + {\mathcal M}_{a} ) p \cdot p \Big)'
= \left(  \int_{ \mathcal{F}_0} \bigg( [ u_\mathcal{S}, \omega, u ] 
  - \big[ \omega, u - u_\mathcal{S}, \sum_{1 \leqslant i \leqslant 6} p_i \nabla \Phi_{i}\big] \bigg) \dx 
\right)_{1 \leqslant i \leqslant 6} .
  \end{gather*}

A key instrument in the proof of Proposition~\ref{reform-macro} is the following lemma. 
\begin{lem} \label{lamb}
Let $\tilde{u},\tilde{v}$ be two divergence-free vector fields in 
$C^\infty (\overline{\mathcal{F}}_0;\, \R^3)$   with $\curl \tilde u,\curl \tilde v$ in $L^\infty$.
Let $z$ be a vector field in $H^1(\overline{\mathcal{F}}_0 ;\, \R^3)$. 
Assume that we have that
\begin{equation} \label{suff decay}
\lim_{|x|\rightarrow \infty}  |x|^2|\tilde u(x)||\tilde v(x)||z(x)|=0.
\end{equation}
Then we have the following equality:
\begin{equation}\begin{aligned} \label{vrai}
\int_{\de\mathcal{S}_0}{(\tilde u\cdot \tilde v)(z\cdot n)}
  & - (\tilde u\cdot z)(\tilde v\cdot n) - {(\tilde v\cdot z)(\tilde u\cdot n)} \dd\sigma \\ 
  & = \int_{\mathcal{F}_0} (\tilde u \cdot \tilde v)\div z + (\tilde v\wedge \curl \tilde u + \tilde u\wedge\curl \tilde v)\cdot z \dx \\ 
  & - \int_{\mathcal{F}_0} \tilde u \cdot \left( \left( \tilde v \cdot \nabla\right) z \right)
    - \tilde v \cdot \left( \left(\tilde u \cdot \nabla \right) z \right) \dx.
\end{aligned}\end{equation}
 In particular, in the case where $z=\zeta_i$, the equation \eqref{vrai} reduces to 
\begin{multline} \label{vraiaussi}
\int_{\partial \mathcal S_0} (\tilde u\cdot \tilde v) K_{i} \dd\sigma
= \int_{\partial  \mathcal S_0} \zeta_{i}   \cdot \Big(  (\tilde u \cdot n) \tilde v +  (\tilde v \cdot n)  \tilde u  \Big)\dd\sigma \\ 
+  \int_{\mathcal  F_0} \zeta_{i}   \cdot  \Big( \tilde u \wedge  \operatorname{curl} \tilde v +  \tilde v \wedge  \curl \tilde u  \Big)\dx. 
\end{multline}
\end{lem}
\begin{proof}[Proof of Lemma~\ref{lamb}]
It is sufficient to prove \eqref{vrai} in the case where the three vector fields $\tilde u$, $\tilde v$ and $z$ are smooth, since then the result follows from an approximation process. 
We start with using Stokes' formula to obtain that
\begin{equation} \label{vrai1}
\int_{\de\cals}{(\tilde u\cdot \tilde v)(z\cdot n)} \dd\sigma
    = \int_{\mathcal{F}_0}{\div((\tilde u\cdot \tilde v)z)}\dx 
    = \int_{\mathcal{F}_0} \left( {\nabla(\tilde u\cdot \tilde v)\cdot z} + {(\tilde u\cdot \tilde v)\div z} \right)\dx,
\end{equation}
where the partial integration is justified by the decay condition \eqref{suff decay}.
Applying the identity \begin{equation} \label{etlesigne}
  \nabla (\tilde u\cdot \tilde v)  =  \tilde u \cdot \nabla   \tilde v + \tilde v\cdot \nabla \tilde u + \tilde u \wedge  \operatorname{curl} \tilde v +   \tilde v \wedge  \operatorname{curl} \tilde u  ,
  \end{equation}
to $\tilde u$ and $\tilde v$ we obtain
\begin{equation}
\begin{aligned} \label{vrai2}
\int_{\de\cals}{( \tilde u\cdot \tilde v)(z\cdot n)} \dd \sigma=
  & \int_{\mathcal{F}_0} (\tilde u\cdot \tilde v)\div  \,z +  (\tilde v\wedge \curl \tilde u + \tilde u\wedge\curl \tilde v)\cdot z \dd x\\  
  & + \int_{\mathcal{F}_0}  {\bigl((\tilde u\cdot \nabla)\tilde v+(\tilde v\cdot \nabla)\tilde u\bigr) \cdot z} \dx.
\end{aligned}\end{equation}
Finally, integrating by parts again, since $\tilde u$ and $\tilde v$ are divergence-free
and still relying on \eqref{suff decay}, we obtain
\begin{equation}\begin{aligned} \label{vrai3}
\int_{\mathcal{F}_0} ( (\tilde u\cdot \nabla) \tilde v ) \cdot z\dx
   & = -\int_{\mathcal{F}_0}{\tilde v \cdot ((\tilde u\cdot \nabla)z)}\dx
     + \int_{\de\cals} (\tilde v \cdot z) (\tilde u \cdot n)\dd\sigma, \quad \text{and} \\
\int_{\mathcal{F}_0} ( (\tilde v \cdot \nabla) \tilde u ) \cdot z\dx
   & = -\int_{\mathcal{F}_0} \tilde u \cdot ((\tilde v \cdot \nabla) z )\dx
     + \int_{\de\cals} (\tilde u \cdot z) (\tilde v \cdot n)\dd\sigma .
\end{aligned}\end{equation}
Gathering \eqref{vrai2} and \eqref{vrai3}, we obtain \eqref{vrai}.

In the special case \eqref{vraiaussi}, we first observe that $\div z=0$ for all $i=1,\ldots,6$.
Moreover, in the case $i=1,2,3$ we have $\nabla z=0$, while in the case $i=4,5,6$ we have 
\begin{equation*}
- \tilde v \cdot ( \tilde u \cdot \nabla_x \zeta_{i} )-  \tilde u \cdot (\tilde v\cdot \nabla_x  \zeta_{i})
 = - \big(\tilde v \cdot  (e_{i-3} \wedge \tilde u ) +  \tilde u \cdot (e_{i-3} \wedge \tilde v ) \big)  = 0, 
\end{equation*}
by skew-symmetry. 
This allows us to conclude.
\end{proof}

\begin{proof}[Proof of Proposition~\ref{reform-macro}]
The proof involves various partial integrations; for the sake of keeping this section as non-technical as possible, we will not concern ourselves with the question of whether the involved functions decay fast enough at $\infty$ here. However, everything here can be properly justified, see Lemma \ref{decay prop} below.

Therefore, by integration by parts,
we obtain that the force and torque acting on the body 
(which occur in the right-hand sides of \eqref{chSolide1}  and \eqref{chSolide2}) satisfy
\begin{equation} \label{prelooser}
\left( \int_{ \partial \mathcal{S}_0}\pi n \dd\sigma, \
\int_{ \partial \mathcal{S}_0}\pi   ( x \wedge  n)  \dd\sigma \right) = 
\left( \int_{ \mathcal{F}_0} \nabla\pi \cdot \nabla \Phi_{i} \dd x \right)_{1 \leqslant i \leqslant 6} .
\end{equation}
Now we observe that thanks to \eqref{etlesigne} 
the equation \eqref{chE1} can be written as
\begin{equation} \label{looser}
\frac{\partial u}{\partial t}
+  \nabla \left( \frac{1}{2} u^2 + \pi \right) + \omega \wedge (u -  u_{\mathcal{S}} )
- \nabla ( u \cdot u_{\mathcal{S}} ) = 0,
\end{equation}
where we recall that $\omega := \curl u$ denotes the fluid vorticity. 
By \eqref{prelooser} and \eqref{looser} the Newton equations \eqref{chSolide1} and \eqref{chSolide2} become
\begin{multline} \label{postlooser}
\mathcal M_g p '  + \langle \Gamma_{g}  , p, p \rangle 
= - \left( \int_{ \mathcal{F}_0}  \partial_{t}  u  \cdot \nabla \Phi_{i} \dd x \right)_{1 \leqslant i \leqslant 6} 
- \frac{1}{2} \left( \int_{ \mathcal{F}_0}
  \nabla ( u^2)  \cdot \nabla \Phi_{i} \dd x \right)_{1 \leqslant i \leqslant 6} \\
+ \left( \int_{ \mathcal{F}_0}
  \nabla (  u \cdot u_{\mathcal{S}} ) \cdot \nabla \Phi_{i} \dd x \right)_{1 \leqslant i \leqslant 6} 
- \left( \int_{ \mathcal{F}_0}
  [ \omega , u -  u_{\mathcal{S}} , \nabla \Phi_{i} ] \dd x \right)_{1 \leqslant i \leqslant 6} .
\end{multline}
Using an integration by parts, the boundary condition \eqref{chE3} and \eqref{def M}, one observes that the first term on the right-hand side of \eqref{postlooser} is
\begin{equation} \label{vieux}
\left( \int_{  \mathcal{F}_{0}  }  \partial_{t}  u  \cdot  \nabla   \Phi_i (x) \dd x  \right)_{1 \leqslant i \leqslant 6}
= {\mathcal M}_{a} \, p' .
\end{equation}
By integration by parts of the second and third terms on the right-hand side of \eqref{postlooser}, we arrive at 
\begin{multline} \label{anoter}
( {\mathcal M}_{g} + {\mathcal M}_{a} ) p ' + \langle  \Gamma_g ,  p ,  p \rangle =  \\
\left(
- \frac{1}{2} \int_{\partial  \mathcal{S}_{0}  } |u|^2  K_i \dd\sigma
+ \int_{\partial  \mathcal{S}_{0}  } (u \cdot u_{\mathcal{S}} ) K_i \dd\sigma
- \int_{ \mathcal{F}_0}  \big(\omega \wedge (u -  u_{\mathcal{S}} ) \big) \cdot \nabla \Phi_{i} \dd x 
\right)_{1 \leqslant i \leqslant 6} .
\end{multline}
\noindent To deal with the first term in the right-hand side of \eqref{anoter}, we apply Lemma~\ref{lamb} with $u=\tilde{v}=\tilde u$ and obtain for $1 \leqslant i \leqslant 6$, 
\begin{eqnarray}
\nonumber
\frac{1}{2} \int_{\partial  \mathcal{S}_{0}  } |u|^2  K_i \dd\sigma 
&=&
\int_{\partial  \mathcal{S}_{0} }  (u \cdot n) (u \cdot \zeta_{i} )  \dd\sigma
-  \int_{\mathcal  F_0} \zeta_{i}   \cdot  ( \omega \wedge  u ) \dd x \\
\label{24}
&=& \int_{\partial  \mathcal{S}_{0}  } \big( u_{\mathcal{S}} \cdot  n \big)  \big( u \cdot \zeta_{i}  \big)  \dd\sigma 
-  \int_{\mathcal  F_0} \zeta_{i}   \cdot ( \omega \wedge  u )  \dd x  ,
\end{eqnarray}
thanks to the boundary conditions \eqref{chE3}.

Therefore, the right-hand side of \eqref{anoter} can be rephrased as 
\begin{equation} \label{anoterLU}
\left( \int_{\partial  \mathcal{S}_{0} } 
\Big(  (u \cdot u_{\mathcal{S}} )  n  - \big( u_{\mathcal{S}} \cdot  n \big) u  \Big) 
  \cdot \zeta_{i} \dd\sigma \right)_{1 \leqslant i \leqslant 6}
 + {\mathcal D} [ p, u, \omega ]   , 
\end{equation}
where ${\mathcal D} [ p, u, \omega ]$ is given by  \eqref{defD}. 
Recalling \eqref{calBdef} and using the Cauchy-Binet identity \begin{align}(a\wedge b)\cdot(c\wedge d)=(a\cdot c)(b\cdot d)-(a\cdot d)(b\cdot c)\label{CauchyBinet}\end{align} as well as $\sum_j p_j\zeta_j=u_\mathcal{S}$, we observe that the first term above is $ {\mathcal B} \lbrack  u \rbrack p$, 
which concludes the proof of Proposition \ref{reform-macro}.
\end{proof}

In an Appendix, see Section \ref{app-IPP}, we establish some decay estimates to justify the integrations by parts done above in the proof of Proposition \ref{reform-macro}.
%
%
%

\subsection{The case of a thick rigid closed simple filament}
\label{Subsec:toro}

In this paper, we will consider the case where the rigid body is topologically a solid torus.
Let us be more specific on this geometric setting and state its first consequences.

We start from a loop ${\mathcal C}_0$ in $\R^3$, that is a closed, oriented, non-self-intersecting, smooth curve, with normalized tangent $\tau$. We further let $\gamma:\R /L\Z\rightarrow \mathcal{C}_0$ be an arc-length parametrization so that \begin{align}
 \mathcal{C}_0=\left\{\gamma(t)\,\big|\, t\in \R/L\Z\right\}.\label{def gamma}   
\end{align}
We will assume that the domain occupied by the rigid body takes the following form: For $\varepsilon \in (0,1]$, we let
\begin{equation} \label{def-solide-scaled}
\cals = \cals^\eps := \{x\,\big|\,\dist(x,\mathcal{C}_0)\leq \eps\},
\end{equation}
 and let $\mathcal{F}_0 = \mathcal{F}_0^\eps = \R^3 \backslash \cals^\eps$ denote the fluid domain. See also Figure \ref{fig:solid} in the introduction for a sketch.

We introduce further notations to parameterize $\mathcal{S}_0$.
We first let $({s_1}(t),{s_2}(t),\tau(t))$ be some smooth orthonormal frame associated 
with some point $\gamma(t)$ on the curve.
In a neighborhood of size $\delta>0$ of the curve, we can write every point $x$ uniquely as
\begin{equation*}
x=x_{s_1}{s_1}(t(x)) + x_{s_2}{s_2}(t(x)) + \gamma(t(x)),
\end{equation*}
as e.g.\ the implicit function theorem shows, which in particular determines an orthogonal projection $\gamma(t(\cdot))$. 
By an abuse of notation, we will sometimes also write these unit vectors as functions of $x$, e.g.\ as in 
$\tau(x)=\tau(t(x))$.

We extend the normal $n$ (naturally defined on $\partial \mathcal{S}_0^{\varepsilon}$) as 
\begin{equation} \label{ext normal}
n = - \frac{1}{x_{s_1}^2 + x_{s_2}^2} \left(x_{s_1}{s_1}(x) + x_{s_2}{s_2}(x) \right),
\end{equation}
in this $\delta$-neighborhood of $\mathcal{C}_0$, deprived of $\mathcal{C}_0$ itself.

We will set
\begin{equation*}
e_\theta:=n\wedge \tau,
\end{equation*}
which, together with $\tau$, is a basis of the tangent space of $\de\cals$. We will also write $\de_\tau,\de_n,\de_\theta$ for the derivatives in these directions.

Let
\begin{equation} \label{Def:CalO}
\mathcal{O} = \mathcal{C}_0 + B_{\delta}(0)
\end{equation}
denote a domain where all these quantities are well-defined. 
Notice that for suitably small $\varepsilon$, all solid domains $\mathcal{S}^\varepsilon_0$ 
are included in $\mathcal{O}$. \par
Now for $t \in \R/L\Z$, we consider the following slice of $\partial \mathcal{S}_0^\varepsilon$
\begin{equation} \label{Def:CTeps}
C_t^\eps := \de \mathcal{S}_0 \cap \Bigl( \gamma(t) + \R e_{s_1}(t) + \R e_{s_2}(t) \Bigr) \cap \mathcal{O}.
\end{equation}
For small enough $\eps$, $C_t^\eps$ is a circle of radius $\eps$ with center $\gamma(t)$, 
and $e_\theta$ and $n$ are its tangent and normal unit vector fields, respectively.
As a last geometric definition, we  let 
\begin{equation} \label{Eq:DefTildeC}
\tilde{C}=C_0^{\tilde{\eps}},
\end{equation}
for such a fixed $\tilde{\varepsilon}\geq \eps$; its unit tangent vector is $e_\theta$ 
and we will define the circulation with respect to this oriented loop. \par
\ \par

In this geometry, the $\div/\curl$ system in $\mathcal{F}^\varepsilon_0$ takes a specific form, 
which we now describe. 
This will allow for the decomposition of the velocity field into elementary fields. \par
\ \par
\noindent
{\it Harmonic field.} To take the velocity circulation around the body into account, we first introduce the following harmonic field:
let $H=H^\eps$ be the unique solution (cf.\ \cite{Hieber} for well-posedness), vanishing at infinity, of 
\begin{gather}
\operatorname{div} H = 0    \text{ and }   \operatorname{curl} H = 0 \,    \text{ in }   \mathcal{F}_{0}, \, 
H \cdot n = 0 \,    \text{ on }  \partial \mathcal{S}_0, \, 
\int_{\tilde{C} } H \cdot e_\theta \dd s = 1 .\label{def H}
\end{gather}
It follows from Stokes' theorem that this definition does not depend on the choice of $\tilde{C}$.

\noindent
{\it Biot-Savart field.}
To account for the vorticity $\omega$, we introduce the Biot-Savart law $K_{\mathcal{F}_0}[\omega]$, defined through 
\begin{align*}
&\curl \Psi= K_{\mathcal{F}_0}[\omega] , \\
&-\Delta \Psi=\omega \text{ in $\mathcal{F}_0$} ,\\
&\curl\Psi\cdot n=0 \text{ on $\de \mathcal{F}_0$} ,\\
&\int_{\tilde{C} } \curl \Psi \cdot e_\theta \dd s = 0,\\
&\lim_{x\rightarrow \infty}\Psi=0.
\end{align*}

\noindent For future reference, we also introduce the Biot-Savart law in $\R^3$, defined as  $K_{\R^3}[\omega] =\linebreak -\curl \Delta^{-1}\omega$, or equivalently
\begin{equation} \label{KR2}
K_{\mathbb{R}^{3}}[\omega](x) := \frac{-1}{4\pi} \int_{\mathbb{R}^{3}} \frac{x-y}{|x-y|^{3}} \wedge \omega(y) \dd y.
\end{equation}
Now we have the following standard decomposition of the velocity field, in the considered geometry.
\begin{lem} \label{vitdec}
For $p=(\ell,\Omega)\in \mathbb{R}^{6}$ and $\mu \in\mathbb{R}$ given, and for any smooth, divergence-free $\omega$, compactly supported in $\overline{\mathcal{F}}_0$, there exists a unique vector field $u$ verifying  the following div/curl type system:
\begin{gather}
\label{r1}
 \operatorname{div} u = 0     \text{ and } 
 \operatorname{curl} u  =  \omega  \text{ in }   \mathcal{F}_{0}  , \, 
 \\ \label{r2}
 u \cdot n = \left(\ell+ \Omega \wedge x \right)\cdot n     \text{ on }  \partial \mathcal{S}_0  , \,    \quad   \int_{ \tilde{C}} u  \cdot  e_\theta \dd s=  \mu ,
 \\ \label{r3}
\lim_{x\to \infty} u = 0 .
 \end{gather}  
and it is given by the  law:
\begin{equation} \label{vdecomp}
u =  \mu   H + \sum_{1 \leqslant i \leqslant 6}   p_i \nabla \Phi_i + K_{ \mathcal F_0} [\omega] .
\end{equation}
\end{lem}
The advantage of working with the fluid vorticity $\omega := \curl u$ is that it satisfies 
the following transport-with-stretching equation: 
\begin{equation} \label{omega-eq}
\partial_t \omega + \big((u-  u_{\mathcal{S}}) \cdot\nabla\big)  \omega   +\Omega \wedge \omega  
  =  \big(\omega  \cdot\nabla \big) u
\quad  \text{ in } \quad  \mathcal{F}_0 ,
\end{equation}
which allows us to reformulate the system again as follows.
Recalling the definition of $\mathcal{B}$ in \eqref{calBdef}, we let
\begin{equation}
B := {\mathcal B} [ H ] \in \R^{6 \times 6},\label{def B}
\end{equation}
and we introduce $\Gamma_{a}: \mathbb{R}^{6} \times \mathbb{R}^{6} \rightarrow \mathbb{R}^{6}$ as 
the bilinear symmetric mapping given, for all ${p} = ( \ell , \Omega ) \in \mathbb{R}^{6}$, by
\begin{equation} \label{DefGammaa}
\langle \Gamma_{a}, p, p \rangle := 
  -\sum_{1 \leqslant i \leqslant 6}   p_i {\mathcal B} [\nabla \Phi_i] p .
\end{equation}
It follows from the skew-symmetry of the matrix ${\mathcal B} [ \nabla \Phi_i ]$  that 
\begin{equation} \label{AnnulationGamma2}
\forall p \in \mathbb{R}^{6}, \ \ \ \langle \Gamma_{a}, {p}, p \rangle \cdot p =0.
\end{equation}
In the original frame, such terms can be interpreted as Christoffel symbols associated with the added inertia, see, for instance, \cite{Munnier}. \par
\begin{prop} \label{reup}
In the case where $\mathcal S_0$ is  a thick rigid closed simple filament (i.e.\ \eqref{def-solide-scaled} holds) and the solution is strong in the sense of Prop.\ \ref{loc wellp}, 
the equations \eqref{Qom}--\eqref{bc-intro}
are equivalent to \eqref{Qom}--\eqref{syl}, \eqref{bc-intro}, \eqref{vdecomp}, \eqref{omega-eq} and 
\begin{equation} \label{eq}
\big(  {\mathcal M}_{g} + {\mathcal M}_{a} \big) p' 
  + \langle  \Gamma_g ,  p ,  p \rangle + \langle  \Gamma_a ,  p ,  p \rangle
=  \mu B p + {\mathcal B}  \left[ K_{ \mathcal F_0} [\omega] \right] p
  + {\mathcal D} \lbrack  p, u, \omega \rbrack .
\end{equation}
\end{prop}

\begin{proof}
According to Proposition \ref{reform-macro}, Equations \eqref{Qom}--\eqref{bc-intro}
are equivalent to \eqref{Qom}--\eqref{intrau1}, \eqref{bc-intro} and \eqref{anoter2}. 
Combining with Lemma~\ref{vitdec}, by substitution of the decomposition \eqref{vdecomp} into \eqref{anoter},
we prove the statement, since the rotation $\mu$ is conserved under the evolution by the Kelvin circulation theorem.
\end{proof}

Observe that in the irrotational case $\omega=0$, the equation \eqref{eq} simplifies to 
\begin{equation} \label{eq-simpli}
\big(  {\mathcal M}_{g} +  {\mathcal M}_{a} \big) p' 
 + \langle  \Gamma_g ,  p ,  p \rangle
 + \langle  \Gamma_a ,  p ,  p \rangle
=  \mu B p ,
\end{equation}
which is an autonomous equation, while the fluid can be recovered from $p$ 
through the formula \eqref{vdecomp}, which simplifies into 
\begin{equation}
\label{vdecomp-simpli}
u =  \mu   H + \sum_{1 \leqslant i \leqslant 6}   p_i \nabla \Phi_i .
\end{equation}
In that case, we have the following statement.
\begin{prop} 
\label{Macro-glob}
For any initial datum $p_0 \in \R^6$, 
there is a unique smooth global solution to \eqref{eq-simpli}.
\end{prop}
This follows from the Cauchy-Lipschitz theorem together with the energy conservation \eqref{Conserv-set}, since $\mathcal{M}_a+\mathcal{M}_g$  is positive definite.


\section{The zero-radius rigid filament limit}
\label{sec-z}
In this section, we describe the main problem studied in this paper and formulate our main statements. The geometric situation is the one described in
Subsection~\ref{Subsec:toro}.
%

%
%

\subsection{General framework}
\label{Subsec:GeneralFr}
The main question that we address in this paper is the asymptotic behavior of the solution of the system  
associated with a given fixed initial vorticity and a given fixed circulation around $\tilde{C}$ (which was introduced in \eqref{Eq:DefTildeC}), 
when the solid domain occupies $\mathcal{S}^\varepsilon_0$ and $\varepsilon \rightarrow 0^+$. 
The circulation $\mu$ around the body will be consequently considered fixed and independent of $\varepsilon$.
We will consider first the simpler case of irrotational flows in 
Section~\ref{Subsec:without}, and then the more intricate case of a fluid with vorticity in Section~\ref{Subsec:with}. \par
\ \par
Concerning the inertia coefficients, we will first consider the case of a {\it massive filament}.
We define a massive filament as the limit of a rigid body of the above form when its radius $\varepsilon$ goes to $0$
while its mass $m^\varepsilon$ and its initial momentum of inertia $ \mathcal J^\varepsilon_0$  
both satisfy $m^\varepsilon=  m $ and  $\mathcal{J}^\varepsilon_0 =   \mathcal{J}_0$. 
Here $m>0$  and  $\mathcal{J}_{0}$ is a $3 \times 3$ symmetric positive definite matrix, 
and both are fixed, independently of $\varepsilon$. \par
\ \par
Finally, to state our main results, we will need the following {\it area and volume vectors}, which merely depend on the
geometry of $\mathcal{C}_0$.
We first associate with ${\mathcal C}_0$ the unique vector-valued measure $\kappa_{\mathcal{C}_0}$ 
obtained from the restriction of the one-dimensional Hausdorff measure on ${\mathcal C}_0$: 
\begin{equation*}
\kappa_{\mathcal{C}_0} = \tau \, \mathcal{H}^1 \, \mres {\mathcal C_0} .
\end{equation*}
\begin{defn} \label{svmatrix}
We define the area vector and volume vector associated with $\mathcal{C}_0$ as: 
\begin{equation*}
\mathcal A_0 := \frac12 \int_{\mathcal C_0} x \wedge \dd\kappa_{\mathcal{C}_0} (x) \in \R^3 
\quad \text{ and } \quad 
\mathcal V_0 := -\frac{1}{2}\int_{\mathcal C_0} \vert x \vert^2 \dd\kappa_{\mathcal{C}_0} (x)   \in \R^3 .
\end{equation*}
\end{defn}
We recall that the vector area $ \mathcal A_0 $ is useful when determining the flux of a constant vector
field through a surface, which is given by the dot product of the vector field and of the vector area of 
the surface. 

Based on $\mathcal{A}_0$ and $\mathcal{V}_0$, we define the $6 \times 6$ skew-symmetric matrix $B^*$ as follows: 
\begin{equation} \label{defBstar}
B^* :=  
\begin{pmatrix}
   0  &   \mathcal A_0 \wedge \cdot\\ 
   -\mathcal A_0 \wedge \cdot  &  \mathcal V_0 \wedge \cdot
\end{pmatrix} ,
\end{equation}
where we write ``$M\wedge\cdot$'' to denote the $3\times3$-matrix associated with the linear map $x\rightarrow M\wedge x$. 
Using the parametrization $\gamma$ of $\mathcal{C}_0$ as in \eqref{def gamma}, it holds that
\begin{equation*}
\mathcal A_0 = \frac12 \int_0^L \gamma \wedge \gamma'\dd s \in \R^3 
\quad \text{ and } \quad
\mathcal V_0 =\frac{1}{2} \int_{0}^L \vert \gamma \vert^2   \gamma' \dd s\in \R^3 .
\end{equation*}

%
%

\subsection{Case without vorticity}
\label{Subsec:without}

We start by describing our results in the case of an irrotational fluid for which the equations 
at stake are \eqref{eq-simpli} and  \eqref{vdecomp-simpli}. 
To indicate the dependence on the scaling parameter $\varepsilon > 0$,
the equation \eqref{eq-simpli} now reads as
\begin{equation} \label{eqeps}
\big(  {\mathcal M}_{g} +  {\mathcal M}_{a}^\varepsilon \big) (p^\varepsilon)' 
 + \langle \Gamma_g, p^\varepsilon, p^\varepsilon \rangle
 + \langle \Gamma_a^\varepsilon, p^\varepsilon, p^\varepsilon \rangle
=  \mu B^\varepsilon p^\varepsilon .
\end{equation}
On the other hand, adapting \eqref{vdecomp-simpli} to the scaling, 
the corresponding  fluid velocity $u^\varepsilon$ is  given by the formula: 
\begin{equation} \label{vdecomp-simpli-eps}
u^\varepsilon 
= \mu  H^\varepsilon + \sum_{1 \leqslant i \leqslant 6}  p_i^\varepsilon \nabla \Phi_i^\varepsilon. 
\end{equation}
The main outcome is to identify the limit dynamic, which is given by the following first-order ODE:
\begin{equation} \label{eql1}
{\mathcal M}_{g}  (p^*)' + \langle  \Gamma_g ,  p^* ,  p^* \rangle
 = \mu B^* p^* .
\end{equation}
By \eqref{AnnulationGamma} 
and the skew-symmetry of the matrix $B^*$, we observe that the energy 
\begin{equation*}
 \frac12 {\mathcal M}_{g}  p^* \cdot p^* ,
\end{equation*}
is conserved in time for the solutions to \eqref{eql1}.
Then, by the Cauchy-Lipschitz theorem, for any initial data $p_0 \in \R^6$, there is a unique global smooth solution of 
\eqref{eql1}.

Another outcome of our investigations is the limit behaviour of the fluid. Observe that the zero radius limit then corresponds to a singular perturbation problem in space for the fluid velocity in the case of a non-zero circulation $\mu$ around the filament.  
Let us also introduce
\begin{equation} \label{def-H}
  H^*  := K_{\mathbb{R}^{3}}[ \kappa_{\mathcal{C}_0} ],
\end{equation}
where $K_{\R^3}$ is the full-space Biot-Savart law as defined in \eqref{KR2}. \par
To describe our convergence results, we will rely on the following definition.

\begin{defn} \label{def conv}
Let $f^\eps:\R^3\backslash \cals^\eps\rightarrow \R^3$ be a sequence of functions 
and $f^*:\R^3\rightarrow \R^3$.
\begin{itemize}
\item We say that $f^\eps\rightarrow f^*$ in $L^2(\mathcal{F}_0^\eps)$ if $\norm{f^\eps-f^*}_{L^2(\mathcal{F}_0^\eps)}\rightarrow 0$ (note that this does not require that $f^*\in L^2(\R^3)$).
\item We say that $f^\eps\rightarrow f^*$ in $C_{loc}^\infty$ if $f^\eps\rightarrow f^*$ in $C^m(U)$ for every $m>0$ and for every compact subset $U$ of $\R^3\backslash \mathcal{C}_0$.
\end{itemize}
\end{defn}
The second definition makes sense because for each such $U$ and suitably small $\varepsilon>0$,
$U$ and $\cals^\eps$ do not intersect. \par
\ \par
Our main result in the irrotational case is the following.
\begin{thm} \label{pasdenom-VO}
Let ${\mathcal C}_0$ be a loop in $\R^3$ as above, and associate 
$B^*$ to it by \eqref{defBstar}.
Let $p_0 := ( \ell_0 , \Omega_0 ) \in \R^6$ be some initial solid translation and rotation velocities 
and let a circulation $\mu\in \R$ be given. 
Let $m>0$  and  $\mathcal{J}_{0}$ be a  $3 \times 3$ symmetric positive definite matrix. 
Let $p^*$ be the unique global smooth solution of \eqref{eql1} associated with $p_0$. 
For each $\varepsilon > 0$, let $p^\varepsilon$ be the unique smooth global solution of \eqref{eqeps} 
with initial data $p (0) = p_0$, as given by Proposition \ref{Macro-glob}. 

Then, for all $T>0$,  for every $k\in \N_{\geq 0}$,
\begin{align*}
p^\varepsilon\rightarrow p^*  & \text{ in $C^k([0,T];\mathbb R^6)$ as $\varepsilon \to 0^+$}.
\end{align*}
Moreover, as $\varepsilon \to 0^+$, the corresponding fluid velocity $u^\eps$ 
given by \eqref{vdecomp-simpli-eps} converges to $\mu H^*$ in $L^2(\mathcal{F}_0^\eps)$ 
and $C_{loc}^\infty$, in the sense of Definition~\ref{def conv}.
\end{thm}

The proof of Theorem \ref{pasdenom-VO} is given in Section  \ref{sec-cv}
 after some preliminary technical work in  Section  \ref{sec-Asymptotic} and \ref{Sec:compu}.

\begin{rem}
We point out that it is possible to strengthen the result of Theorem \ref{pasdenom-VO} by quantifying the convergence of $p^\varepsilon $ to $ p^*$. Actually, relying on the asymptotics of the coefficients (see Section \ref{Sec:compu}), 
it is possible to establish that the convergence holds in $C^k([0,T];\mathbb R^6)$ at the rate $O(\eps |\log\eps|^\frac{3}{2})$.
Instead, we rather use a compactness method which we will pursue in the general case with vorticity. 
\end{rem}

\begin{rem}
It would be interesting to investigate an extension of the results above to several rigid bodies,  see for instance in different contexts for several vortex filaments \cite{GS,GLMS,Meyer}.  
\end{rem}

%
%

\subsection{Case with vorticity}
\label{Subsec:with}

We now deal with the more difficult case where the fluid vorticity does not vanish.
The equations at stake are now \eqref{u_S}, \eqref{vdecomp}, \eqref{omega-eq} and   \eqref{eq}. 
To indicate the dependence on the scaling parameter $\varepsilon $, we write them as 
\begin{equation} \label{eqeps-WV}
\big(  {\mathcal M}_{g} +  {\mathcal M}_{a}^\varepsilon \big) (p^\varepsilon)' 
 + \langle  \Gamma_g ,  p^\varepsilon ,  p^\varepsilon \rangle
 + \langle  \Gamma_a^\varepsilon ,  p^\varepsilon ,  p^\varepsilon \rangle
=  \mu B^\varepsilon p^\varepsilon 
 + {\mathcal B} \big[ K_{ \mathcal F_0} [\omega^\varepsilon] \big] p^\varepsilon
 +  {\mathcal D} [ p^\varepsilon, u^\varepsilon, \omega^\varepsilon ],
\end{equation}
\begin{equation} \label{omega-eqeps}
\partial_t \omega^\varepsilon 
  + \big((u^\varepsilon -  u^\varepsilon_{\mathcal{S}}) \cdot\nabla\big)  \omega^\varepsilon
= \big(\omega^\varepsilon  \cdot\nabla \big) (u^\varepsilon -  u^\varepsilon_{\mathcal{S}}), 
  \quad  \text{ in } \quad  \mathcal{F}_0^\varepsilon ,
\end{equation}
\begin{equation} \label{vdecompeps}
u^\varepsilon =  \mu H^\varepsilon + \sum_{1 \leqslant i \leqslant 6} p^\varepsilon_i \nabla \Phi_i^\varepsilon 
+ K_{ \mathcal F_0^\varepsilon} [\omega^\varepsilon], \quad  \text{ in } \quad  \mathcal{F}_0^\varepsilon .
\end{equation}

\noindent In the Appendix (Section~\ref{Sec:WP-limit}),  
 we prove that for $\varepsilon > 0$ small enough,  
for any initial datum $p(0) = p_0 $ in $\R^6$, for any initial vorticity in the H\"older space 
$C^{\lambda,r}(\R^3 ;\, \R^3)$ where $r \in (0,1)$ and $\lambda\in \N_{\geq 0}$, supported away from ${\mathcal S}_0^\eps$, 
there is a $T^\varepsilon >0$, uniformly bounded from below, and a unique strong solution $(p^\varepsilon , u^\varepsilon, \omega^\varepsilon )$
of \eqref{eqeps-WV}--\eqref{vdecompeps}.

In such a case, the limit dynamics of the system are more coupled, in the sense that the limit dynamics of the filament are influenced by the fluid vorticity. 
Before stating our convergence result, let us first describe this limiting dynamic. 
This requires a few more notations/definitions. 

For two smooth vector fields $u$ and $\omega$ in $ \R^3  $, we introduce
the vector ${\mathcal D}^* [u, \omega]$ in $\R^6$ as: 
\begin{equation} \label{defD*}
{\mathcal D}^*  [u, \omega] :=
\left( \int_{\R^3}  [ \zeta_{i}  ,\omega , u ]  \ dx \right)_{1 \leqslant i \leqslant 6} .
\end{equation}
Using again the index $*$ to allude to quantities associated with the limit dynamics,
the equation for the limit filament is then given by the following ODE:
\begin{equation} \label{eql1-WV}
{\mathcal M}_{g}  (p^*)'  +  \langle  \Gamma_g ,  p^* ,  p^* \rangle
=  \mu B^* p^*  +  {\mathcal D}^* [ u^*, \omega^* ],
\end{equation}
where again $B^*$ is given by \eqref{defBstar}.
Observe that, compared to \eqref{eql1}, the equation \eqref{eql1-WV} contains an extra term
which involves, in particular, the fluid vorticity. 
Similarly to \eqref{omega-eq}, this limit vorticity $\omega^*$ satisfies the equation: 
\begin{equation} \label{omega-eq*}
\partial_t \omega^*  + \big((u^* -  u^*_{\mathcal{S}}) \cdot\nabla\big)  \omega^*
=  \big(\omega^* \cdot \nabla \big) (u^* - u^*_{\mathcal{S}})   \quad  \text{ in } \quad  \R^3 , 
\end{equation}
where 
\begin{equation} \label{u*}
 u^*_{\mathcal{S}} (t,x) := \ell^*  (t)  + \Omega^* (t) \wedge x 
 \quad  \text{ where } \quad p^* := ( \ell^* , \Omega^* ) \in  \R^3 \times  \R^3 .
 \end{equation}
Recalling \eqref{def-H}, the limit velocity $u^*$ satisfies 
\begin{equation} \label{vdecomp*}
u^* =  \mu   H^*  + K_{ \R^3} [\omega^*]  ,
\end{equation}
so that, in particular, 
\begin{equation*}
\omega^*  = \curl u^*- \mu . \kappa_{\mathcal{C}_0} .
\end{equation*}
Note in passing that the last term $\big( \omega^* \cdot \nabla \big) u^*_{\mathcal{S}}$ 
in \eqref{omega-eq*} simplifies to
\begin{equation} \label{Eq:NablaUS}
\big( \omega^* \cdot \nabla \big) u^*_{\mathcal{S}} = \Omega^* \wedge \omega^*.
\end{equation}
We emphasize that despite the singularity of $u^*$ along $\mathcal{C}_0$, all quantities 
$(u^* \cdot\nabla) \omega^*$, $(\omega^* \cdot \nabla) u^*$ and $\mathcal{D}^*[u^*, \omega^*]$ used above
are correctly defined as long as $\mathcal{C}_0$ and $\Supp (\omega^*)$ are separated.

We have the following result concerning the Cauchy problem for this system.
\begin{prop} \label{limit-WV}
Let $\mathcal{C}_0$ be a loop in $\R^3$. Let $\lambda\in \N_{\geq 0}$ and $r \in (0,1)$. 
Consider $p_0\in\R^6$ and $\omega_0\in C^{\lambda,r}(\R^3;\R^3)$, divergence-free 
and with compact support in $\R^3 \setminus \mathcal{C}_0$. 

Then there exists $T^*>0$ and a unique maximal solution $(p^*, \omega^*)$ 
of \eqref{eql1-WV}--\eqref{vdecomp*} with $p^*\in W^{\lambda,\infty}([0,T^*);\,\R^6)$ and \begin{equation*}\omega^*\in C([0,T^*);\, C^{\lambda,r}(\R^3;\,\R^3)-w^*) \cap C^l([0,T^*);\, C^{\lambda-l,\alpha}(\R^3;\,\R^3)-w^*) \quad \text{ for all $l\in \N_{\leq \lambda}$,}\end{equation*}
such that for any $t \in [0,T^*)$ we have $\dist(\supp \omega^*(t), \mathcal{C}_0) >0$. 
\end{prop}
Here, $"w^*-"$ refers to the weak$\ast$-topology. Proposition~\ref{limit-WV} is a part of Theorem~\ref{start3} and is proved in the Appendix
 (Section~\ref{Sec:WP-limit}). \par
\ \par
We are now ready to state our second main result regarding the convergence in the zero-radius limit 
$\varepsilon \to 0$, of the solutions of \eqref{eqeps-WV}--\eqref{vdecompeps} towards the ones of
\eqref{eql1-WV}--\eqref{vdecomp*}, up to the maximal existence time  $T^* > 0$ of the latter. 
\begin{thm} \label{pasdenom-VO2}
Let ${\mathcal C}_0$ be a loop in $\R^3$. 
Let $ \mathcal A_0$ and  $\mathcal V_0$ be given by Definition~\eqref{svmatrix},
and $B^*$ given by \eqref{defBstar}. 
Let $m>0$ and let $\mathcal{J}_{0}$ be a $3 \times 3$ symmetric positive definite matrix. 
Let $p_0 := ( \ell_0 , \Omega_0 ) \in \R^6$ and $\omega_0\in C^{\lambda,r}(\R^3 ;\, \R^3)$ be divergence-free and
compactly supported away from ${\mathcal S}_0$ with $\lambda\in \N_{\geq 0}$ and $r\in (0,1)$ and let the circulation $\mu\in \R$ be fixed. 
Let $T^* > 0$ and $(p^*,u^* , \omega^*)$ be the unique strong solution of the limit system \eqref{eql1-WV}--\eqref{vdecomp*} on $ [0,T^*) $ as  given by Proposition~\ref{limit-WV}. 
Let, for each $\varepsilon > 0$,  $T^\varepsilon >0$ and $(p^\varepsilon , u^\varepsilon, \omega^\varepsilon )$ 
be the unique smooth solution  of \eqref{eqeps-WV}--\eqref{vdecompeps} with initial data $p(0) = p_0$,
on the time interval $ [0,T^\varepsilon)$. 

Then 
$$\liminf_{\varepsilon \to 0} \, T^{\varepsilon} \geq T^* \quad \text{for all} \quad   T\in (0,T^*),$$
\begin{equation*}
p^\varepsilon\rightarrow p^*\quad \text{ in $W^{\lambda,\infty}([0,T];\,\R^6)-w^*$ as $\eps\rightarrow 0^+$},
\end{equation*}
and 
\begin{equation*}
\omega^\varepsilon\rightarrow \omega^* 
\quad\text{in $L^\infty ([0,T];\, C^{\lambda,r}(\R^3 \setminus {\mathcal C}_0))-w^*$ as $\eps\rightarrow 0^+$}.
\end{equation*}
\end{thm}
 The proof of Theorem \ref{pasdenom-VO2} is given in Section  \ref{Sec:MainProof}.
%
%
%
\begin{rem}
The analysis here also works for the case of the Euler equations in a fixed domain, consisting of the $\R^3$ minus an obstacle shrinking to a curve, in which case the governing equations are \eqref{omega-eqeps}, \eqref{vdecompeps}, and \eqref{eqeps-WV} is replaced with $p=0$. In this case, the same proof shows that the dynamics converge in the same topologies to the solution of \eqref{omega-eq*}, \eqref{u*} with $p=0$ replacing \eqref{eql1}. These limiting equations can be interpreted as the Euler equations, where the Biot-Savart law is modified to contain an additional vortex filament, similar to how fixed point vortices can occur in the 2D vanishing obstacle limit (cf. \cite{Ift,Lopes}).
\end{rem}
%

\subsection{An example of divergence in the massless limit}

One can also consider the massless limit of the body dynamics, where the inertia/rotational inertia scales 
as for a body of constant density. More precisely we set 
\begin{equation} \label{Eq:LightScaling}
\tilde{m}^\eps = \eps^2 m \quad \mathrm{and} \quad \tilde{\mathcal{J}}_0^\eps=\eps^2\mathcal{J}_0,
\end{equation} 
where $m>0$ is fixed and $\mathcal{J}_0$ is a $3\times 3$ symmetric positive definite matrix 
and both are fixed independently of $\eps$. 
Then we consider $\mathcal{M}_g$ (and $\Gamma_g$) defined with respect to this $\tilde{m}^\eps$ and $\tilde{\mathcal{J}}^\eps_0$. 
For the sake of better distinction, we will denote this inertial matrix and the corresponding Christoffel symbol 
with $\tilde{\mathcal{M}}_g^\eps$ and $\tilde{\Gamma}_g^\eps$ instead of $\mathcal{M}_g$ and $\Gamma_g$.

Then, similarly as in \eqref{eqeps-WV}-\eqref{vdecompeps}, the governing equations read as 
\begin{equation} \label{eqeps-WV massless}
\big(  \tilde{\mathcal M}_{g}^\eps +  {\mathcal M}_{a}^\varepsilon \big) (p^\varepsilon)' 
 + \langle  \tilde{\Gamma}_g^\eps ,  p^\varepsilon ,  p^\varepsilon \rangle
 + \langle  \Gamma_a^\varepsilon ,  p^\varepsilon ,  p^\varepsilon \rangle
=  \mu B^\varepsilon p^\varepsilon 
 + {\mathcal B} \lbrack  K_{ \mathcal F_0} [\omega^\varepsilon] \rbrack p^\varepsilon
 + {\mathcal D} \lbrack  p^\varepsilon, u^\varepsilon, \omega^\varepsilon \rbrack ,
\end{equation}
\begin{equation} \label{omega-eqeps massless}
\partial_t \omega^\varepsilon + \big((u^\varepsilon - u^\varepsilon_{\mathcal{S}}) \cdot\nabla\big) \omega^\varepsilon
=  \big(\omega^\varepsilon  \cdot\nabla \big) (u^\varepsilon - u^\varepsilon_{\mathcal{S}}), 
\quad  \text{ in } \quad  \mathcal{F}_0^\varepsilon ,
\end{equation}
\begin{equation} \label{vdecompeps massless}
u^\varepsilon =  \mu   H^\varepsilon 
+ \sum_{1 \leqslant i \leqslant 6}   p^\varepsilon_i \nabla \Phi_i^\varepsilon 
+ K_{ \mathcal F_0^\varepsilon} [\omega^\varepsilon], \quad 
\text{ in } \quad  \mathcal{F}_0^\varepsilon. 
\end{equation}

\noindent Surprisingly, the solutions of Equation~\eqref{eqeps-WV massless} can diverge in the $\eps\searrow 0$ limit, 
as shown in the following statement.
\begin{thm} \label{thm:div}
Let $\mathcal{C}_0$ be a circle around the $e_3$-axis.
There exists a smooth, compactly supported $\omega_0$ with $\dist(\supp\omega_0,\mathcal{C}_0)>0$, an initial $p_0=(\ell_0,\Omega_0)$ and a $T_0>0$, all independent of $\eps$, such that in the $\eps\searrow0$ limit of \eqref{eqeps-WV massless}-\eqref{vdecompeps massless}, we have 
\begin{equation*}
p^\eps(t)\cdot e_3\geq \eps^{-\frac{1}{10}},
\end{equation*} 
for all $t\in (\eps,T_0)$ for all small enough $\eps$. 
In particular, the body travels a distance $\geq \eps^{-\frac{1}{10}}$ in the original frame.

Furthermore, the vorticity in these solutions does not blow up in the sense that 
the solutions of \eqref{omega-eqeps massless}-\eqref{vdecompeps massless} exist at least up to time $T_0$ for small $\eps>0$ 
and for every $m\in \N_{\geq 0}$ we have 
\begin{gather}
\label{w nice 1}
\norm{\omega^\eps}_{L^\infty([0,T_0],C^m(\R^3))} \lesssim_m 1, \\
\label{w nice 2}
\min_{t \in [0,T_0]} \dist(\supp\omega_t^\eps,\cals^\eps) \gtrsim 1,
\end{gather}
uniformly in $\eps$.
\end{thm}
Theorem~\ref{thm:div} is proved in Section~\ref{Sec:Proof-div}.

\begin{rem}
\textbf{1)} This result is in striking contrast with the situation in 2D, 
where the solution of the analogous kind of equation converges to the motion of a point vortex \cite{GMS,GS}.

It is quite remarkable that the data in the theorem can even be chosen to be axisymmetric, 
where it is known \cite{Meyer} that bodies which are allowed to undergo some non-rigid motions behave similarly to infinitesimal vortex filaments in an appropriate limit. There is no contradiction with \cite{Meyer} here because the bodies in \cite{Meyer} are allowed to undergo other motions. 

Let us also stress here that this divergence is on a completely different scale than the (logarithmic) divergence of the local induction approximation of vortex filaments/rings (cf.\ \cite{BanicaFilaments}).

\textbf{2)} Strictly speaking, the rotational inertia of $\mathcal{S}_0^\eps$ will only equal \eqref{Eq:LightScaling} to leading order if the density of the body is kept constant. The theorem is however still true with the same proof if one incorporates lower-order terms in \eqref{Eq:LightScaling}.

\end{rem}

 \subsection{Organization of the rest of the paper}

The last chapters of the paper are devoted to the proofs of Theorems~\ref{pasdenom-VO}, \ref{pasdenom-VO2} and \ref{thm:div}. 
In Section~\ref{sec-Asymptotic}, we establish asymptotic behaviors of various parts of the velocity field: the Kirchhoff potentials, the Biot-Savart field, and the harmonic field. 
In Section~\ref{Sec:compu}, we obtain asymptotic expansions for the various coefficients driving the solid's ODE.
We prove Theorems~\ref{pasdenom-VO}, \ref{pasdenom-VO2} and \ref{thm:div} in Sections~\ref{sec-cv}, \ref{Sec:MainProof} and \ref{Sec:Proof-div}, respectively.
Finally, in the  Appendix (Section \ref{app-IPP}), we establish some decay estimates to justify the integrations by parts done in the proof of Proposition \ref{reform-macro}, and in the Appendix (Section~\ref{Sec:WP-limit}), we prove that the limit system obtained in Theorem~\ref{pasdenom-VO2} is well-posed, as well as the system for $\varepsilon >0$, with uniform estimates with respect to $\varepsilon$.


\section{Asymptotics of the potential and the harmonic field}
\label{sec-Asymptotic}
This section is devoted to the computation of the asymptotic of the potential and the harmonic field. 

\subsection{Notation}\label{geo not}

Below we will often use the following coordinates system around $\mathcal{C}_0$, that is, 
using the notation of Section~\ref{Subsec:toro}, we will use the diffeomorphism 
\begin{equation} \label{def big J}
W(x):=W \left( \gamma(t) + a s_1(x) + b {s_2}(x) \right)= \left( t(x), a, b \right),
\end{equation}
which maps $\mathcal{O}$ to $\R/L\Z\times B_{\delta}(0)\subset \R/L\Z\times \R^2$,
where we recall that $\mathcal{O}$ was defined in \eqref{Def:CalO}.
Concerning this change of coordinates, we have the following lemma.
\begin{lem} \label{lemma j}
The modulus of the Jacobian determinant of $W$ is given by 
\begin{equation} \label{def j}
w(x) := \de_\tau t(x) .
\end{equation}
It also holds that \begin{align}\label{j identity}
\int_{\de\cals} f \dd\sigma = \int_{\R/L\Z} \int_{C_t^\eps} w^{-1} f \dx \dt,
\end{align}

\noindent In $\mathcal{O}$ it holds that
\begin{gather}
\label{bd j1}
| w - 1 | \lesssim \dist(x,\mathcal{C}_0),  \\
\label{bd j2}
| \nabla w | \lesssim 1.
\end{gather}
\end{lem}
\begin{proof}
By definition it holds that $\de_\tau W\cdot e_3=\de_\tau t$, hence, in the orthonormal frames $\tau(x)$, ${s_1}(x)$, ${s_2}(x)$ and $e_3,e_1,e_2$ the Jacobian $w$ of $W$ is
\begin{equation} \label{jac j}
\begin{pmatrix}
\de_\tau W\cdot e_3 & 0 & 0\\
\de_\tau W\cdot e_1 & 1 & 0\\
\de_\tau W\cdot e_2 & 0 & 1
\end{pmatrix}
\end{equation}
since $t$ is constant on the plane orthogonal to $\tau$ by definition. 
As the Jacobian determinant is independent of the choice of the orthonormal frame, this shows the first statement.

Regarding the formula \eqref{j identity}, we note by the definition of $W$ and Fubini, it holds that \begin{align*}
\int_{\R/L\Z} \int_{C_t^\eps} w^{-1} f \dd\mathcal{H}^1(x) \dd\mathcal{H}^1(t)=\int_{\R/L\Z\times \de B_\eps(0)} (f\circ W) (w^{-1}\circ W)\dd\sigma.
\end{align*}
Hence, it suffices to show that the Jacobian factor of the map $W$ is given by $w$. The restriction of $W$ is a bijection from $\de\cals^\eps$ to $\R/L\Z\times \de B_\eps(0)$. Therefore, by the area formula, this factor is (locally) given by $\frac{G(\mathrm{D}(W\circ \phi))}{G(\mathrm{D} \phi)}$, where $\phi$ is any smooth parametrization and $G$ is the Gramian determinant of the gradient.
Recalling that the Gramian determinant of a matrix $A$ equals the modulus of the determinant of $A$ restricted to its image with respect to any orthonormal basis of the image, we see that this quotient equals the determinant of the gradient of $W|_{\de\cals^\eps}$ with respect to any orthonormal basis of the tangent spaces of $\de\cals^\eps$ and $\R/L\Z\times \de B_\eps(0)$, independently of $\phi$.

The vectors $\tau$ and $e_\theta$ are such a basis of the tangent space of $\de\cals^\eps$. As an orthonormal basis of $\R/L\Z\times \de B_\eps(0)$, we take $e_3$ and $\tilde{e}_\theta$, the tangential vector in the azimuthal direction (with the same orientation as $e_\theta)$.

With respect to these bases the Jacobian of $W|_{\de\cals^\eps}$ is then given by 
\begin{equation*}
\begin{pmatrix}
\de_\tau W \cdot e_3 & 0 \\
\de_\tau W \cdot \tilde{e}_\theta & 1 
\end{pmatrix},
\end{equation*}
since $t(x)$ is constant in the $e_\theta$-direction.
Clearly, this matrix has the desired determinant.

For the bounds, we first note that on $\mathcal{C}_0$ it holds that $w=1$ by definition, 
and hence \eqref{bd j2} implies \eqref{bd j1}. 
The bound \eqref{bd j2} simply follows from the fact that $t(\cdot)$ and $\tau$ are $C^2$ functions that do not depend on $\eps$.
\end{proof}
%
%

%
%

\subsection{A priori estimates on the exterior Neumann problem}

We begin with a lemma estimating the dependence on $\varepsilon$ of the trace inequality on $\partial \mathcal{S}_0^\varepsilon$. We remark that similar estimates also appeared in \cite{Mori}, however, we prefer to keep the presentation self-contained here.
\begin{lem} \label{trace lemma}
We have the following trace inequality for all $v\in H^1(\mathcal{O} \backslash \cals^\varepsilon)$ 
\begin{equation} \label{trace ineq 2}
\norm{v}_{L^2(\de\cals^\eps)/\const}^2 \lesssim \eps|\log\eps| \int_{\mathcal{O} \backslash \cals^\eps} |\nabla v|^2 \dx,
\end{equation}
uniformly in $\eps$, where we set 
\begin{equation*}
\norm{v}_{L^2(\de\cals^\eps)/\const} := \inf_{a \in \R} \norm{ v - a }_{L^2(\de\cals^\eps)}.
\end{equation*}
\end{lem}
The proof relies on the following auxiliary lemma.
\begin{lem} \label{1D lem}
For all $g\in H^1((\eps,\delta))$ it holds that, for small enough $\eps>0$ and fixed $\delta \gtrsim 1$,
\begin{equation} \label{1D ineq}
|g(\eps)|^2 \lesssim |\log\eps| \left( \int_\eps^\delta x|g'|^2 \dx + \int_{\delta/2}^\delta |g|^2 \dx \right).
\end{equation}
\end{lem}
\begin{proof}[Proof of Lemma \ref{1D lem}]
We make the transformation $h(x)=g(e^x)$, and estimate with the fundamental theorem of calculus and the substitution rule
\begin{align*}
|g(\eps)|^2
& = |h(\log\eps)|^2 \lesssim |h(\log\delta)|^2 + \left( \int_{\log\eps}^{\log\delta} |h'| \dx \right)^2 \\
& \lesssim |h(\log\delta)|^2 + \left| \log\frac{\delta}{\eps} \right|  \int_{\log\eps}^{\log\delta} |h'|^2 \dx \\
& \lesssim |g(\delta)|^2 + \left| \log\frac{\delta}{\eps} \right| \int_{\eps}^{\delta}x |g'|^2 \dx .
\end{align*}
As we have $|g(\delta)| \lesssim \norm{g}_{H^1(\delta/2,\delta)}$, this shows \eqref{1D ineq}.
\end{proof}

\begin{proof}[Proof of Lemma \ref{trace lemma}]
We use the diffeomorphism $W$ defined in \eqref{def big J}.
As it does not depend on $\eps$, we see that $1\lesssim\mathrm{D}W\lesssim 1$, in the sense that the smallest 
and highest eigenvalue fulfill these bounds, and its Jacobian determinant is also bounded uniformly in $\eps$
by Lemma \ref{lemma j}.
Hence, by setting $f=v\circ W$, we see that \eqref{trace ineq 2} is equivalent to showing that 
\begin{equation} \label{trace fixed dom}
\norm{f}_{L^2(\de B_\eps(0) \times \R/L\Z)/\const}^2 
  \lesssim \eps|\log\eps| \int_{\left(B_{\delta}(0) \backslash B_\eps(0) \right) \times \R/L\Z} |\nabla f|^2 \dx,
\end{equation}
for all $f \in H^1(\left(B_{\delta}(0) \backslash B_\varepsilon(0)\right)\times \R/L\Z)$, uniformly in $\eps$. 
It is not restrictive to assume that $f$ is smooth.

To show \eqref{trace fixed dom}, we first apply the Poincar\'e-Wirtinger inequality to see that 
there is some constant $\bar v\in \R$ such that
\begin{equation} \label{poincare}
\norm{f - \bar f}_{L^2((B_{\delta}\backslash B_{\delta/2})\times \R/L\Z)}^2
  \lesssim \int_{\left(B_{\delta}(0) \backslash B_\eps(0)\right) \times \R/L\Z} |\nabla f|^2 \dx.
\end{equation}
We now use standard cylindrical coordinates $(t,r,\theta)$ in $\left(B_{\delta}(0)\backslash B_\eps(0)\right)\times \R/L\Z$.
We apply \eqref{1D ineq} to $f-\bar f$ in the variable $r$ only and integrate over $\theta$ and $t$ to see that
\begin{align*}
& \int_{\partial B_\eps(0) \times \R/L\Z} | f - \bar f|^2 \dd\sigma
 = \eps \int_{[0,2\pi] \times \R/L\Z} | f(t,\eps,\theta) - \bar f|^2 \dd \theta \dd t \\
& \lesssim \eps|\log\eps|\left(\int_{[\eps,\delta]\times[0,2\pi]\times \R/L\Z} r|\nabla f|^2\dd r\dd\theta\dd t
 + \int_{[\delta/2,\delta] \times [0,2\pi] \times \R/L\Z} | f - \bar f |^2 \dd r \dd\theta \dd t \right) \\
& \lesssim \eps|\log\eps| \left(\int_{\left( B_{\delta}(0) \backslash B_\eps(0) \right)\times \R/L\Z} | \nabla f |^2 \dd x
 + \int_{\left(B_{\delta}(0) \backslash B_{\delta/2}(0)\right) \times \R/L\Z} | f - \bar f |^2 \dd x \right),
\end{align*}
where we used that $r \lesssim 1$ in the first integral and $r \gtrsim 1$ in the second one for the last step.
By \eqref{poincare}, this implies \eqref{trace fixed dom}.
\end{proof}
\begin{lem} \label{cor trace}
Let $\phi$ be harmonic in $\R^3 \backslash \cals^\eps$, vanishing at $\infty$, and such that $\int_{\de\cals^\eps} \de_n \phi \dd\sigma = 0$, then we have 
\begin{align}
\label{est L2 norm} 
& \norm{\phi}_{L^2(\de\cals^\eps)/\const} \lesssim \eps|\log\eps| \norm{ \de_n \phi}_{L^2(\de\cals^\eps)}, \\
\label{est nabla}
& \norm{\nabla\phi}_{L^2(\R^3\backslash \cals^\eps)} \lesssim \eps^\frac{1}{2}|\log\eps|^{\frac{1}{2}}
  \norm{\de_n\phi}_{L^2(\de\cals^\eps)}, \\
\label{ell est}
&\norm{\nabla \phi}_{L^2(\de\cals^\eps)} \lesssim |\log\eps|^{\frac{1}{2}} \norm{\de_n \phi}_{L^2(\de\cals^\eps)},
\end{align}
uniformly in $\eps$.
\end{lem}
\begin{proof}
Partial integration reveals that 
\begin{equation*}
\int_{\R^3 \backslash \cals^\eps} |\nabla \phi|^2 \dx = \int_{\de\cals^\eps} \phi \de_n \phi \dd\sigma.
\end{equation*}
There is no contribution from $\infty$, for instance by the decay estimate \eqref{bddec} below. 
Using the assumption that $\de_n\phi$ is mean-free, we see that for all $a\in \R$ we have 
\begin{equation*}
\int_{\de\cals^\eps} (\phi - a) \de_n \phi \dd\sigma = \int_{\de\cals^\eps} \phi \de_n \phi \dd\sigma,
\end{equation*}
and by the Cauchy-Schwarz inequality, we have
\begin{equation*}
\norm{\nabla\phi}_{L^2(\R^3\backslash\cals^\eps)}^2 
  \leq \norm{\phi}_{L^2(\de\cals^\eps) / \const} \norm{\de_n\phi}_{L^2(\de\cals^\eps)},
\end{equation*} 
and hence, by Lemma~\ref{trace lemma} above,
\begin{equation*}
\norm{\nabla\phi}_{L^2(\R^3 \backslash \cals^\eps)} 
  \lesssim \eps^\frac{1}{2}|\log\eps|^{\frac{1}{2}} \norm{ \de_n \phi }_{L^2(\de\cals^\eps)},
\end{equation*}
that is, \eqref{est nabla}.
By applying Lemma \ref{trace lemma} a second time, we conclude \eqref{est L2 norm}.

To see \eqref{ell est}, we use Lemma \ref{lamb}, which in the case $\tilde u=\tilde{v}$ and $\curl \tilde u = 0$ reads as
\begin{equation} \label{gen lamb}
\int_{\de\cals^\eps} |\tilde u|^2 z \cdot n \dd\sigma - 2 \int_{\de\cals^\eps} (\tilde u\cdot n)(\tilde u\cdot z) \dd\sigma
= \int_{\R^3 \backslash \cals^\eps} |\tilde u|^2 \div z\dx 
- 2\int_{\R^3 \backslash \cals^\eps} \tilde u \cdot ((\tilde u \cdot \nabla) z) \dx.
\end{equation}
 We then set $\tilde u=\nabla \phi$, which is naturally curl-free, and take $z=n \eta$, where the extension of the normal 
 was defined in \eqref{ext normal} and $\eta$ is a smooth cutoff function, independent of $\eps$, 
 which is $1$ on a neighborhood of $\cals^\eps$ and supported in $\mathcal{O}$.
Then it follows from the definition that $| \nabla z | \lesssim \eps^{-1}$ in $\R^3 \setminus \mathcal{S}_0^\varepsilon$
and hence we see from rearranging \eqref{gen lamb} that 
\begin{equation*}
\int_{\de\cals^\eps} |\nabla\phi|^2 \dd\sigma
  \lesssim \int_{\de\cals^\eps} | \de_n \phi|^2 \dd\sigma
  + \eps^{-1} \int_{\mathcal{O} \backslash \cals^\eps} | \nabla\phi |^2 \dx,
\end{equation*}
which, together with \eqref{est nabla}, shows \eqref{ell est}.
\end{proof}
\begin{lem}\label{harmonic decay}
Let $\phi$ be harmonic in $\R^3\backslash \cals^\eps$ and vanishing at $\infty$ with $\int_{\de\cals^\eps} \de_n\phi \dd\sigma=0$, 
then it holds that 
\begin{equation} \label{bddec}
|\nabla^m \phi(x)| 
\lesssim_m \,  \eps^{\frac{1}{2}} \norm{\de_n\phi}_{L^2(\de\cals^\eps)} \dist(x,\cals^\eps)^{-2-m},
\end{equation}
for all $m\in \N_{\geq 0}$ and all $x\in \R^3\backslash\cals^\eps$.

\end{lem}

\begin{proof}
We only show the case $m=0$, the other cases work in a completely similar fashion by using the derivatives 
of the fundamental solution instead of the fundamental solution itself. 
We start with the following integral representation of the harmonic function $\phi$: 
for $x$ in $\R^3\backslash \de\cals^\eps$, 
\begin{equation} \label{bd integral}
\phi(x) = - \int_{\de\cals^\eps} \Big( \frac{-1}{4\pi |x-y| } \de_n \phi - \phi \de_n \frac{-1}{4\pi|x-y|} \Big) \dd\sigma(y).
\end{equation}
In the first integrand, we can use that $\de_n\phi$ is mean-free by assumption and estimate $\frac{1}{|x-y|}$ with its gradient to see that 
\begin{align*}
\left| \int_{\de\cals^\eps} \frac{1}{4\pi |x-y|} \de_n \phi \dd\sigma(y) \right|
& \lesssim  \norm{\de_n\phi}_{L^2(\de\cals^\eps)} \norm{ \nabla \frac{1}{4\pi |x-\cdot| }}_{L^2(\de\cals^\eps)} \\
& \lesssim \eps^\frac{1}{2} \norm{\de_n \phi}_{L^2(\de\cals^\eps)} \dist(x,\cals^\eps)^{-2} .
\end{align*}
For the second integrand, we use that
\begin{equation*}
\int_{\de\cals^\eps} \de_n \frac{1}{|x-y|} \dd\sigma(y) = \int_{\cals^\eps} \Delta \frac{1}{|x-y|} \dy = 0.
\end{equation*}
Hence, using \eqref{est L2 norm} we see
\begin{align*}
\left| \int_{\de\cals^\eps} \phi \de_n \frac{1}{|x-y|} \dd\sigma(y) \right|
& \leq \norm{\phi}_{L^2(\de\cals^\eps)/\const} \norm{ \nabla \frac{1}{|\cdot-x|} }_{L^2(\de\cals^\eps)} \\
& \lesssim \eps^{\frac{3}{2}} |\log\eps|  \norm{ \de_n \phi }_{L^2(\de\cals^\eps)} \dist(x,\de\cals^\eps)^{-2}.
\end{align*}
Gathering the two inequalities above, we find \eqref{bddec}.
\end{proof}

\subsection{Estimates on the potentials}
%
%
%
We now apply the estimates above to the Kirchhoff potentials and the Biot-Savart fields as defined in Section \ref{Subsec:toro}.
Concerning the former, we have the following result. 
\begin{prop} \label{est pot}
For all $i\in\{1,\dots,6\}$ and all $m\in\N_{\geq 0}$ it holds uniformly in $\eps$ and $x\in\R^3\backslash \cals^\eps$ that 
\begin{align*}
& \norm{\nabla\Phi_i}_{L^2(\de\mathcal{S}_0^\eps)} \lesssim |\log\eps|^\frac{1}{2} \eps^\frac{1}{2} ,\\
& \norm{\nabla \Phi_i}_{L^2(\R^3\backslash\mathcal{S}_0^\eps)} \lesssim \eps|\log\eps|^\frac{1}{2} ,\\
& |\nabla^m\Phi_i(x)| \lesssim_m \eps \, \dist(x,\de\mathcal{S}_0^\eps)^{-2-m}.
\end{align*}
\end{prop}
\begin{proof}
This is a direct consequence of Lemmas~\ref{cor trace} and \ref{harmonic decay}. 
In fact, it suffices to observe that
$\| \de_n \Phi_i \|_{L^2(\partial \mathcal{S}_0^\varepsilon)} \lesssim \varepsilon^{\frac{1}{2}}$ 
and that,
since $\de_n \Phi_i = e_i \cdot n$ for $i=1,2,3$ 
and $\de_n \Phi_i = n \cdot (e_{i-3} \wedge x)$ for $i=4,5,6$,
the assumption that the normal component should be mean-free over $\de\cals$ is satisfied in all cases.
\end{proof}
We now turn to the Biot-Savart field.
We can relate the Biot-Savart law $K_{\mathcal{F}_0}$ in $\mathcal{F}_0$ 
to the Biot-Savart law $K_{\R^3}$ in $\R^3$ through the identity 
\begin{equation} \label{def ref}
K_{\mathcal{F}_0}[\omega] = K_{\R^3}[\omega] + \uref[\omega] ,
\end{equation}
where the {\it reflection term} $\uref[\omega]$ is defined through (recalling the notation \eqref{Def:CTeps})
\begin{align} 
\label{a1}
& \div \uref[\omega] = \curl \uref[\omega] = 0 \text{ in $\mathcal{F}_0$} , \\
\label{a2}
& \uref[\omega]\cdot n = - K_{\R^3}[\omega] \cdot n \text{ on $\de\mathcal{F}_0$} , \\
\label{a3}
& \int_{C_0^\eps} \uref[\omega] \cdot e_\theta\ds = 0 , \\
\label{a4}
& \lim_{|x|\rightarrow \infty} \uref[\omega] = 0.
\end{align}
As it is $\curl$-free and has $0$ circulation around ${C}_0^\varepsilon$, 
the function $\uref[\omega]$ can also be written as a gradient of a potential $\Phi_{\mathrm{ref}}[\omega]$ so that 
\begin{equation*}
\nabla \Phi_{\mathrm{ref}}[\omega] = \uref[\omega].
\end{equation*}
We can obtain the following estimates on $\uref$.
\begin{prop} \label{est ref}
Suppose $\omega\in L^1$ is compactly supported away from $\cals^\eps$.
Then, for the reflection term $\uref$, it holds uniformly in $\eps$, $x\in \R^3\backslash \mathcal{S}_0^\eps$, and $\omega$ that 
\begin{align*}
& \norm{\uref}_{L^2(\de\mathcal{S}_0^\eps)} \lesssim 
  |\log\eps|^\frac{1}{2}\eps^\frac{1}{2} \, \norm{\omega}_{L^1} \dist(\cals^\eps,\supp\omega)^{-2}, \\
& \norm{\uref}_{L^2(\R^3\backslash\mathcal{S}_0^\eps)} \lesssim 
   \eps|\log\eps|^\frac{1}{2} \, \norm{\omega}_{L^1} \dist(\cals^\eps,\supp\omega)^{-2}, \\
& |\nabla^m\Phi_{\mathrm{ref}}[\omega](x)| \lesssim_m 
  \eps \, \norm{\omega}_{L^1} \dist(x,\de\mathcal{S}_0^\eps)^{-2-m} \dist(\cals^\eps,\supp\omega)^{-2},
\end{align*}
for all $m\in \N_{\geq 0}$.
\end{prop}
\begin{proof}
As we have $\uref\cdot n=-K_{\R^3}[\omega]\cdot n$ on $\de\cals$, 
we see from the explicit form of the Biot-Savart law \eqref{KR2}, that 
\begin{equation*}
|\de_n \Phi_{\mathrm{ref}}[\omega]| = |\uref[\omega] \cdot n|
\lesssim \norm{\omega}_{L^1} \dist(\cals,\supp \omega)^{-2},
\end{equation*}
pointwise on $\de\cals^\eps$. We also have the following: 
\begin{align*}
\int_{\de\cals^\eps} \uref[\omega] \cdot n \dd\sigma
= - \int_{\de\cals^\eps} K_{\R^3}[\omega] \cdot n\dd\sigma
= -\int_{\cals^\eps} \div K_{\R^3}[\omega] \dx = 0.
\end{align*}
The statement then follows directly from Lemmas \ref{cor trace} and \ref{harmonic decay} 
applied to $\Phi_{\mathrm{ref}}[\omega]$.
\end{proof}

%
%

\subsection{Expansion of the harmonic field}
Concerning the harmonic field, which is the most singular part, we have the following estimates.
\begin{prop} \label{est vor filament}
Let 
\begin{equation} \label{defH2D}
     H_{2D} = \frac{1}{2\pi\eps} e_\theta .
\end{equation}
Then for all $x\in \R^3\backslash \cals^\eps$ 
and all $m\in \N_{\geq 0}$ it holds that
\begin{align*}
& \norm{ H^\eps - H_{2D} }_{L^2(\de\cals^\eps)} 
  \lesssim \eps^\frac{1}{2} |\log\eps|^{\frac{3}{2}}, \\
& \norm{ H^\eps - K_{\R^3}[ \kappa_{\mathcal{C}_0} ] }_{L^2(\R^3 \backslash \cals^\eps)} 
  \lesssim \eps |\log\eps|^\frac{3}{2}, \\
& |\nabla^m(H^\eps-K_{\R^3}[\kappa_{\mathcal{C}_0}])(x)| 
  \lesssim_m \eps |\log\eps| \dist(x,\de\mathcal{S}_0^\eps)^{-3-m}.
\end{align*}
\end{prop}
\begin{proof}[Proof of Proposition \ref{est vor filament}]
The proof requires the following Lemma, whose proof is postponed.
\begin{lem} \label{ind approx}
It holds that
\begin{equation*}
\norm{ K_{\R^3}[\kappa_{\mathcal{C}_0}] - H_{2D} }_{L^\infty(\de\cals^\eps)} \lesssim |\log\eps|,
\end{equation*}
uniformly in $\eps$.
\end{lem}
The crucial observation is that $H^\eps-K_{\R^3}[\kappa_{\mathcal{C}_0}]$ is harmonic 
in $\R^3\backslash \cals^\eps$ and has $0$ circulation around $C_0^\varepsilon$.
The harmonicity is obvious.
Concerning the circulation, since $H^\eps$ has circulation $1$ by definition (see \eqref{def H},
it suffices to check that $K_{\R^3}[\kappa_{\mathcal{C}_0}]$ has circulation $1$ 
as well. 
This follows from Lemma~\ref{ind approx}, by using that the circulation is the same 
for all homology equivalent loops in $\R^3\backslash \mathcal{C}_0$ by Stokes' theorem.
Indeed we can apply the Lemma with a different $\eps'>0$ to the loop $C_0^{\eps'}$ and obtain that 
\begin{multline*}
\int_{C_0^\eps} e_\theta \cdot K_{\R^3}[\kappa_{\mathcal{C}_0}] \dx
= \int_{C_0^{\eps'}} e_\theta \cdot K_{\R^3}[\kappa_{\mathcal{C}_0}] \dx \\
= \int_{C_0^{\eps'}} \frac{1}{2\pi\eps'}e_\theta \cdot e_\theta\dx + O(\eps'|\log\eps'|) 
= 1 + O(\eps'|\log\eps'|).
\end{multline*}
As $\eps'$ was arbitrary, this shows that the circulation is $1$. 
%

Therefore, there exists a $\tilde{\phi}:\R^3\backslash \cals^\eps$ such that 
\begin{equation*}
\nabla \tilde{\phi} = H^\eps-K_{\R^3}[\kappa_{\mathcal{C}_0}] 
\quad \text{and} \quad
\Delta\tilde{\phi} = 0.
\end{equation*}
Furthermore, on $\de\cals^\eps$, it holds that $n\cdot H^\eps=0$ by definition and hence 
\begin{equation*}
\de_n\tilde{\phi} = - K_{\R^3}[\kappa_{\mathcal{C}_0}] \cdot n,
\end{equation*}
which is $\lesssim |\log\eps|$ pointwise on $\de\cals^\eps$ by Lemma~\ref{ind approx}, 
since $e_\theta$ and $n$ are orthogonal by definition.

It holds that
\begin{equation*}
\int_{\de\cals^\eps} \de_n \tilde{\phi} \dd\sigma 
= - \int_{\de\cals^\eps} n \cdot K_{\R^3}[\kappa_{\mathcal{C}_0}] \dd\sigma
= \int_{\cals^\eps} \div K_{\R^3}[\kappa_{\mathcal{C}_0}] \dx = 0,
\end{equation*}
where the last equality follows from $K_{\R^3}[\kappa_{\mathcal{C}_0}]$ being divergence-free (as it is a curl by definition).

Hence, using Lemma \ref{ind approx} and that $\mathcal{H}^2(\de\cals^\eps) \approx \eps$, we have that 
\begin{align*}
\norm{H_{2D}-H^\eps}_{L^2(\de\cals^\eps)}
& \leq \norm{ H_{2D} - K_{\R^3}[\kappa_{\mathcal{C}_0}] }_{L^2(\de\cals^\eps)}
  +\norm{ \nabla \tilde{\phi} }_{L^2(\de\cals^\eps)} \\
& \lesssim \eps^\frac{1}{2} |\log\eps| + |\log\eps|^\frac{1}{2} \norm{ \de_n \tilde{\phi} }_{L^2(\de\cals^\eps)} \\
& \lesssim \eps^\frac{1}{2} |\log\eps|^\frac{3}{2},
\end{align*}
where we used the \textit{a priori} estimate from Lemma \ref{cor trace} in the penultimate step.

For the second estimate we have by \eqref{est nabla}
\begin{equation*}
\norm{ H^\eps - K_{\R^3}[\kappa_{\mathcal{C}_0}] }_{L^2(\R^3\backslash \cals^\eps)}
= \norm{ \nabla \tilde{\phi} }_{L^2(\R^3\backslash \cals^\eps)}
\lesssim \eps^\frac{1}{2} |\log\eps|^\frac{1}{2} \norm{ \de_n \tilde{\phi} }_{L^2(\de\cals^\eps)}
\lesssim \eps|\log\eps|^\frac{3}{2}.
\end{equation*}
The third estimate follows similarly from Lemma \ref{harmonic decay}.
\end{proof}

\begin{proof}[Proof of Lemma \ref{ind approx}]
This kind of computation is well-known, we provide the proof here for the convenience of the reader.
Let $x \in \partial \mathcal{S}_0^\varepsilon$.
It is not restrictive to assume that $x$ is such that $t(x)=0$.
We may write the definition of the Biot-Savart law as
\begin{equation} \label{curv pot}
K_{\R^3}[\kappa_{\mathcal{C}_0}](x)
= \int_{-\frac{L}{2}}^\frac{L}{2} \frac{1}{4\pi} 
\frac{\gamma(t)-\gamma(0) - x_{s_1} s_1(0) - x_{s_2} s_2(0)}{|\gamma(t) - \gamma(0) - x_{s_1} {s_1}(0) - x_{s_2} {s_2}(0)|^3}
\wedge\gamma'(t)\dt.
\end{equation}
Using the fact that $(x_{s_1} s_1(0) + x_{s_2} s_2(0)) \wedge \gamma'(0) =- e_\theta$ 
and the orthogonality of the three vectors $s_1(0)$, $s_2(0)$, $\gamma'(0)$, 
we may write, for $x \in \R^3 \setminus \mathcal{C}_0$,
\begin{multline} \label{cyc pot}
H_{2D}(x) = \frac{1}{2\pi\eps}e_\theta
= \frac{-1}{4\pi} \int_{\R} \frac{ x_{s_1} s_1(0) + x_{s_2} s_2(0) }{\sqrt{\eps^2 + t^2}^3} \wedge \gamma'(0)\dt \\
= \int_\R \frac{1}{4\pi}
\frac{t\gamma'(0) - x_{s_1} {s_1}(0) - x_{s_2} {s_2}(0)}{|t\gamma'(0) - x_{s_1} {s_1}(0) - x_{s_2} {s_2}(0)|^3}
\wedge \gamma'(0)\dt.
\end{multline}
We then obtain from \eqref{curv pot} and \eqref{cyc pot}, that 
\begin{align*}
& K_{\R^3}[\kappa_{\mathcal{C}_0}]-H_{2D}(x)
=\underbrace{ \int_{-\frac{L}{2}}^\frac{L}{2} 
\frac{1}{4\pi} \frac{\gamma(t)-\gamma(0)-x_{s_1}{s_1}(0)-x_{s_2}{s_2}(0)}{|\gamma(t)-\gamma(0)-x_{s_1}{s_1}(0)-x_{s_2}{s_2}(0)|^3}
\wedge \left( \gamma'(t) - \gamma'(0) \right) \dt}_{=:I} \\
& + \underbrace{\int_{\frac{L}{2}}^\frac{L}{2}
\frac{1}{4\pi} \left(
\frac{\gamma(t)-\gamma(0)-x_{s_1}{s_1}(0)-x_{s_2}{s_2}(0)}{|\gamma(t)-\gamma(0)-x_{s_1}{s_1}(0)-x_{s_2}{s_2}(0)|^3}
- \frac{t\gamma'(0)-x_{s_1}{s_1}(0)-x_{s_2}{s_2}(0)}{|t\gamma'(0)-x_{s_1}{s_1}(0)-x_{s_2}{s_2}(0)|^3} \right)
\wedge\gamma'(0)\dt}_{=:II} \\
& +\underbrace{\int_{\R\backslash[-\frac{L}{2},\frac{L}{2}]} \frac{-1}{4\pi}
\frac{t\gamma'(0)-x_{s_1}{s_1}(0)-x_{s_2}{s_2}(0)}{|t\gamma'(0)-x_{s_1}{s_1}(0)-x_{s_2}{s_2}(0)|^3}\wedge\gamma'(0)\dt}_{=:III}.
\end{align*}
Concerning $I$, we proceed as follows.
For small $|t|$, using the orthogonality of $\gamma'(0)$, $s_1(0)$, $s_2(0)$,
and the fact that $\gamma$ is $C^2$, we have the following estimate
\begin{align}
\nonumber 
\left| \gamma(0) \right. & \left. - \ \gamma(t) + s_1(x) e_{s_1}(0) + s_2(x) e_{s_2}(0) \right| \\
\nonumber 
& \geq |t\gamma'(0) + s_1(x) e_{s_1}(0) + s_2(x) e_{s_2}(0)| 
  - \left| \gamma(0) - \gamma(t) + t \gamma'(0) \right| \\
\label{den est 1}
& \gtrsim \sqrt{\eps^2+t^2} - O(t^2) \geq |t|+\eps,
\end{align}
for $t$ in some interval $[-K,K]$. \par
Now for large $|t|$, we use that, since $\gamma$ is $C^2$, 
we have
\begin{equation} \label{gamma est}
| \gamma'(t) - \gamma'(0) | \lesssim \min(1,|t|).
\end{equation}
Since $x \in \partial \mathcal{S}_0^\varepsilon$, we have that $x_{s_1}^2+x_{s_2}^2=\eps^2$.
Since $\mathcal{C}_0$ does not intersect itself, 
we deduce that for all $t \in [-\frac{L}{2},\frac{L}{2}] \setminus [-K,K]$,
\begin{equation} \label{den est 2}
|\gamma(0) - \gamma(t) + s_1(x) e_{s_1}(0) + s_2(x) e_{s_2}(0)| \gtrsim 1.
\end{equation}
Putting \eqref{den est 1}, \eqref{gamma est} and \eqref{den est 2} together, we obtain that 
\begin{equation*}
|I| \lesssim \int_{0}^1 t \frac{t + \eps}{|t+\eps|^3} \dt
+ \int_1^L1 \dt \lesssim |\log\eps|.
\end{equation*}
Similarly, we may bound $II$ by making use of the inequalities 
\begin{equation} \label{est den 3}
\begin{aligned}
& \left| \frac{a}{|a|^3} - \frac{b}{|b|^3} \right|
  \lesssim |a-b| \max\left( \frac{1}{|a|^3}, \frac{1}{|b|^3} \right), \\
& |t\gamma'(0) - s_1(x) e_{s_1}(0) - s_2(x) e_{s_2}(0)|
  = \sqrt{\eps^2 + t^2} \gtrsim t + \eps ,\\
& \left|\gamma(0) - \gamma(t) + t\gamma'(0) \right| \lesssim \min(1, |t|^2).
\end{aligned}
\end{equation}
This yields that 
\begin{align*}
|II|\lesssim \int_0^\infty \frac{\min(1,|t|^2)}{(t+\eps)^3}\dt\lesssim |\log\eps|.
\end{align*}
The bound of $III$ follows directly from \eqref{est den 3} and the fact that 
$\int_1^\infty \frac{1}{t^2} \dt \leq 1$. \par
This completes the proof of Proposition~\ref{est vor filament}.
\end{proof}
For future reference, we note that:
\begin{cor} \label{Cor:DecayH}
For all $m\in \N_{\geq 0}$ and all $x\in\R^3\backslash \cals$ we have 
\begin{align}
\label{decay H}
| \nabla^m H^\eps(x) | & \lesssim_m \dist(x, \de\cals^\eps)^{-3-m} , \\
\label{decay K}
|\nabla^m K_{\R^3}[\kappa_{\mathcal{C}_0}](x)| & \lesssim_m \dist(x,\de\cals^\eps)^{-3-m}. 
\end{align}
\end{cor}
\begin{proof}
Clearly \eqref{decay H} follows from Proposition \ref{est vor filament} and \eqref{decay K}. 
The estimate \eqref{decay K} (which corresponds classically to the estimate for a magnetic dipole)
follows directly from the definition of the Biot-Savart law and the fact that 
$\int \dd \kappa_{\mathcal{C}_0} = 0$.
\end{proof}


\section{Expansions of the reduced ODE coefficients}
\label{Sec:compu}
Recalling \eqref{calBdef} and \eqref{defH2D}, we consider the $6 \times 6$ skew-symmetric matrix $B$ given as follows:
\begin{equation} \label{Eq:Bij}
(B_{i,j} )_{1 \leq i,j \leq 6} := {\mathcal B} [ H_{2D} ] 
\text{ that is, }
B_{i,j} := \int_{\partial  \mathcal{S}_{0}^\eps  } \,  [ \zeta_{j} , \zeta_i , H_{2D} \wedge n ] \dd \sigma .
\end{equation}
Below, we prove that its leading part can be computed in terms of the matrix $B^*$ defined in \eqref{defBstar}.
\begin{prop} \label{compuB}
It holds that: 
${\mathcal B} [ H_{2D} ] = B^* + O(\eps )$.
\end{prop}
\begin{proof}[Proof of Proposition \ref{compuB}]
Starting from \eqref{Eq:Bij}, we note that $H_{2D} \wedge n = \frac{1}{2\pi\eps} \tau$ 
and use the identity \eqref{j identity} to see that 
\begin{equation*}
B_{i,j} = \frac{1}{2\pi\eps} \int_{\de\cals^\eps}[\zeta_j, \zeta_i ,\tau] \dd \sigma
= \frac{1}{2\pi\eps} \int_0^L \int_{C_t^\eps} w^{-1} [\zeta_j, \zeta_i, \gamma'] \dx \dt
= \int_0^L [\zeta_j, \zeta_i, \gamma'] \dt + O(\eps),
\end{equation*}
since by Lemma \ref{lemma j} we have $|w-1|\lesssim \eps$ pointwise and $|C_t^\eps|=2\pi \eps$ by definition. \par
For $i, j$ in $\{ 1,2,3 \}$, we obtain that  
\begin{equation*}
B_{i,j} = \left[ e_{j} , e_i ,  \int_0^L \, \gamma' \dt \right] + O(\eps) 
= O(\eps).
\end{equation*}
Similarly, for $i$ in $\{ 1,2,3 \}$ and $j$ in $\{ 1,2,3 \}$,
\begin{equation*}
B_{i,3+j} =  \int_0^L \, [e_{j}\wedge\gamma , e_i , \gamma'] \dt + O(\eps) .
\end{equation*}
But, using the Cauchy-Binet identity as in \eqref{CauchyBinet},
\begin{equation*}
[e_{j}\wedge\gamma , e_i , \gamma'] = 
(e_{j} \cdot  e_i ) (\gamma \cdot \gamma')-(e_{i} \cdot \gamma) (e_j \cdot \gamma') ,
\end{equation*}
and
\begin{equation*}
\int_0^L \, (\gamma \cdot \gamma') \dt= 0 .
\end{equation*}
Thus, by partial integration 
\begin{align*}
B_{i,3+j} & =- \int_0^L \, (e_{i} \cdot \gamma) (e_j \cdot \gamma')\dt + O(\eps) \\
& = -\frac12 \int_0^L \, (e_{i} \cdot \gamma) (e_j \cdot \gamma')\dt 
  + \frac12  \int_0^L \, (e_{j} \cdot \gamma) (e_i \cdot \gamma')\dt + O(\eps) \\
& =- \frac12 \int_0^L \,  \Big( (\gamma \wedge \gamma') \wedge e_i \Big) \cdot e_j\dt + O(\eps) \\
& =- [ e_i , e_j ,\mathcal{A}_0 ]  + O(\eps),
\end{align*}
where the last step directly follows from the Definition \ref{svmatrix} of $\mathcal{A}_0$. Finally, for  $i$ in $\{ 1,2,3 \}$ and  $j$ in $\{ 1,2,3 \}$, we have that 
\begin{equation*}
B_{i+3,j+3} = \int_0^L \, [e_{j} \wedge \gamma, e_i \wedge \gamma, \gamma'] \dt+ O(\eps) .
\end{equation*}
But
\begin{align*}
[e_{j} \wedge \gamma,  e_i \wedge \gamma ,\gamma' ] 
&= (e_{j} \wedge \gamma) \cdot \big( (  e_i \wedge \gamma ) \wedge \gamma') \big) \\
& = \big[ e_{j},\, \gamma, \, (e_i \wedge \gamma) \wedge \gamma' \big] \\ 
& = e_{j} \cdot \Big( \gamma \wedge \big( (e_i \wedge \gamma) \wedge \gamma' \big) \Big) \\ 
& = e_{j} \cdot \big( (\gamma \cdot \gamma') e_i \wedge \gamma \big),
 \end{align*}
so that 
\begin{align*}
B_{i+3,j+3} 
& = \left[ e_{j} ,  e_i ,  \int_0^L \, \big( \gamma \cdot\gamma' \big) \gamma\dt \right] + O(\eps) \\ 
& = [ e_{j}, e_i, \mathcal V_0 ] + O(\eps),
 \end{align*}
by an integration by parts and the Definition \eqref{svmatrix} of $\mathcal{V}_0$.
\end{proof}
Let $B^\eps := {\mathcal B} [H^\eps] \in \R^{6 \times 6}$.
We have the following asymptotic expansion for the first coefficients in \eqref{eqeps-WV}.
\begin{prop} \label{co}
For any $p\in\R^3$, as $\eps \rightarrow 0$, 
\begin{gather} 
\label{co1}
B^\eps = B^* + O(\eps |\log\eps|^\frac{3}{2}), \\
\label{co2}
{\mathcal M}_{a}^\eps = O(\eps^2 |\log\eps|), \\
\label{co3}
\langle \Gamma_a^\eps, p, p \rangle =  O(\eps|\log\eps|^\frac{1}{2}).
\end{gather}
\end{prop}
\begin{proof}
We use Proposition~\ref{compuB} as well as the Cauchy-Schwarz inequality to obtain that
\begin{equation*}
\big|\, \mathcal{B}[H^\eps] - \mathcal{B}[H_{2D}] \,\big|\
  \lesssim \eps^\frac{1}{2} \norm{ H^\eps - H_{2D} }_{L^2(\de\cals^\eps)},
\end{equation*}
and \eqref{co1} follows from the first estimate in Proposition~\ref{est vor filament}.
Estimate \eqref{co2} follows directly from the definition \eqref{def M} and Proposition~\ref{est pot}.
Finally \eqref{co3} follows from the definition \eqref{DefGammaa}, the $L^2$-estimate in Proposition~\ref{est pot}
and Cauchy-Schwarz inequality.
\end{proof} 
Before estimating the last coefficients in \eqref{eqeps-WV}, we study the limit of $u^\varepsilon$.
\begin{prop} \label{kin fluid conv}
Suppose that $p=(\ell,\Omega)\in \R^6$ and $\mu\in \R$ are given and $\omega\in L^1(\R^3)$ is fixed, divergence-free,
compactly supported and such that $\dist(\supp\omega,\mathcal{C}_0)>0$. 
Let $u^\eps$ be the velocity $u$ in Lemma~\ref{vitdec} determined by this data.
Then as $\varepsilon \rightarrow 0^+$, $u^\varepsilon$ converges to
\begin{equation} \label{Eq:ustar}
u^* = K_{\R^3}[\omega] + K_{\R^3}[\kappa_{\mathcal{C}_0}],
\end{equation}
in $L^2(\mathcal{F}_0^\eps)$ and in $C_{loc}^\infty$ as defined in Definition \ref{def conv}. \par
Moreover, for sufficiently small $\eps$ we have the following quantitative bound, for all $m \in \N_{\geq 0}$,
\begin{equation} \label{kin decay u}
|\nabla^m (u^\eps - u^*)(x)| \lesssim (1+|p|) \eps |\log\eps|^2 (\dist(x,\mathcal{C}_0)-\eps)^{-3-m},
\end{equation}
where the implicit constant depends on $m$, $\norm{\omega}_{L^1}$ and $\dist(\supp\omega,\mathcal{C}_0)$. 
\end{prop}
\begin{proof}
We know from Lemma~\ref{vitdec} that
\begin{equation*}
u^\eps = \mu H^\eps + \sum_i p_i\nabla \Phi_i + K_{\mathcal{F}_0}[\omega]
\end{equation*}
and from \eqref{def ref} that $K_{\mathcal{F}_0}[\omega]=K_{\R^3}[\omega] + \uref[\omega]$.
Hence,
\begin{equation*}
u^\eps - u^* = \mu \bigl( H^\eps - K_{\R^3}[\kappa_{\mathcal{C}_0}] \bigr) 
  + \sum_i p_i \nabla \Phi_i + \uref[\omega].
\end{equation*}
The $L^2$-convergence then follows directly from the estimates in the Propositions~\ref{est pot},
\ref{est ref} and \ref{est vor filament}.
The estimate \eqref{kin decay u} and hence also the $C_{loc}^\infty$-convergence also follows 
from these estimates upon noticing that for all $x$ it holds that 
$\dist(x,\mathcal{C}_0) - \eps \leq \dist(x,\cals^\eps)$.
\end{proof}
Now we have the following estimates for the last coefficients in \eqref{eqeps-WV}.
\begin{prop} \label{co4}
For any fixed divergence-free and compactly supported $\omega\in L^1(\R^3)$ such that 
$\dist(\supp \omega,\mathcal{C}_0)>0$ we have
\begin{equation} \label{co5}
\bigl|{\mathcal B} [ K_{ \mathcal F_0} [\omega] ] \bigr| \lesssim |\log\eps|^\frac{1}{2} \eps \norm{\omega}_{L^1},
\end{equation}
where the implicit constant depends only on $\dist(\supp\omega,\mathcal{C}_0)$. \par
Furthermore, if $p=(\ell,\Omega)\in \R^6$ and $\mu\in \R$ are fixed, and $u^\eps$ is given as
the velocity $u$ in Lemma~\ref{vitdec} with this data, and $u^*$ is given as in \eqref{Eq:ustar}, we have
\begin{equation} \label{Eq:CvD}
\bigl| {\mathcal D}^\eps [ p, u^\eps, \omega ]-  {\mathcal D}^* [ u^*, \omega ] \bigr|
  \lesssim \eps |\log\eps|^2 \, (1+|p|) \norm{\omega}_{L^1},
 \end{equation}
where the implicit constant merely depends on $\dist(\supp\omega,\mathcal{C}_0)$ and on the size of the support of $\omega$.
\end{prop}
\begin{proof}
For the first estimate, we use that $K_{\mathcal{F}_0}[\omega]=K_{\R^3}[\omega]+\uref[\omega]$ 
and the linearity of $\mathcal{B}$.
Using the Biot-Savart law and the fact that $\mathcal{H}^2(\cals^\eps)\approx \eps$, it is easy to see that 
$| \mathcal{B}(K_{\R^3}[\omega]) | \lesssim \eps$ whenever $\eps \ll \dist(\supp\omega, \mathcal{C}_0)$.
Using the Cauchy-Schwarz inequality and the $L^2$-estimate in Proposition~\ref{est ref}, we also see that
$| \mathcal{B}[\uref[\omega]] | \lesssim \eps|\log\eps|^\frac{1}{2}$. 
This proves \eqref{co5}. \par
For the second estimate, we note that we have
\begin{equation*}
\bigl| \mathcal{D}[p,u^\eps,\omega]_i - \mathcal{D}^*[u^*, \omega]_i \bigr|
=\left| \int_{\mathcal{F}_0}[\zeta_i, \omega, u^\eps - u^*]\dx 
  - \int_{\mathcal{F}_0}[\omega, u^\eps - u_\mathcal{S}, \nabla\Phi_i] \dx \right|.
\end{equation*} 
We estimate both integrals separately and use the triangle inequality. 
The first integral is $\lesssim \norm{\omega}_{L^1}\norm{u^\eps-u^*}_{L^\infty(\supp\omega)}$, 
which goes to $0$ with the desired rate by the estimate \eqref{kin decay u} and because of the compact support of $\omega$.
Similarly, the second integral is
\begin{equation*}
 \lesssim \norm{\omega}_{L^1} (1 + \norm{u^\eps}_{L^\infty(\supp\omega)}) \norm{\nabla\Phi_i}_{L^\infty(\supp\omega)},
 \end{equation*}
which also goes to $0$ with the desired rate, since $\norm{u^\eps}_{L^\infty(\supp\omega)}$ is bounded 
by Proposition~\ref{kin fluid conv} and by the convergence of $\nabla \Phi_i$ in Proposition~\ref{est pot}.
\end{proof} 
%
%
%


\section{The irrotational case: Proof of Theorem \ref{pasdenom-VO}}
\label{sec-cv}
In this section, we prove Theorem~\ref{pasdenom-VO}. 
We first consider for any $\eps$ the solution $p^\eps$ to \eqref{eqeps}.
The total energy  $ \frac12 p^\eps(t) \cdot \big({\mathcal M}_{g} + {\mathcal M}_{a}^\eps \big) p^\eps(t)$
is constant in time when $p^\eps$ satisfies $\eqref{eqeps}$. 
Since the genuine inertia is bounded from below and the added inertia is non-negative, 
this proves that $(p^\eps)_\eps$ is bounded in $L^{\infty}([0,+\infty);\mathbb R^6)$. 

We observe that all the different terms in Equation~\eqref{eqeps} are polynomial (and in particular smooth) in $p^\eps$ and only depend on time through $p^\eps$.
Hence we see from rearranging the $k$-th derivative of equation \eqref{eqeps} as \begin{align*}
\frac{\mathrm{d}^{k+1}}{\mathrm{d}^{k+1}t} p^\eps=\frac{\mathrm{d}^{k}}{\mathrm{d}^{k}t}\Big((\mathcal{M}_g+\mathcal{M}_a^\eps)^{-1}\left(\mu B^\eps p^\eps - \langle\Gamma_g+\Gamma_a^\eps,p^\eps,p^\eps\rangle\right)\Big),
\end{align*}
and a straightforward induction argument that $p^\eps$ is in fact bounded in every $C^k([0,\infty), \R^6)$ and this bound is uniform in $\eps$ because the coefficients are bounded uniformly in $\eps$ by \eqref{co1}-\eqref{co3}.
In particular, by the Arzel\`a-Ascoli theorem, there exists a limit of a subsequence, which we denote by $p^*$.

 
We pass to the limit in \eqref{eqeps}. 
Since $p^\eps$ converges uniformly, using the bound \eqref{co3} on $\Gamma_a^\eps$ and because $\Gamma_g$ is fixed, we have that
\begin{equation*}
\langle \Gamma_g + \Gamma_a^\eps, p^\eps, p^\eps \rangle \rightarrow \langle \Gamma_g, p^*, p^* \rangle,
\end{equation*}
in $C^k$.
Concerning the term $B^\eps$, by \eqref{co1}, it holds that 
\begin{equation*}
\mu B^\eps p^\eps \rightarrow \mu B^\ast p^\ast,
\end{equation*} 
in $C^k$.
By the boundedness of ${p^\eps}^\prime$  
and \eqref{co2},  we have
\begin{equation*}
\mathcal{M}_a^\eps {p^\eps}^\prime \longrightarrow 0,
\end{equation*}
in $C^k$.
Finally, because $\mathcal{M}_g$ is constant and fixed, we have 
\begin{equation*}
\mathcal{M}_g{p^\eps}^\prime \rightarrow \mathcal{M}_g{p^\ast}^\prime
\end{equation*}
in $C^k$.
Hence, we obtain that  $p^* $ satisfies \eqref{eql1}.
It also satisfies $p^* (0) = p_0$.
By uniqueness of the Cauchy problem for \eqref{eql1}, we have that the whole family
$p^\eps$ converges to $p^*$.

Finally, it follows from \eqref{kin decay u} that as $\varepsilon \to 0$, the corresponding fluid velocity $u^\eps$,
given by \eqref{vdecomp-simpli-eps}, converges to $\mu H^*$ in $L^2(\mathcal{F}_0^\eps)$ 
and locally in $C_{loc}^\infty$ in the sense explained in Definition~\ref{def conv}. \hfill \qedsymbol
%
%

%
%

\section{Case with vorticity: Proof of Theorem~\ref{pasdenom-VO2}}\label{Sec:MainProof}
\subsection{An improved existence theorem}
To prove Theorem~\ref{pasdenom-VO2}, we will need an improved version of the Cauchy theory
for the system with positive radius $\varepsilon > 0$ 
(that is, \eqref{eqeps-WV}--\eqref{vdecompeps}), as well as for the limit system
(that is, \eqref{eql1-WV}-\eqref{u*}).
The main point is to obtain a uniform time of existence and uniform estimates for solutions,
with respect to $\varepsilon>0$. 
This will in particular prove Proposition~\ref{limit-WV}. \par
For that purpose, we fix $\varepsilon_0$ as 
\begin{equation*}
\varepsilon_0 := \frac{\dist(\supp(\omega_0), \mathcal{C}_0))}{10} >0,
\end{equation*}
and for $\varepsilon \in [0, \varepsilon_0]$, we put both systems in the form
\begin{gather}
\label{1}
\partial_t \omega + \big((u -  u_{\mathcal{S}}) \cdot\nabla\big)  \omega     
=  \big(\omega \cdot \nabla \big) (u - u_{\mathcal{S}})  \quad  \text{ in } \quad \R^3, \\ 
\label{3}
 {\mathcal M}_{g}  (p)' + \langle \Gamma_g, p , p \rangle
=  \mu B p  + {\mathcal D} \lbrack  u, \omega \rbrack , \\
\label{4}
 u_{\mathcal{S}} (t,x) := \ell  (t)  + \Omega (t) \wedge x 
 \quad  \text{ where } \quad p := ( \ell , \Omega ) \in  \R^3 \times  \R^3 ,
\end{gather}
and where the vector ${\mathcal D} \lbrack  u, \omega \rbrack$ in $\R^6$ is given 
by \eqref{defD} for $\varepsilon>0$ and by \eqref{defD*} for $\varepsilon =0$, and $B$ is given by the $B^*$ in \eqref{defBstar} if $\eps=0$ and by \eqref{def B} if $\eps>0$.
In the latter case, we will continue to write $\mathcal{D}$ as  
${\mathcal D} \lbrack p, u, \omega \rbrack$, even if it does not depend on $p$,
for the sake of notational uniformity.
Note that $u$ is slightly differently decomposed in the case $\varepsilon>0$ 
(see \eqref{vdecomp}) and in the case $\varepsilon=0$ (see \eqref{vdecomp*}).
To treat these decompositions in a uniform way, we write
\begin{equation} \label{Eq:UnifiedDecomp}
u =  \mu   H^\varepsilon  + K_{ \R^3} [\omega]  + \uref^\varepsilon[\omega] 
    + \sum_{i=1}^{6} p_i \nabla \Phi_i^\varepsilon,
\end{equation}
with the natural convention that for $\varepsilon=0$,
\begin{equation*}
H^0 =H^*,\, \uref^0=0 \ \text{ and } \ \Phi_i^{0}=0. 
\end{equation*}
We will use the following notation for $\Omega \subset \R^3$:
\begin{equation*}
{C}^{\lambda,r}_\sigma (\Omega) 
  := \Bigl\{ u \in {C}^{\lambda,r}(\Omega;\R^3) \ \Big/ \ \div(u) = 0 \Bigr\}. 
\end{equation*}
We define ${C}^{\lambda,r}_{\sigma,c}$ as the subspace of functions that are additionally compactly supported. When considering the weak$^*$-topology, we add a ``$-w^*$''.
The main statement is the following. 
\begin{thm} \label{start3}
Let $\lambda\in\N_{\geq 0}$ 
and $r \in (0,1)$ be given.
Consider $(p_0 ,\omega_0 )$ in  $\R^6 \times  {C}^{\lambda,r}_\sigma ( \R^3, \R^3)$ with 
$\omega_0$  compactly supported and moreover satisfying
$\supp \omega_0  \cap \mathcal{C}_0 = \emptyset $.
There exists a constant $\underline{c} > 0$ depending only on 
\begin{equation} \label{Eq:R0D0}
D_0:= \dist(\supp(\omega_0), \mathcal{C}_0) \text{ and }
R_0:= \max \left\{|x|, \ x \in \supp(\omega_0)\right\},
\end{equation} 
such that the following holds.
We set
\begin{equation} \label{Eq:EstTexistence}
\underline{T} := \underline{c} \, \frac{1}{1 + |p_0| + \|\omega_0\|_{C^{\lambda,r}}}.
\end{equation}
Then for all $\varepsilon \in [0,\varepsilon_0]$,
the problem \eqref{eqeps-WV}--\eqref{vdecompeps} for $\varepsilon>0$ 
or \eqref{eql1-WV}-\eqref{vdecomp*} for $\varepsilon=0$ admits a unique solution 
\begin{equation*}
(p, \omega) \in C ([0, \underline{T}];\, \R^6) \times 
C ( [0,\underline{T}], C^{\lambda,r} (\R^3)-w^* ),
\end{equation*}
which also enjoys the regularity
\begin{alignat*}{2}p &\in W^{\lambda,\infty} ([0,\underline{T}];\, \R^6),\\
\de_t^l\omega &\in L^\infty([0,\underline{T}];\,C^{\lambda-l,r}(\mathcal{F}_0^\varepsilon ))\quad &&\text{(for all $l\leq \lambda+1$),}\\ 
u &\in C ( [0,\underline{T}] ;\, C^{\lambda+1,r} (\mathcal{F}_0^\varepsilon )-w^*)\quad 
&&\text{(for $\varepsilon>0$),}\\
u -\mu H^* &\in C ( [0,\underline{T}] ;\, C^{\lambda+1,r} (\R^3 )-w^*)\quad&&\text{(for $\varepsilon=0$),}\\
\partial_t u &\in C ( [0,\underline{T}] ;\, C^{\lambda,r} (\mathcal{F}_0^\varepsilon )-w^*)
\quad&&\text{(resp.  $C ( [0,\underline{T}] ;\, C^{\lambda,r} (\R^3 )-w^*)$)}.
\end{alignat*}
%
Moreover, given $M>0$, there exists a $C= C(M, D_0, R_0)>0$ such that 
for all $\varepsilon \in [0,\varepsilon_0]$, if $(p_0,\omega_0)$ satisfies
\begin{equation}
\label{init M}
|p_0 | \leq M, \ \ 
\| \omega_0 \|_{C^{\lambda,r}} \leq M
\ \text{ and } \ 
\| \omega_0 \|_{L^1} \leq M, 
\end{equation}
the corresponding solution $(p^\varepsilon,\omega^\varepsilon)$ satisfies
for all $t \in [0,\underline{T}]$ and all $l\in \N_{\geq 0}\cap [0,\lambda+1]$:
\begin{equation} \label{Eq:UniformEstimates}
\left|\frac{\mathrm{d}^l}{\mathrm{d} t^l}p^\varepsilon(t)\right| \leq C, \ \ 
\| \de_t^l\omega^\varepsilon(t) \|_{C^{\lambda-l,r}} \leq C, \ \ 
\| \omega^\varepsilon(t) \|_{L^1} \leq C, \ \text{and} \ 
\dist(\supp \omega(t), \mathcal{C}_0) \geq \frac{1}{C}.
\end{equation}
\end{thm}
Here the negative Hölder space $C^{-1,r}$ which occurs when $l=\lambda+1$ can be defined e.g.\ in terms of the Littlewood-Paley decomposition, as $B_{\infty,\infty}^{r-1}$, see e.g.\ \cite[Sec.\ 2]{Chemin}.
Theorem~\ref{start3} 
is proved in the Appendix,
Section~\ref{Sec:WP-limit}.
We have the following corollary of Theorem~\ref{start3}, whose only slight novelty is 
to say that for $\omega_0 \in {C}^{\lambda,r}_\sigma ( \R^3)$, 
we may require a bound only on $\| \omega_0 \|_{C^{\lambda,r'}}$, (for $r'<r$), to obtain a 
bounded solution in $C^{\lambda,r}$ and still get a uniform existence time.
 \par
\begin{cor} \label{Cor:UniformExistence}
Let $\lambda \in \N_{\geq 0}$, $r \in (0,1)$ and $r' \in (0,r)$.
Given $D_0 >0$, $R_0 >0$, $M>0$, there exist  
$c=c(M, D_0, R_0)>0$ and $C=C(M, D_0, R_0)>0$ 
such that the following holds.
For all $\varepsilon \in [0,\varepsilon_0]$, 
if $(p_0,\omega_0) \in \R^6 \times C^{\lambda,r}_{\sigma,c}(\R^3;\R^3)$ satisfies
\begin{equation} \label{init M2}
|p_0 | \leq M, \ \ 
\| \omega_0 \|_{C^{\lambda,r'}} \leq M, \ \ 
\| \omega_0 \|_{L^1} \leq M, \ \ 
\ \text{ and } \ 
\dist(\supp \omega_0, \mathcal{C}_0) \geq \frac{1}{M} ,
\end{equation}
then, setting
\begin{equation} \nonumber 
\underline{T}(M) := c \, \frac{1}{1 + M},
\end{equation}
the corresponding solution $(p^\varepsilon,\omega^\varepsilon)$
defined during the time-interval $[0,\underline{T}]$, is in 
$C ([0,\underline{T}];\, \R^6) \times 
      L^\infty ( [0,\underline{T}], C^{\lambda,r} (\R^3))$
and satisfies for all $t \in [0,\underline{T}]$ and all $l\in \N_{\geq 0}\cap [0,\lambda+1]$:
\begin{equation} \label{Eq:UniformEstimates2}
\left|\frac{\mathrm{d}^l}{\mathrm{d} t^l}p^\varepsilon(t)\right| \leq C, \ \ 
\| \de_t^l\omega^\varepsilon(t) \|_{C^{\lambda-l,r}} \leq C\norm{\omega_0}_{C^{\lambda,r}}, \ \ 
\| \omega^\varepsilon(t) \|_{L^1} \leq C,  \ \text{and} \ 
\dist(\supp \omega(t), \mathcal{C}_0) \geq \frac{1}{C}.
\end{equation}
%
\end{cor}
The fact that we can merely require a bound on $\| \omega_i \|_{C^{\lambda,r'}}$ to obtain
 a bounded solution in $C^{\lambda,r}$ is classical (and weaker than Beale-Kato-Majda \cite{BKM1984}). 
It can be seen for instance, as a consequence of classical lemmas concerning H\"older flows and tame estimates (see, for instance Chemin, \cite{Chemin} Lemma 4.1.1 and Corollary 2.4.1).

Taking these results for granted, we can proceed to the proof of Theorem~\ref{pasdenom-VO2}.
\subsection{Proof of Theorem~\ref{pasdenom-VO2}}
Given the data of Theorem~\ref{pasdenom-VO2}, we associate the maximal solutions 
$(p^*, \omega^*)$ and  $(p^\varepsilon, \omega^\varepsilon)$ of the limit system 
and of the macroscopic system with thickness $\varepsilon$.
We denote by $T^*$ and $T^\varepsilon$ the maximal times of existence for these solutions. 

We then let $\mathcal{G}$ denote the set of all times for which uniform estimates as
in \eqref{Eq:UniformEstimates} hold for all sufficiently small $\eps$, more precisely:
\begin{multline*}
\mathcal{G} := \Big\{ T >0 \ \Big/ \ \exists \varepsilon_1>0, \ \exists M>0, \ 
\forall \varepsilon \in (0, \varepsilon_1), \  T^\varepsilon \geq T \text{ and } \forall t \in [0,T]:\\
|p^\varepsilon(t)| \leq M, \ 
\| \omega^\varepsilon(t)\|_{C^{\lambda,r}} \leq M, \ 
\| \omega^\varepsilon(t)\|_{L^1} \leq M, \ |p^\varepsilon(t)| \leq M, \ 
\dist(\supp \omega^\varepsilon(t), \mathcal{C}_0) \geq \frac{1}{M}
\Big\}
\end{multline*}
Now, Theorem~\ref{pasdenom-VO2} is a consequence of the three following lemmas.
\begin{lem} \label{Lem:MinCVTime}
The set $\mathcal{G}$ is not empty and connected. 
\end{lem}
Based on Lemma~\ref{Lem:MinCVTime}, we define
\begin{equation*}
\widehat{T} := \sup \mathcal{G}.
\end{equation*}
\begin{lem} \label{Lem:Convergence}
For all $T \in \mathcal{G}$ and $r'< r$, one has the convergences:
\begin{gather}
\label{Eq:CV1}
p^\varepsilon \longrightarrow p^* \ \text{ in } \ W^{\lambda,\infty}([0,T])-w^*, \\
\label{Eq:CV2}
\omega^\varepsilon \longrightarrow \omega^* \ 
\text{ in } \ L^\infty(0,T;C^{\lambda,r}(\R^3))-w^*
\ \text{ and in } \ C^l([0,T];C^{\lambda-l,r'}(\R^3)) \ \text{ for $l\in \N_{\leq \lambda}$}.
\end{gather}
\end{lem}
Here we use the convention that $\omega^\varepsilon$ is extended by $0$ in 
$\mathcal{S}_0^\varepsilon$ (which causes no regularity issues due to its support.) 
\begin{lem} \label{Lem:TT*}
One has $\liminf T^\eps \geq \widehat{T} \geq T^*$.
\end{lem}
We prove Lemmas~\ref{Lem:MinCVTime}--\ref{Lem:TT*} in order.
\begin{proof}[Proof of Lemma~\ref{Lem:MinCVTime}]
The connectedness of $\mathcal{G}$ is straightforward,
and its nonemptiness is a direct consequence of Theorem~\ref{start3}.
\end{proof}
\begin{proof}[Proof of Lemma~\ref{Lem:Convergence}]
For such a $T \in \mathcal{G}$, let us fix an $\eps_1>0$ and an $M>0$ as in the definition of $\mathcal{G}$.

We first observe that we have a uniform $W^{\lambda+1,\infty}([0,T])$-bound on $p^\eps$ 
for sufficiently small $\eps$ due to the statement about time regularities in Theorem \ref{start3}.

We further have a uniform bound on $\norm{\de_t\omega^\eps}_{L^\infty([0,T]; C^{\lambda-1,r})(\R^3)}$ and $\norm{\omega^\eps}_{L^\infty([0,T]; C^{\lambda,r}(\R^3)}$ by \eqref{Eq:UniformEstimates}.

Hence, we may apply the Aubin-Lions lemma to extract a subsequence of 
$(p^\eps,\omega^\eps)$ which converges to some $(p^*,\omega^*)$ 
in $W^{\lambda+1,\infty}([0,T])-w^*\times L^\infty([0,T],C^{\lambda,r'}(\R^3))$ for every $r'\in (0,r)$,
where the $w^*$ denotes the weak$^\ast$-topology.

We also have convergence of the velocities $u^\eps$ to the corresponding 
$u^*=\mu H^\eps+K_{\R^3}[\omega^*]$ in $L^\infty([0,T]; C^{\lambda+1,r'}(\mathfrak{B}_M))$, where 
\begin{align*}
\mathfrak{B}_M=B_M(0)\backslash\left\{x\,\big|\, \dist(x, \mathcal{C}_0)\leq \frac{1}{M}\right\}.
\end{align*}
Indeed, we know from Proposition \ref{kin fluid conv} that 
\begin{align*}
u^\eps - K_{\R^3}[\omega^\eps] - \mu K_{\R^3}[\kappa_{\mathcal{C}_0}] \longrightarrow 0
  \quad \text{ in $C^\infty(\mathfrak{B}_M)$},
\end{align*}
and $K_{\R^3}[\cdot]$ is bounded from $C^{k,\alpha}$ to $C^{k+1,\alpha}$ 
for all $k\in \N_{\geq 0}$.



%
Then one may start to pass to the limit in the Newton equations 
by using Proposition~\ref{co4} to obtain that 
\begin{equation*}
{\mathcal B} \big[ K_{ \mathcal F_0} [\omega^\varepsilon] \big] \longrightarrow 0 ,
\end{equation*}
in $L^{\infty}([0,T]; \R^{6 \times 6})$ and then Proposition~\ref{co} to deduce
\begin{equation*}
B \longrightarrow B^* .
\end{equation*}
Also relying on Proposition~\ref{co4}, we have
\begin{equation*}
{\mathcal D} [ p^\varepsilon, u^\varepsilon, \omega^\varepsilon ] 
  \longrightarrow {\mathcal D}^* [ u^*, \omega^* ]
\end{equation*}
in $L^{\infty}([0,T];\mathbb R^6)$.
Arguing similarly as in the proof of Theorem~\ref{pasdenom-VO} above (see Section~\ref{sec-cv}), we see that 
all the other coefficients in \eqref{eqeps-WV} converge at least weakly$^\ast$ 
to the corresponding terms in \eqref{eql1-WV}, 
in particular $p^*$ is a solution of \eqref{eql1-WV}.

Since $u_\mathcal{S}$ is linear in $p$, it also converges, and we may pass 
to the limit in \eqref{omega-eqeps}.

It follows from the uniqueness of the limiting solution that all convergent subsequences have the same limit and that, hence, convergence of the whole sequence holds.

Finally, we also have the convergence of the time derivatives by the Aubin-Lions Lemma and their estimates in \eqref{Eq:UniformEstimates}.\end{proof}
%
%
%
\begin{proof}[Proof of Lemma~\ref{Lem:TT*}]
The leftmost inequality trivially follows from the condition $T^\eps\geq T$ in 
the definition of $\mathcal{G}$.

For the right inequality, we argue by contradiction and suppose that $\widehat{T} < T^*$.
Then on $[0,\widehat{T}]$, the solution $(p^*,\omega^*)$ of the limit system is well-defined, 
and for some $\widehat{M}>0$, one has for all $t \in [0,\widehat{T}]$:
\begin{equation} \label{Eq:EstLimitSol}
|p^*(t)|  \leq \widehat{M}, \
\| \omega^*(t,\cdot)\|_{C^{\lambda,r'}} \leq \widehat{M}, \ 
\| \omega^*(t,\cdot)\|_{L^1} \leq \widehat{M}, \ 
\ \text{and} \ 
\dist(\supp \omega^*(t), \mathcal{C}_0) \geq \frac{1}{\widehat{M}} .
\end{equation}
Now, using the notation $\underline{T}(M)$ of Corollary~\ref{Cor:UniformExistence}, we let
\begin{equation*}
\eta := \frac{\min(\underline{T}(\widehat{M} + 1), \widehat{T})}{2} \ \text{ and } \ 
\widetilde{T}:= \widehat{T} - \eta.
\end{equation*}
Then, $\widetilde{T} \in \mathcal{G}$, and consequently we have the convergences 
\eqref{Eq:CV1} and \eqref{Eq:CV2} in $[0,\widetilde{T}]$.
Hence, for small enough $\varepsilon$ and $t\in[0,\widetilde{T}]$, we have
\begin{equation*}
|p^\varepsilon(t)|  \leq \widehat{M} + 1, \
\| \omega^\varepsilon(t,\cdot)\|_{C^{\lambda,r'}} \leq \widehat{M} + 1, \ 
\| \omega^\varepsilon(t,\cdot)\|_{L^1} \leq \widehat{M} + 1
\ \text{and} \ 
\dist(\supp \omega^\varepsilon(t), \mathcal{C}_0) \geq \frac{1}{\widehat{M} + 1} .
\end{equation*}
Using Corollary~\ref{Cor:UniformExistence}, we deduce that for such $\varepsilon$, 
the solutions $(p^\varepsilon,\omega^\varepsilon)$ are well-defined up to time 
$\widetilde{T} + \widehat{T}(\widehat{M}+1)$, with uniform estimates as in \eqref{Eq:UniformEstimates2}.
Hence, we have $\widehat{T} \geq \widetilde{T} + \widehat{T}(\widehat{M}+1) > \widehat{T} + \eta$,
which is a contradiction.
\end{proof}

\section{Divergence in the non-massive case: Proof of Theorem~\ref{thm:div}}
\label{Sec:Proof-div}



%
%

\subsection{First steps}

Given that $\mathcal{C}_0$ is a circle around the $e_3$-axis,
we consider axisymmetric data, i.e.\ initial data such that, 
\begin{equation*}
p_i^\eps(0)  =0 \quad \text{for } i \neq 3,
\end{equation*}
and such that $u_0$ satisfies that
for every rotation $A$ around the $e_3$-axis it holds that 
\begin{equation*}
u_0(A\cdot) = A u_0(\cdot).
\end{equation*}
We further pick the rotational inertia matrix $\mathcal{J}_0$ as an arbitrary rotational inertia matrix of an axisymmetric body, that is any positive symmetric definite matrix $\mathcal{J}_0$ such that it holds that $A^T\mathcal{J}_0A=\mathcal{J}_0$ for any rotation $A$ around the $e_3$-axis. The exact choice is irrelevant to the construction below since there are never any rotations of the body involved.

%
This axisymmetry of $u_0$ can be equivalently expressed as $\omega_0$ being of the form 
\begin{equation*}
\omega_0(x) = \omega_0^\theta \left( \sqrt{x_1^2+x_2^2}, \, x_3 \right)
\begin{pmatrix} - x_2 \\ x_1 \\ 0 \end{pmatrix}.
\end{equation*}
If the initial data has these symmetries, then the solution must also be axisymmetric at later times, as symmetry-breaking would imply non-uniqueness.

With such data we have in particular that $p_i^\eps(t) = 0$ for $i\neq 3$ for all $t$.
Consequently, the equation \eqref{eqeps-WV massless} simplifies to
\begin{equation} \label{redux eq}
(\tilde{\mathcal{M}}_g^\eps + \mathcal{M}_a)_{3,3} \, p_3^\eps
= \mathcal{D}[p^\eps, u, \omega]_3,
\end{equation}
because $\mathcal{B}$ is skew-symmetric and because by \eqref{DefGammag}, $\langle \Gamma_g, p^\eps, p^\eps \rangle = 0$ 
whenever $p_i^\eps=0$ for $i=4, 5, 6$. \par
\ \par
The main mechanism at stake here is that $(\tilde{\mathcal{M}}_g^\eps + \mathcal{M}_a)_{3,3} = O( \eps^2 |\log\eps|)$,
(by \eqref{co2} and the definition of the massless limit) 
while the term $\mathcal{D}$ on the right hand side is of order $1$ (at least if $p^\eps$ is not too big), 
yielding the divergence provided that we can control its sign. 

To arrange this, we will carry out most of the analysis under the assumption that the solid movement remains 
away from the vorticity, an assumption that will be checked to be true afterwards on some uniform time interval
as $\varepsilon \rightarrow 0^+$.
Correspondingly, we introduce for each $\varepsilon>0$ the solution $(p^\varepsilon,\omega^\varepsilon)$ of the system on the maximal interval $[0,T^*_\varepsilon)$, and we introduce the possibly smaller time 
$\widehat{T}_\varepsilon$ as follows:
\begin{equation} \label{Def:THat} 
\widehat{T}_\varepsilon := 
\sup \left\{ t \in [0,T^*_\varepsilon) \ \Big/ \ 
  \forall t' \in [0,t], \ \int_0^{t'} p_3^\eps(s) \ds > -1  \right\}.
\end{equation}
Obviously, one has $\widehat{T}_\varepsilon >0$ for each $\varepsilon>0$. \par
The main statement yielding Theorem~\ref{thm:div} is the following proposition.
\begin{prop} \label{key lemma}
There exists an axisymmetric initial datum $(\omega_0,p_0)$,
an axisymmetric circle $\mathcal{C}_0$, $T_0>0$, $c>0$, $\mu>0$ and $C_m>0$ for each $m \in \N_{\geq 0}$,
all independent of $\eps$, such that the following holds for all sufficiently small $\eps$. 
\begin{enumerate}
\item One has
\begin{equation} \label{wnice-local}
\norm{\omega^\eps}_{L^\infty([0,T_0] \cap [0,\widehat{T}_\varepsilon) ,C^m(\R^3))} \leq C_m
\ \text{and} \
\min_{t\in [0,T_0] \cap [0,\widehat{T}_\varepsilon)} \dist(\supp\omega_t^\eps,\cals^\eps) \geq c.
\end{equation}

\item For all $t \in [0,\widehat{T}_\varepsilon) \cap [0,T_0]$, if moreover 
\begin{equation} \label{Eq:PepsLess}
p^\eps_3(t) \leq \eps^{-\frac{1}{10}},
\end{equation}
then 
\begin{equation*}
\left( 2 +  \int_0^t p_3^\eps(s) \ds  \right)^{-4}
  \lesssim \mathcal{D}[p^\eps(t), u(t), \omega(t)]_3.
\end{equation*}
\end{enumerate}
\end{prop}
Theorem~\ref{thm:div} can be deduced from Proposition~\ref{key lemma} as follows.
\begin{proof}[Proof of Theorem~\ref{thm:div} from Proposition~\ref{key lemma}]
First, a direct consequence of \eqref{wnice-local} and of the Cauchy theory for the system is that
 $T^*_\varepsilon > \min(T_0,\widehat{T}_\varepsilon)$, 
 and the only possibility to have $\widehat{T}_\varepsilon < T_0$ is that the integral condition 
\begin{equation} \label{Eq:IntegralCondition}
\int_0^{t'} p_3^\eps(s) \ds > -1
\end{equation}
is violated at time $t'=\widehat{T}_\varepsilon$. \par
Now, from \eqref{Eq:LightScaling} and Proposition~\ref{est pot}, we see that
\begin{equation*}
(\tilde{\mathcal{M}}_g + \mathcal{M}_a)_{3,3} \lesssim \varepsilon^2 |\log \varepsilon|.
\end{equation*}
Relying on \eqref{redux eq}, we see that
\begin{equation*} 
{p_3^\eps}^\prime 
= \frac{ \mathcal{D}[p^\eps, u, \omega]_3 }{(\tilde{\mathcal{M}}_g + \mathcal{M}_a)_{3,3}}
\gtrsim \eps^{-2} |\log\eps|^{-1} \mathcal{D}[p^\eps, u, \omega]_3 .
\end{equation*}
Now we let
\begin{equation*}
\overline{T}_\varepsilon := 
\sup \left\{ t \in [0,\widehat{T}_\varepsilon) \cap [0,T_0] \ \Big/ \ 
   \eqref{Eq:PepsLess} \text{ holds true on } [0,t]  \right\}.
\end{equation*}
By mere continuity, $\overline{T}_\varepsilon>0$ for each $\varepsilon>0$.
Due to Proposition~\ref{key lemma}, for all $t \in [0,\overline{T}_\varepsilon)$ we have 
\begin{equation} \label{est acc}
{p_3^\eps}^\prime 
\gtrsim \eps^{-2} |\log\eps|^{-1} \left| 2+\int_0^t p_3^\eps(s) \ds \right|^{-4} > 0.
\end{equation}
%
In particular, $\overline{T}_\varepsilon < \widehat{T}_\varepsilon$.
Moreover, we see from \eqref{est acc} that in $[0,\overline{T}_\varepsilon)$,
$p_3^\eps$ grows at least like $\eps^{-2} |\log\eps|^{-1} t$.
It follows that for small $\varepsilon$, 
$\overline{T}_\varepsilon < T_0$ and that 
${p_3^\eps}(\overline{T}_\varepsilon) = \eps^{-\frac{1}{10}}$.
By continuity and by the estimate \eqref{est acc}, ${p_3^\eps}$ can never fall below $\eps^{-\frac{1}{10}}$ 
for times in $[\overline{T}_\varepsilon, T_0]$.
Consequently $\widehat{T}_\varepsilon \geq T_0$ since the integral condition \eqref{Eq:IntegralCondition}
cannot be violated in $[\overline{T}_\varepsilon, T_0]$ either. 
Hence, we have $T^*_\varepsilon > \widehat{T}_\varepsilon \geq T_0$.
This proves Theorem~\ref{thm:div} when Proposition~\ref{key lemma} is established.
\end{proof}
%
%
%

%
%

\subsection{Proof of Proposition \ref{key lemma}}

The rest of this section is devoted to the proof of Proposition~\ref{key lemma}.
%
We set 
\begin{align}
\label{data 1}
& \mathcal{C}_0 = \big\{ x \in \R^3 \, \big| \, x_1^2 + x_2^2 = 1 ; \ x_3=0 \big\}, \\
\label{data 2}
& p_0 = 0, \\
\label{data 3}
& \mu = 1, \\
\label{data 4}
& \omega_0 = \eta(x + s_0 e_3) \begin{pmatrix} x_2 \\ - x_1 \\ 0 \end{pmatrix},
\end{align}
where $s_0 \gg 1$ is some sufficiently large positive number 
and $\eta\neq 0$ is a smooth nonnegative function supported in $B_1(0)$ which only depends on $x_3$ and $x_1^2+x_2^2$. 
We further pick $\eta$ so that $\eta\leq 1$ pointwise and that it holds that 
\begin{equation} \label{ass eta}
\int_{\R^3} (x_1^2 + x_2^2)\eta(x) \dx \geq \frac{1}{100}.
\end{equation}
%
%
We choose the orientation of $\mathcal{C}_0$ so that at the point $(1,0,0)$ the tangent is $(0,1,0)$. \par
\ \par
The principle of the proof of Proposition~\ref{key lemma} is as follows.
On the one hand, as long as $\omega$ does not blow up and $p_3^\eps$ is not too big, we expect that over short times 
the main contribution to $\mathcal{D}$ is given by 
$\mathcal{D} \big[0, K_{\R^3}[\kappa_{\mathcal{C}_0}+\omega_0], \omega_0 \big]$,
up to the translation induced by the moving frame.
On the other hand we expect that $\mathcal{D} \big[0, K_{\R^3}[\kappa_{\mathcal{C}_0}+\omega_0], \omega_0 \big]$
fulfills a desired sign condition, with a quantitative lower bound.
These two parts, together with a control on the vorticity on uniform time-intervals,
establish the behavior described above. \par
\ \par
The above ideas correspond to the three following lemmas, which are
proved in the next subsections, and which involve Proposition~\ref{key lemma}.
In these statements, it will be convenient to drop the $\varepsilon$ exponents in $p$, $u$, $\omega$,
and to use the notation $\omega_t := \omega(t,\cdot)$ and
\begin{equation} \label{def s}
s(t) = s_0 + \int_0^t p_3^\eps(t') \dd t'.
\end{equation}
Note that, by the definition of $\widehat{T}_{\varepsilon}$, we have for all $t \in [0,\widehat{T}_{\varepsilon})$,
\begin{equation} \label{bd s}
s(t) \geq s_0-1 \gg 1.
\end{equation}
We recall that, in line with the statement of Proposition~\ref{key lemma}, the whole analysis will be performed in 
the time interval $[0,\widehat{T}_{\varepsilon})$. \par
The first lemma concerns the control of the vorticity.
\begin{lem} \label{no blow up}
For every $\tilde{\delta}>0$ there is a time $T_0>0$, independent of $\eps$, 
such that for all $\varepsilon>0$, the unique maximal solution of \eqref{eqeps-WV}--\eqref{vdecompeps}
satisfies for all $t \in \min(\widehat{T}_\varepsilon,T_0)$:
\begin{align}
\label{cont vort}
& \norm{ \omega_t(\cdot + s(t)e_3) - \omega_0(\cdot + s_0e_3)}_{L^2(\mathcal{F}_0)} \lesssim \tilde{\delta} ,\\
\label{confinement}
&\supp(\omega_t)\subset B_2(-s(t)e_3).
\end{align}
Furthermore, for all $m\in \N_{\geq 0}$ it holds that
\begin{equation} \label{vort smooth}
\norm{\omega_t}_{L^\infty([0,\min(\widehat{T}_\varepsilon,T_0)],C^m)} \lesssim_m 1 ,
\end{equation}
uniformly in $\eps$.
\end{lem}
For further reference, we note in particular that Lemma~\ref{no blow up} implies by \eqref{bd s}
that up to time $\min(\widehat{T}_\varepsilon,T_0)$ it holds that
\begin{equation} \label{dist bd}
\dist(\supp\omega_t, \cals^\eps) \geq s(t) - 2 - \eps \gtrsim s(t) \gg 1.
\end{equation}
The second lemma establishes the claim concerning the main term 
$\mathcal{D} \big[0, K_{\R^3}[\kappa_{\mathcal{C}_0}+\omega_0], \omega_0 \big]$.
\begin{lem} \label{main ord D}
For all $s>-1$ and sufficiently large $s_0$ it holds that
\begin{align}
\mathcal{D} \big[ 0, \, K_{\R^3}[\kappa_{\mathcal{C}_0} + \omega_0(\cdot + s e_3)], \, \omega_0(\cdot + s e_3) \big]
\geq \frac{1}{1000}(s+s_0)^{-4},
\end{align}
for sufficiently small $\eps$, where the smallness condition on $\eps$ is uniform in $s,s_0$.
\end{lem}
%
%
%
%
The third lemma establishes the claim that this term is indeed the main contribution to $\mathcal{D}$.
\begin{lem} \label{diff D}
Given $T_0$ as in Lemma~\ref{no blow up} for suitably small $\tilde{\delta}$, the following holds. \par
\noindent
If $t\leq \min(T_0,\widehat{T}_{\varepsilon})$ and $p_3^\eps(t) \leq \eps^{-\frac{1}{10}}$, then we have that 
\begin{equation*}
\Bigl| \mathcal{D}[p^\eps,u(t),\omega_t]_3 -
\mathcal{D} \bigl[ 0, K_{\R^3}[\kappa_{\mathcal{C}_0}+\omega_0(\cdot+(s(t)-s_0)e_3)], \omega_0(\cdot+(s(t)-s_0)e_3) \bigr]_3 \Bigr|
\leq \frac{1}{2000} s(t)^{-4}.
\end{equation*}
\end{lem}
Once Lemmas~\ref{no blow up}, \ref{main ord D} and \ref{diff D} are proved, 
the proof of Proposition~\ref{key lemma} is immediate.
\begin{proof}[Proof of Proposition~\ref{key lemma}]
We conclude the estimate on $\mathcal{D}$ in the proposition from Lemmata~\ref{main ord D} and \ref{diff D}, 
setting $s=s(t)-s_0$, which is $\geq -1$ by the definition of $\widehat{T}_{\varepsilon}$ and \eqref{def s}, 
while the statement about the non-blow-up of $\omega$ was proven in Lemma \ref{no blow up}.
\end{proof}

%
%
%
%
%

\subsection{Proof of Lemma \ref{no blow up}}

Our first estimate in $[0,\widehat{T}_{\varepsilon})$ is the following, 
and uses an additional assumption which will be proved later to hold true in $[0,\widehat{T}_{\varepsilon})$.
\begin{lem} \label{L2 est u}
For $t$ in $[0,\widehat{T}_{\varepsilon})$, if
\begin{equation} \label{aux ass}
\supp(\omega_t) \subset B_2(-s(t)e_3) \ \text{ and } \ \norm{ \omega_t }_{L^2} \leq 100,
\end{equation} 
then one has for some constant $C>0$ independent of $\eps$ and $t$,
\begin{align}
(\mathcal{M}_a^\eps)_{3,3} \, p_3^\eps(t)^2 \leq C.
\end{align}
\end{lem}
\begin{proof}[Proof of Lemma~\ref{L2 est u}]
Notice that when \eqref{bd s} and \eqref{confinement} are both satisfied, we have in  particular
\begin{equation*}
\mbox{dist}(\supp \omega_t, \mathcal{S}^{\varepsilon}_0) \geq 1.
\end{equation*}
Now, splitting the fluid velocity as in Lemma~\ref{vitdec} as
\begin{equation*}
u= H^\eps + p_3^\eps \nabla \Phi_3 + K_{\mathcal{F}_0}[\omega],
\end{equation*}
and inserting the definition of $\mathcal{M}_a^\eps$, we can rewrite the fluid energy 
$\frac{1}{2}\norm{u}_{L^2(\mathcal{F}_0)}^2$ as
\begin{align*}
\frac{1}{2}\Bigl(
& \norm{H^\eps}^2_{L^2(\mathcal{F}_0)} + (\mathcal{M}_a^\eps)_{3,3} \, p_3(t)^2
 + \norm{K_{\mathcal{F}_0}[\omega_t]}_{L^2(\mathcal{F}_0)}^2 + 2 p_3^\eps(t) \scalar{H^\eps}{\nabla\Phi_3}
 + 2 \scalar{H^\eps}{K_{\mathcal{F}_0}[\omega_t]} \\
& + 2p_3^\eps(t) \scalar{\nabla\Phi_3}{K_{\mathcal{F}_0}[\omega_t]} \Bigr).
\end{align*}
By Subsection~\ref{subsec-en}, the total energy 
$\frac{1}{2} (\tilde{\mathcal{M}}_g^\eps)_{3,3} \, p_3^\eps(t)^2 + \frac{1}{2}\norm{u}_{L^2(\mathcal{F}_0)}^2$
is conserved over time.
Now, note that $\norm{H^\eps}_{L^2(\mathcal{F}_0)}^2$ is independent of $t$ 
and hence the energy without this term is conserved too. 
Furthermore, as $\div H^\eps=0$ we always have
\begin{equation} \label{part int}
\int_{\mathcal{F}_0} \nabla \Phi_3 \cdot H^\eps \dx = \int_{\de\cals^\eps} n \cdot H^\eps \Phi_3 \dd\sigma = 0,
\end{equation}
where this partial integration is justified by the decay estimates in Proposition~
\ref{est pot} and Corollary~\ref{Cor:DecayH}.
Therefore, 
\begin{align*}
\tilde{\mathcal{E}}(t):=\frac{1}{2}
&\left((\mathcal{M}_a^\eps)_{3,3} \, p_3^\eps(t)^2 + \norm{K_{\mathcal{F}_0}[\omega_t]}_{L^2(\mathcal{F}_0)}^2
  + 2\scalar{H^\eps}{K_{\mathcal{F}_0}[\omega_t]} + 2p_3^\eps(t) \scalar{\nabla\Phi_3}{K_{\mathcal{F}_0}[\omega_t]} \right) \\
& + \frac{1}{2}(\tilde{\mathcal{M}}_g^\eps)_{3,3}p_3^\eps(t)^2,
\end{align*}
is a conserved quantity. We first check that 
\begin{equation*}
\tilde{\mathcal{E}}(0) \lesssim 1.
\end{equation*}
As $p_3^\eps(0)=0$ all terms involving $p_3^\eps$ are $0$.
Concerning the others, using \eqref{def ref}, we have 
\begin{equation} \label{in vort e}
\norm{ K_{\mathcal{F}_0}[\omega_0] }_{L^2(\mathcal{F}_0)}^2
  \lesssim \norm{ K_{\R^3}[\omega_0] }_{L^2(\mathcal{F}_0)}^2
  + \norm{ \uref[\omega_0] }_{L^2(\mathcal{F}_0)}^2 \lesssim 1,
\end{equation}
where the estimate on $K_{\R^3}[\omega_0]$ is an easy consequence of $\omega_0$ being smooth and compactly supported,
while the estimate on $\uref$ directly follows from Proposition~\ref{est ref} with the fact that by definition 
we have $\dist(\supp\omega_0,\cals^\eps)\geq 1$. \par
In the other scalar product, we can split $K_{\mathcal{F}_0}[\omega_0] = \uref[\omega_0] + K_{\R^3}[\omega_0]$. 
\begin{equation}\begin{aligned} \label{Eq:EstKvsK}
| \scalar{ K_{\mathcal{F}_0}[\omega_0] }{ H^\eps }| 
\leq& |\scalar{\uref[\omega_0]}{H^\eps}|+
\norm{ H^{\eps} - K_{\R^3}[ \kappa_{\mathcal{C}_0} ] }_{L^2(\mathcal{F}_0)} 
    \norm{K_{\R^3}[\omega_0] }_{L^2(\mathcal{F}_0) }\\
&+ | \scalar{ K_{\R^3}[\omega_0] }{ K_{\R^3}[ \kappa_{\mathcal{C}_0} ] }|.
\end{aligned}\end{equation}
As $\uref$ is the gradient of a potential, we can partially integrate as in \eqref{part int} above to see that 
\begin{equation*}
\int_{\mathcal{F}_0} \uref[\omega_0] \cdot K_{\R^3}[ \kappa_{\mathcal{C}_0} ] \dx = 0.
\end{equation*}
The second term is $\lesssim 1$ by \eqref{in vort e} and the $L^2$-estimate in Proposition \ref{est vor filament}.

Recalling that 
$K_{\R^3}[ \kappa_{\mathcal{C}_0} ] = \curl \left(\frac{1}{4\pi|\cdot|} * \kappa_{\mathcal{C}_0}\right)$, 
we can partially integrate the last term in \eqref{Eq:EstKvsK} to see that
\begin{align*}
& \int_{\mathcal{F}_0} K_{\R^3}[\omega_0] \cdot K_{\R^3}[\kappa_{\mathcal{C}_0}] \dx \\
& = - \int_{\mathcal{F}_0} \curl K_{\R^3}[\omega_0] \cdot \left( \frac{1}{4\pi|\cdot|} * \kappa_{\mathcal{C}_0} \right) \dx
+ \int_{\de \mathcal{F}_0} K_{\R^3}[\omega_0] \cdot
    \left( n \wedge \left( \frac{1}{4\pi|\cdot|} * \kappa_{\mathcal{C}_0} \right) \right) \dd\sigma.
\end{align*}
The first integral is $\lesssim 1$ because $\curl K_{\R^3}[\omega_0]=\omega_0$ and because of the natural decay estimate
$\left| \frac{1}{4\pi|\cdot|} * \kappa_{\mathcal{C}_0} (x)\right| \lesssim \dist(x, \mathcal{C}_0)^{-1}$.
In the second integral, we note that $\frac{1}{4\pi|\cdot|} * \kappa_{\mathcal{C}_0} \lesssim \eps^{-1}$ on $\de\mathcal{F}_0$ 
because $\dist(\mathcal{C}_0,\de\mathcal{F}_0)=\eps$ by definition and obtain that 
\begin{equation*}
\left| \int_{\de \mathcal{F}_0}K_{\R^3}[\omega_0] \cdot 
  \left( n \wedge \left( \frac{1}{4\pi|\cdot|} * \kappa_{\mathcal{C}_0} \right) \right)\dd\sigma \right|
\lesssim \eps^{-1} \mathcal{H}^2(\de\mathcal{F}_0) \norm{ K_{\R^3}[ \omega_0 ]}_{L^\infty(\de\mathcal{F}_0)}
\lesssim 1,
\end{equation*}
where the last estimate uses the explicit form of the Biot-Savart law and the assumption that 
$\dist(\supp \omega_0,\mathcal{S}_0) \geq 1$.
Together, we have obtained that
\begin{equation} \label{prod 1}
| \scalar{ K_{\mathcal{F}_0}[\omega_0] }{ H^\eps } | \lesssim 1.
\end{equation}

%

%
 
\noindent Now, we can estimate
\begin{equation}\begin{aligned}  
( \tilde{\mathcal{M}}_g + \mathcal{M}_a )_{3,3} \, (p_3^\eps(t))^2 
\, & \, \leq \tilde{\mathcal{E}}(t) + |p_3^\eps(t)| \, | \scalar{\nabla\Phi_3}{K_{\mathcal{F}_0}[\omega_t]} |
  + | \scalar{H^\eps}{K_{\mathcal{F}_0}[\omega_t]} | \\
\label{est mod en} 
\, & \, \lesssim  1 +|p_3^\eps(t)| \,| \scalar{\nabla\Phi_3}{K_{\mathcal{F}_0}[\omega_t]} |
  + | \scalar{H^\eps}{K_{\mathcal{F}_0}[\omega_t]} |.
\end{aligned}\end{equation}
One can show that 
\begin{equation}
\norm{K_{\mathcal{F}_0}[\omega_t]}_{L^2(\mathcal{F}_0)}\lesssim 1\label{bd kft}
\end{equation}
by the same argument as above in \eqref{in vort e},
since $\norm{\omega_t}_{L^2}$ is controlled by the Assumption \eqref{aux ass} and 
\begin{equation*}
\dist(\supp\omega_t,\cals^\eps) \geq s(t) - 2 - \eps \gg 1,
\end{equation*}
which is due to \eqref{bd s} and again Assumption\eqref{aux ass}.

One can also show that 
\begin{equation*}
| \scalar{ H^\eps }{ K_{\mathcal{F}_0}[\omega(t)] } | \lesssim 1
\end{equation*}
with the same argument as above in \eqref{prod 1}, again using that $\dist(\supp\omega_t,\cals^\eps) \gg 1$.

Using \eqref{bd kft} and the definition of $\mathcal{M}_a$ (see \eqref{def M}), we then have for every $\alpha>0$ that
\begin{align*}
|p_3^\eps(t)| \, | \scalar{ \nabla\Phi_3 }{ K_{\mathcal{F}_0}[\omega_t] }|
\lesssim |p_3^\eps(t)|\norm{\nabla\Phi_3}_{L^2(\mathcal{F}_0)}
= \sqrt{(\mathcal{M}_a^\eps)_{3,3} \, (p_3^\eps)^2}
\leq \frac{1}{\alpha} + \alpha(\mathcal{M}_a^\eps)_{3,3} \, (p_3^\eps(t))^2.
\end{align*}
For small enough $\alpha$ (not depending on the other parameters), we can absorb the second term back into the left hand side. Finally we obtain from \eqref{est mod en} that
\begin{equation*}
( \tilde{\mathcal{M}}_g^\eps + \mathcal{M}_a^\eps )_{3,3} \, (p_3^\eps(t))^2 \lesssim 1,
\end{equation*}
which establishes Lemma~\ref{L2 est u} since $\tilde{\mathcal{M}}_g^\eps$ is positive definite.
\end{proof}
Our next estimate in $[0,\widehat{T}_{\varepsilon})$ and under Assumption~\eqref{aux ass} is the following.
\begin{lem} \label{vel est}
For all $m\in\N_{\geq 0}$, there exists some $C_m>0$, 
such that for all $\varepsilon>0$ and for all $t \in [0,\widehat{T}_{\varepsilon})$
for which \eqref{aux ass} is valid, one has
\begin{equation*}
\norm{ u(t) - K_{\R^3}[\omega_t] }_{C^m(B_2(-s(t)e_3))} \leq C_m.
\end{equation*}
\end{lem}
\begin{proof}[Proof of Lemma~\ref{vel est}]
We again note that $\dist(\supp\omega_t,\cals^\eps) \gg 1$. 
We may split
\begin{equation*}
u(t) - K_{\R^3}[\omega_t] = H^\eps + p_3^\eps(t) \nabla \Phi_3 + \uref[\omega_t].
\end{equation*}
By \eqref{decay H}, the $C^m(\supp\omega_t)$-norm of $H^\eps$ is $\lesssim 1$ and by the Assumption \eqref{aux ass}, 
and the decay estimate in Proposition \ref{est ref}, the same holds for $\uref$.

Next we note that by Lemma \ref{L2 est u} above and the Definition \eqref{def M}, we have 
\begin{equation} \label{l2 bd}
\norm{ p_3^\eps \nabla\Phi_3 }_{L^2(\mathcal{F}_0^\eps)}
  = \sqrt{ (\mathcal{M}_a^\eps)_{3,3} \, (p_3^\eps)^2 } \lesssim 1.
\end{equation}
We may then use that $\Phi_3$ is harmonic to estimate 
\begin{align}
\norm{ p_3^\eps \nabla \Phi_3}_{C^m(B_2(-s(t)e_3))}
  \lesssim_m \norm{ p_3^\eps \nabla \Phi_3 }_{L^2(B_3(-s(t)e_3))} \lesssim 1,
\end{align}
were we used \eqref{l2 bd} and elliptic regularity (see e.g.\ \cite[Thm.\ 8.10]{GilbargTrudinger}).
\end{proof}
We are now in a position to prove Lemma~\ref{no blow up}, obtaining in particular that 
Assumption~\eqref{aux ass} is satisfied in $[0,\widehat{T}_{\varepsilon})$.
\begin{proof}[Proof of Lemma~\ref{no blow up}]
We recast the vorticity equation \eqref{omega-eqeps} in the frame centered at $-s(t)e_3$, where it reads as
\begin{equation} \label{vort tilde}
\de_t \tilde{\omega} + (\tilde{u} \cdot \nabla) \tilde{\omega} = (\tilde{\omega} \cdot \nabla) \tilde{u},
\end{equation}
with $\tilde{\omega} = \omega(\cdot+s(t)e_3)$ and $\tilde{u} = u(\cdot+s(t)e_3)$,
where we used that the point $-s(t)e_3$ moves with speed $-u_\mathcal{S}$.

By testing with $\Delta^m\tilde{\omega}$ and using some standard commutator estimates, we have that 
\begin{multline} \label{gronwall}
\frac{\mathrm{d}}{\mathrm{d}t}\norm{\nabla^m\tilde{\omega}}_{L^2}^2 
\lesssim_m ( \norm{ \tilde{u} }_{W^{3,\infty}(\supp\tilde{\omega}_t)}
  + \norm{ \tilde{u} }_{H^m(\supp\tilde{\omega}_t)} ) \norm{ \tilde{\omega} }_{H^m}^2 \\
+ ( \norm{ \tilde{u} }_{W^{3,\infty}(\supp\tilde{\omega}_t)}
  + \norm{ \tilde{u} }_{H^{m+1}(\supp\tilde{\omega}_t)} ) \norm{ \tilde{\omega} }_{H^m}^2.
\end{multline}
Observe that the translations do not change the norms and we may hence drop them all.
As long as \eqref{aux ass} holds and $m\geq 4$, we have that 
\begin{align*}
\MoveEqLeft \norm{ u(t) }_{W^{3,\infty}(\supp\omega_t)} + \norm{ u(t) }_{H^{m+1}(\supp\omega_t)}
  \lesssim_m\norm{ u(t) }_{H^{m+1}(\supp\omega_t)} \\
& \lesssim_m \norm{ K_{\R^3}[\omega_t] }_{H^{m+1}(\supp\omega_t)}
  + \norm{ u(t) - K_{\R^3}[\omega_t] }_{H^{m+1}(\supp\omega_t)} \\
& \lesssim_m \norm{ \omega_t }_{H^m}+1,
\end{align*}
where we used Lemma \ref{vel est} in the last step, as well as Sobolev 
and the $H^m \rightarrow H^{m+1}$-boundedness of the Biot-Savart law in $\R^3$.

Inserting this into \eqref{gronwall}, we obtain that as long as \eqref{aux ass} holds
\begin{equation*}
\frac{\mathrm{d}}{\mathrm{d}t} \norm{ \nabla^m \tilde{\omega} }_{L^2}^2
  \lesssim_m \norm{  \tilde{\omega} }_{H^m}^2 + \norm{\tilde{\omega} }_{H^m}^3,
\end{equation*}
for $m\geq 4$. 
This shows that the strong solution of \eqref{omega-eqeps} persists for a time uniformly bounded from below 
and also that \eqref{aux ass} holds for a time uniformly bounded from below. 
Standard estimates also show that if \eqref{aux ass} and if $\norm{\omega_t}_{H^m}\lesssim 1$ both hold up 
to some time $T$ for $m\geq 5$, then it holds that
\begin{equation*}
\frac{\mathrm{d}}{\mathrm{d}t} \norm{ \omega }_{H^{m+1}}^2 \lesssim_m \norm{ \omega }_{H^{m+1}}^2,
\end{equation*}
up to the same time $T$, showing that in that case $\omega_t$ can not blow up in $H^{m+1}$ before time $T$, 
which shows the existence time of the solution in $C^m$ can be chosen independently of $m$. 

Finally $\norm{\de_t\tilde\omega}_{L^2}$ is uniformly bounded over short times, 
which implies that \eqref{cont vort} holds if we pick $T_0$ sufficiently small.
\end{proof}
%
%
%
%
%

\subsection{Proof of Lemma \ref{main ord D}}
To obtain the estimate on $\mathcal{D}$ of Lemma~\ref{main ord D}, we will need some auxiliary estimates on the potentials. 
\begin{lem} \label{lem87}
For all $x$ such that $|x_3|>1$ it holds that 
\begin{align}
\label{4th ord 1}
|\de_1\Phi_3(x)| + |\de_2\Phi_3(x)| 
  \lesssim\, & \eps|\log\eps|^\frac{1}{2} \sqrt{x_1^2+x_2^2}|x_3|^{-4}, \\
\label{4th ord 2}
|(H^\eps - K_{\R^3}[\kappa_{\mathcal{C}_0}])_1| + |(H^\eps - K_{\R^3}[\kappa_{\mathcal{C}_0}])_2|
  \lesssim\, & \eps|\log\eps|^\frac{1}{2}\sqrt{x_1^2+x_2^2}|x_3|^{-4}, \\
\label{4th ord 3}
|\uref[\omega_t]_1(x)| + |\uref[\omega_t]_2(x)|
  \lesssim\, & \eps|\log\eps|^\frac{1}{2} \sqrt{x_1^2+x_2^2}|x_3|^{-4},
\end{align}
where the last estimate is uniform in $t$ up to time $T_0$.
\end{lem}
\begin{proof}[Proof of Lemma~\ref{lem87}]
Observe that all these functions are the gradient of a potential which is axisymmetric and that hence the $e_1$- and $e_2$-components of these gradients must vanish on the $e_3$-axis.

By the axisymmetry, it is not restrictive to assume $x_2=0$ and $x_1>0$, 
in which case, the $e_2$-component of all three functions must vanish. 
We then use the fundamental theorem of calculus along the line segment  $\{(s,0,x_3)\,|\, s\in [0,x_1]\}$
(which does not intersect $\cals^\eps$ because $|x_3|>1$), yielding 
\begin{align*}
|\de_1 \Phi_3(x)| & \leq \int_0^{x_1} |\nabla^2 \Phi_3(s,0,x_3)| \ds \\
&\lesssim x_1 \eps|\log\eps|^\frac{1}{2} \sup_{s\in [0,x_1]} \dist((s,0,x_3),\cals^\eps)^{-4}
\lesssim x_1 \eps|\log\eps|^\frac{1}{2} |x_3|^{-4} ,
\end{align*}
where we used that $\dist((s,0,x_3),\cals^\eps) \geq |x_3| - \eps \gtrsim |x_3| \geq 1$ along the line segment 
and used the decay estimate in Proposition \ref{est pot}.

The estimates \eqref{4th ord 2} and \eqref{4th ord 3} follow by the same argument, 
using the decay estimates in Proposition \ref{est ref} and \ref{est vor filament} instead
(along with the bounds on $\omega_t$ above).
\end{proof}

We can now prove Lemma~\ref{main ord D} and conclude the bound on $\mathcal{D}$.
\begin{proof}[Proof of Lemma \ref{main ord D}]
We insert the Definition \eqref{defD} of $\mathcal{D}$ and see that 
\begin{multline*}
\mathcal{D} \big[ 0, K_{\R^3}[\kappa_{\mathcal{C}_0} + \omega_0(\cdot+se_3)], \omega_0(\cdot+se_3) \big]_3 \\
= \int_{\mathcal{F}_0} \bigl[ e_3, \omega_0(\cdot+se_3), K_{\R^3}[\kappa_{\mathcal{C}_0} + \omega_0(\cdot+se_3)] \bigr]
  - \bigl[ \omega_0(\cdot+se_3), K_{\R^3}[\kappa_{\mathcal{C}_0} + \omega_0(\cdot+se_3)], \nabla \Phi_3 \bigr] \dd x.
\end{multline*}
We first analyze the first bracket on the right-hand side. 
Separating $K_{\R^3}[\kappa_{\mathcal{C}_0} + \omega_0(\cdot+se_3)]$ as 
$K_{\R^3}[\omega_0(\cdot+se_3)] + K_{\R^3}[\kappa_{\mathcal{C}_0} ]$, a direct computation gives
\begin{multline} \label{canc}
\int_{\mathcal{F}_0} [e_3, \omega_0(\cdot+se_3), K_{\R^3}[\omega_0(\cdot+se_3)] \dx 
= \int_{\R^3} [e_3, \omega_0, K_{\R^3}[\omega_0]] \dx \\
= \int_{\R^3} e_3 \cdot (\curl K_{\R^3}[\omega_0] \wedge K_{\R^3}[\omega_0])\dx = 0,
\end{multline}
where the last step follows from the identity $\int_{\R^3} v\wedge \curl v\dx=0$, which holds for every $v\in W^{1,2}$ 
(see e.g.\ \eqref{etlesigne}).

The other term in the first bracket is the main contribution.
It is well known that in the axisymmetric setting $K_{\R^3}[\kappa_{\mathcal{C}_0}]$ can be expressed in terms 
of elliptic integrals and that one can compute its asymptotics, see e.g.\ \cite[Lemma 2.1]{GallaySverak}.
Indeed if we set 
$$e_r=\frac{1}{\sqrt{x_1^2+x_2^2}}(x_1,x_2,0),$$
then for $e_3\rightarrow -\infty$ 
and $\sqrt{x_1^2+x_2^2}\leq 10$ it holds that 
\begin{equation} \label{asy K}
K_{\R^3}[\kappa_{\mathcal{C}_0}](x) \cdot e_r
 = \frac{ -3\sqrt{x_1^2+x_2^2} x_3 }{4|x_3|^5} + O\left( \frac{x_1^2+x_2^2}{x_3^6} \right).
\end{equation}
Observe that only the $e_r$-component of $K_{\R^3}[\kappa_{\mathcal{C}_0}]$ contributes to the bracket in the integral, 
since it has no azimuthal component by axisymmetry and the $e_3$-component does not contribute 
because of the definition of the bracket $[e_3, \cdot, \cdot]$. Also note that it holds that 
\begin{equation} \label{sign bracket}
\left[ e_3, \begin{pmatrix} x_2 \\ -x_1 \\ 0 \end{pmatrix}, e_r \right] = \sqrt{x_1^2+x_2^2},
\end{equation}
by direct calculation.
We then have 
\begin{equation*}
\int_{\mathcal{F}_0} \big[ e_3, \omega_0(\cdot +se_3), K_{\R^3}[\kappa_{\mathcal{C}_0}] \big] \dx
= \int_{\mathcal{F}_0} \frac{-3(x_1^2+x_2^2)x_3}{4|x_3|^5} \eta(x+(s_0+s)e_3) \dx + O((s+s_0)^{-6}),
\end{equation*}
where we have used \eqref{asy K} and \eqref{sign bracket} and furthermore used that $\omega_0(\cdot+se_3)$ is supported in $B_1(-(s_0+s)e_3)$ 
to estimate the error term coming from the error in \eqref{asy K}. 
By making $s_0$ sufficiently large, we can make the error term here much smaller than $(s_0+s)^{-4}$. We have 
\begin{multline*}
\int_{\mathcal{F}_0} \frac{ -3(x_1^2+x_2^2)x_3 }{ 4|x_3|^5 } \eta(x+(s_0+s)e_3) \dx
= \int_{\R^3} \frac{ -3(x_1^2+x_2^2)(x_3-(s+s_0)) }{ 4|x_3-(s_0+s)|^5 } \eta(x) \dx \\
= (s+s_0)^{-4}\int_{\R^3} \frac{3(x_1^2+x_2^2)}{4} \eta(x) \dx + O \left( (s+s_0)^{-5} \right).
\end{multline*}
The last integral no longer depends on $s,s_0$ or $\eps$ and is $\geq \frac{3}{400}$ by \eqref{ass eta}.

For the second bracket, we have 
\begin{equation} \label{2nd bra 1}
\left| [ \omega_0(\cdot+se_3), K_{\R^3}[\kappa_{\mathcal{C}_0}], \nabla \Phi_3 ] \right|
\lesssim \eps|\log\eps|^\frac{1}{2} \dist(\supp \omega_0(\cdot+se_3),\cals^\eps)^{-5}
\lesssim \eps|\log\eps|^\frac{1}{2} (s_0+s)^{-5}
\end{equation}
by \eqref{decay K} and the decay estimate in Proposition \ref{est pot}.

For the second part of the second summand, we use that $[a,b,c]=[b,c,a]$ and use \eqref{canc} again to see that 
\begin{equation} \begin{aligned}
& \int_{\mathcal{F}_0} [ \omega_0(\cdot+se_3), K_{\R^3}[\omega_0(\cdot+se_3)], \nabla \Phi_i ] \dx \\
& = \inf_{a\in \R} \int_{\mathcal{F}_0} [ \nabla \Phi_3 - a e_3, \omega_0(\cdot+se_3), K_{\R^3}[\omega_0(\cdot+se_3)] ] \dx \\
& \lesssim  \inf_{a\in \R} \norm{ \nabla \Phi_3 - a e_3}_{L^\infty(\supp\omega_0(\cdot+se_3))}
  \norm{ K_{\R^{3}}[\omega_0(\cdot+s_3)] }_{L^2(\supp\omega_0(\cdot+se_3))} \norm{\omega_0}_{L^2} \\
& \lesssim \norm{ \nabla_{1,2} \Phi_3 }_{L^\infty(\supp\omega_0(\cdot+se_3))}
  + \norm{ \nabla^2 \Phi_3 }_{L^\infty(\supp\omega_0(\cdot+se_3))} \\
\label{2nd bra est}
& \lesssim \eps|\log\eps|^\frac{1}{2} (s+s_0)^{-4},
\end{aligned} \end{equation}
where in the penultimate step we used that $\supp\omega_0(\cdot+se_3))$ is a ball of radius $1$ 
to estimate $\inf \de_3\Phi_3-ae_3$ with $\nabla^2\Phi_3$ and in the last step we used the decay estimate 
in Proposition~\ref{est pot} and \eqref{4th ord 1} in the last step.

Putting everything together, we obtain Lemma~\ref{main ord D}.
\end{proof}
%
%
%

%
%

\subsection{Proof of Lemma~\ref{diff D}}

\begin{proof}[Proof of Lemma~\ref{diff D}]
We estimate the different contributions directly, using the definition of $\mathcal{D}$ (see \eqref{defD}).
We first note that
\begin{equation*}
\Bigl| \mathcal{D}[p^\eps,u(t),\omega_t]_3 - \mathcal{D}[0,u(t),\omega_t]_3 \Bigr|
= \left| \int_{\mathcal{F}_0} [\omega_t,u_\mathcal{S},\nabla\Phi_3] \dx \right|
.\end{equation*}
The contribution of $\de_3\Phi_3$ to this is $0$, since $u_\mathcal{S}=p_3^\eps e_3$. 
Therefore, we may use the decay estimate \eqref{4th ord 1} 
and the \textit{a priori} estimates \eqref{vort smooth} and \eqref{dist bd} on $\omega_t$ to estimate this by 
\begin{equation*}
\lesssim |p_3^\eps| \norm{\omega_t}_{L^1} \eps|\log\eps|^\frac{1}{2} \dist(\supp\omega_t,\cals^\eps)^{-4}
\lesssim \eps^{\frac{8}{10}} s(t)^{-4},
\end{equation*}
where we used the assumption $|p_3^\eps(t)|\leq \eps^{-\frac{1}{10}}$.

We now focus on $\mathcal{D}[0, u(t), \omega_t]_3$. From \eqref{defD}, we have
\begin{equation} \label{Eq:SecondTermD}
{\mathcal D} [ 0, u(t), \omega_t ]_3
:=  \int_{\mathcal  F_0}  [ e_{3}, \omega_t, u(t) ]  \ dx 
  - \int_{ \mathcal{F}_0}   [ \omega_t, u(t), \nabla \Phi_{3} ] \dx.
\end{equation}
We first estimate the second integral in \eqref{Eq:SecondTermD}. 
It is true that
\begin{multline} \label{Eq:SecondIntergral}
\int_{\mathcal{F}_0}[\omega_t, u(t), \nabla \Phi_3]\dx
= \int_{\mathcal{F}_0} [ \omega_t, K_{\R^3}[\kappa_{\mathcal{C}_0}+\omega_t], \nabla \Phi_3 ] \dx \\
+ \int_{\mathcal{F}_0} [ \omega_t, u(t)-K_{\R^3}[\kappa_{\mathcal{C}_0}+\omega_t], \nabla \Phi_3 ] \dx .
\end{multline}

\noindent We estimate the two integrals in the right-hand side differently.
%
Concerning the second one, we use the decomposition of $u$ in Lemma~\ref{vitdec} and get
\begin{align*}
\MoveEqLeft
\left| \int_{\mathcal{F}_0} [\omega_t, u(t) - K_{\R^3}[\kappa_{\mathcal{C}_0}+\omega_t], \nabla \Phi_3 ] \dx \right| \\
&\leq \norm{\omega_t}_{L^1} \left( |p_3^\eps| \norm{ \nabla \Phi_3 }_{L^\infty(\supp\omega_t)}
  + \norm{ \uref[\omega_t] }_{L^\infty(\supp\omega_t)}
  +\norm{ H^\eps - K_{\R^3}[\kappa_{\mathcal{C}_0}] }_{L^\infty(\supp\omega_t)} \right) \\
& \quad \times \eps|\log\eps|^\frac{1}{2} \dist(\supp\omega_t,\cals^\eps)^{-3} \\
& \lesssim (1+|p_3^\eps|) \eps|\log\eps| \dist(\supp\omega_t,\cals^\eps)^{-6} \lesssim \eps^\frac{9}{10} s(t)^{-4},
\end{align*} 
where we use the decay estimates in the Propositions~\ref{est pot}, \ref{est ref}, and \ref{est vor filament},
as well as the a priori estimates \eqref{vort smooth} and \eqref{dist bd} for $\omega_t$.

Concerning the first integral in the right-hand side of \eqref{Eq:SecondIntergral}, we have shown in the previous proof (with $s=s(t)-s_0$) that 
\begin{equation*}
\int_{\mathcal{F}_0} [ \omega_t, K_{\R^3}[\kappa_{\mathcal{C}_0}+\omega_0(\cdot+(s(t)-s_0)e_3)], \nabla \Phi_3 ] \dx
\lesssim \eps|\log\eps|^\frac{1}{2} s(t)^{-4},
\end{equation*}
(see \eqref{2nd bra 1} and \eqref{2nd bra est}).

Hence, to complete the estimate we only need to compare the first integral in the right-hand side of 
\eqref{Eq:SecondTermD} with
$\mathcal{D} \bigl[ 0, K_{\R^3}[\kappa_{\mathcal{C}_0} + \tilde{\omega}_0^t], \tilde{\omega}_0^t \bigr]_3$, where we denoted $\tilde{\omega}_0^t :=\omega_0(\cdot+(s(t)-s_0)e_3)$ to lighten the writing.
We split 
\begin{align*}
& \left| \int [e_3,\omega_t,u(t)] 
  - \bigl[ e_3, \tilde{\omega}_0^t, K_{\R^3}[\tilde{\omega}_0^t+\kappa_{\mathcal{C}_0}] \bigr] \dx \right| \\
& \leq \left| \int \bigl[ e_3,\omega_t,K_{\R^3}[\omega_t] \bigr]\dx 
  - \int \bigl[ e_3, \tilde{\omega}_0^t, K_{\R^3}[\tilde{\omega}_0^t] \bigr]\dx \right| 
  + \left| \int \bigl[ e_3,\omega_t, u(t) - K_{\R^3}[\omega_t+\kappa_{\mathcal{C}_0}] \bigr] \dx \right| \\
& \:\:\: + \left| \int \Bigl[ e_3, \omega_t - \tilde{\omega}_0^t, K_{\R^3}[\kappa_{\mathcal{C}_0}] \Bigr] \dx \right| \\
& =:I+II+III .
\end{align*}
The term $I$ is $0$ by the same calculations as in \eqref{canc} above.

In $II$, we may write\begin{align*}
u(t)-K_{\R^3}[\omega_t+\kappa_{\mathcal{C}_0}]=p_3^\eps\nabla\Phi_3+\uref[\omega_t]+(H^\eps-K_{\R^3}[\kappa_{\mathcal{C}_0}]),\end{align*}
and observe that the $e_3$-component of these does not contribute to the integral by the definition of the bracket 
and that we may hence estimate them with Lemma \ref{lem87}.
Using the a priori estimates \eqref{vort smooth} and \eqref{dist bd} on $\omega_t$, and the estimates on the velocity components in Section \ref{sec-Asymptotic}, we then see that
\begin{equation*}
II \lesssim \norm{ \omega_t }_{L^1}(1+|p_3^\eps|) \eps|\log\eps|\dist(\supp\omega_t,\cals^\eps)^{-4}
\lesssim \eps^\frac{8}{10} s(t)^{-4} .
\end{equation*}
For $III$, we note that only the $e_r$-component of $K_{\R^3}[\kappa_{\mathcal{C}_0}]$ contributes to the integral 
by the definition of the bracket and because $K_{\R^3}[\kappa_{\mathcal{C}_0}]$ has no azimuthal component. 
Hence, using the formula \eqref{asy K} and that $\omega_0(\cdot+(s(t)-s_0)e_3)$ 
and $\omega_t$ are supported in $B_2(-s(t)e_3)$,
we see that
\begin{multline*}
III \lesssim \norm{ \omega_0(\cdot+(s(t)-s_0)e_3) - \omega_t }_{L^1}
  \norm{ K_{\R^3}[\kappa_{\mathcal{C}_0}] \cdot e_r }_{L^\infty(B_2(-s(t)e_3))} \\
\lesssim \norm{\omega_0(\cdot+(s(t)-s_0)e_3)-\omega_t}_{L^1}s(t)^{-4}.
\end{multline*}
Using \eqref{cont vort}, this is $\leq \frac{1}{10000}s(t)^{-4}$ for suitably small $\tilde{\delta}$
(which fixes $T_0$ according to Lemma~\ref{no blow up}). 

Putting all these estimates together, this concludes the proof of Lemma~\ref{diff D}.
\end{proof}
%


\section{Appendix. Decay estimates}
\label{app-IPP}

For the decay of $u$, we first use the following Lemma. 
\begin{lem}\label{decay pot}
Let $r>0$ be given and assume that $\bar{u}\in (L^2\cap C^\infty)(\R^3\backslash B_r(0), \R^3)$ 
is such that 
\begin{equation}
\div \bar u  = 0 \quad \text{ and } \quad \curl \bar u=0
\end{equation}
and
\begin{equation}
\int_{\de B_r(0)} \bar{u}\cdot n\dd\sigma=0.\label{inc}
\end{equation}
Then for all $m\in \N_{\geq 0}$ it holds that 
\begin{align*}
|\nabla^m\bar u(x)|\lesssim_{r,m} \norm{\bar u\cdot n}_{L^2(\de B_r(0))}|x|^{-3-m} 
  \text{ as $|x|\rightarrow \infty$}. 
\end{align*}
\end{lem}
\begin{proof}[Proof of  Lemma \ref{decay pot}]
Because $\R^3\backslash B_r(0)$ is simply connected, there is some harmonic potential 
$\bar\Phi$ such that $\bar u=\nabla\bar\Phi$.
If we extend $\bar\Phi$ harmonically to $B_r(0)$, we can recover $\bar\Phi$ 
and its derivatives as a single layer potential (see e.g.\ \cite[Chapter 5.12]{Medkova}) 
\begin{align*}
\nabla^m\bar\Phi(x) 
  = \int_{\de B_r(0)} \llbracket \de_n \bar\Phi \rrbracket(y)
    \nabla_x^m \left( \frac{1}{4\pi|x-y|} \right) \dd\sigma(y),
\end{align*}
where $\llbracket\cdot\rrbracket$ denotes the jump across the boundary.
This decays like $|x|^{-2-m}$, because on the one hand $\llbracket\de_n\bar\Phi\rrbracket$ is mean-free 
by the assumption \eqref{inc} on $\bar u\cdot n=\de_n \bar \Phi$ and on the other hand 
$\llbracket\de_n\bar\Phi\rrbracket$ is in $L^2(\de B_r(0))$ and controlled by $u\cdot n$ by elliptic regularity. 
\end{proof}
\begin{lem} \label{decay prop}
Assume the strong solution to \eqref{chE1}-\eqref{chSolide2} exists up to some time $T$ 
and is such that $\omega$ has bounded support. Then we have the following decay estimates 
\begin{align}
|u(x)|+|\de_t u(x)|&\lesssim |x|^{-3} \text{ as $|x|\rightarrow \infty$}, \label{decay u1} \\
|\nabla u(x)|&\lesssim |x|^{-4} \text{ as $|x|\rightarrow \infty$}\label{decay u2},
\end{align}
locally uniformly on $[0,T)$.
Furthermore, $\lim_{x\rightarrow \infty} \pi(x)$ exists and it holds that 
\begin{align}
\left|\pi(x)-\lim_{x\rightarrow \infty} \pi(x)\right|\lesssim |x|^{-2}.
\end{align}
In particular, all partial integrations in the proof of Proposition~\ref{reform-macro}
 are justified. 
\end{lem}
\begin{proof}
We take $r$ large enough for $\mathcal{S}_0$ and $\supp\omega$ to be contained in $B_r(0)$. 
Then $u$ and $\de_t u$ both fulfill the assumptions of Lemma \ref{decay pot}. 
They are $\div$- and $\curl$-free by assumption and the condition \eqref{inc} holds 
for both because, due to the incompressibility, we have 
\begin{align*}
\int_{\de B_r(0)} u\cdot n\dd\sigma=-\int_{\de\mathcal{S}_0} u\cdot n\dd\sigma=0.
\end{align*}
Furthermore, on compact subintervals of $[0,T)$, the quantities $\norm{u}_{H^1(\de B_r(0))}$ 
and $\norm{\de_t u}_{H^1(\de B_r(0))}$ are uniformly controlled by the assumption 
that the solution is strong and elliptic regularity. This shows \eqref{decay u1} and \eqref{decay u2}. 
Regarding the statement for $\pi$, we observe first that 
\begin{align}
\nabla \pi(x) = -\de_t u-(u-u_\mathcal{S})\nabla u-\Omega(t)\wedge u=O(|x|^{-3})
  \text{  as $|x|\rightarrow \infty$}\label{est grad p}
\end{align}
by the previous estimates on $u$ and the definition \eqref{u_S} of $u_\mathcal{S}$. 
Hence, from the fundamental theorem of calculus, we see that 
$\lim_{x_1\rightarrow +\infty}\pi(x_1,0,0)$ exists and it holds that 
\begin{align*}
\left|\pi(x_1,0,0)-\lim_{r\rightarrow +\infty}\pi(x_1,0,0)\right| \lesssim |x_1|^{-2}.
\end{align*}
Furthermore, by \eqref{est grad p}, $\pi$ is Lipschitz on spheres $\de B_{r'}(0)$ with constant $\lesssim |r'|^{-3}$ and therefore it holds that 
\begin{align*}
\left| \pi(x) - \lim_{x_1\rightarrow +\infty} \pi(x_1,0,0) \right|
\leq \left|\pi(x)-\pi(|x|,0,0)\right|+\left|\pi(|x|,0,0)
 - \lim_{x_1\rightarrow +\infty}\pi(x_1,0,0)\right|
 \lesssim |x|^{-2}.
\end{align*}
   Again this is locally uniform in time because the estimate \eqref{est grad p} is.
\end{proof}
%
%
%
%
%
%
%
%
\section{Appendix. Well-posedness of the macroscopic and of the limit system}
\label{Sec:WP-limit}
In this section, we prove Theorem~\ref{start3}, following the methods of
\cite{InitialsBB} and \cite{GST}. \par
We will use the following three lemmas which are elementary consequences of Fa\`a di Bruno's 
formula and the fact that H\"older spaces are algebras, 
see \cite[Lemmas A.2, A.3 \& A.4]{InitialsBB} in the case of Sobolev spaces.
\begin{lem} \label{BBLemmaA.2}
Let $k\in\N_{\geq 1}$ and $\alpha \in (0,1)$, and let $\Omega$, $\Omega'$ 
be smooth bounded domains of $\R^3$. 
Let $F \in C^{k,\alpha}({\Omega}')$ and 
$G \in C^{k,\alpha}({\Omega})$ with $G({\Omega}) \subset {\Omega}'$. 
Then $F \circ G \in C^{k,\alpha}(\Omega)$ with, for some constant $C$ 
depending only on $\Omega$,  $\Omega'$ and $k$:
\begin{equation} \label{EstBBLemmaA.2}
\| F \circ G \|_{C^{k,\alpha}(\Omega)} \leq C \| F \|_{C^{k,\alpha}(\Omega')} 
  \Big( \|G\|_{C^{k,\alpha}(\Omega)}^{k} +1 \Big).
\end{equation}
\end{lem}
\begin{lem}\label{Holem}
Let $k,\alpha, \Omega,\Omega'$ be as above. Let $F\in C^{k,\alpha}(\Omega')$ and $G,G'\in C^{k-1,\alpha}(\Omega')$ with $G(\Omega),G'(\Omega)\subset \Omega'$. Then there is a constant $C$, depending only on $\Omega,\Omega'$ and $k$ such that \begin{align*}
\norm{F\circ G-F\circ G'}_{C^{k-1,\alpha}(\Omega)}\leq C\norm{F}_{C^{k,\alpha}}\norm{G-G'}_{C^{k-1,\alpha}}\left(1+\norm{G}_{C^{k-1,\alpha}(\Omega')}^k+\norm{G'}_{C^{k-1,\alpha}}^k\right).
\end{align*}
\end{lem}
For the next lemma we denote by $\mathrm{Diff}(\overline{\Omega})$ 
the group of $C^{1}$ diffeomorphisms of $\overline{\Omega}$.
\begin{lem} \label{BBLemma4}
Let $k,\alpha$ be as above. Let $\Omega$ a smooth bounded domain of $\R^3$, $F \in C^{k,\alpha}({\Omega})$ 
and $G \in \mathrm{Diff}(\overline{\Omega}) \cap C^{k,\alpha}({\Omega})$. 
Then for some constant $C$ depending only on $\Omega$, $k \in \N_{\geq 1} $, $\alpha \in (0,1)$ 
and $\| G\|_{C^{k,\alpha}({\Omega})}$, one has
\begin{equation} \label{EstBBLemma4}
\| \partial_{i} (F \circ G^{-1}) \circ G - \partial_{i} F \|_{C^{k-1,\alpha}({\Omega})}
\leq C \| G - \Id \|_{C^{k,\alpha}({\Omega})} \|F \|_{C^{k,\alpha}({\Omega})}.
\end{equation}
\end{lem}
%
%
We now proceed to the proof of Theorem~\ref{start3}. \par
\begin{proof}[Proof of Theorem~\ref{start3}]
In the sequel, the various constants $C>0$ may change from line to line,
and can depend on  $$D_0 := \dist(\supp(\omega_0), \mathcal{C}_0)$$ and 
$R_0$, but are independent of $|p_0|$, $\|\omega_0\|_{C^{\lambda,r}}$, $C_\star$ 
and especially $\varepsilon \in [0,\varepsilon_0]$. \par
\ \par
\noindent
We rely on the Banach-Picard fixed-point theorem. \par
\ \par
\noindent
{$\bullet$\ }{\it Fixed-point operator.}
To shorten the writing, we will denote 
\begin{equation*}
\overline{\mathfrak{O}_0}:=\supp \omega_0.
\end{equation*}
We introduce the domains 
$\overline{\mathfrak{O}_1}$ and $\overline{\mathfrak{O}_2}$ as follows:
\begin{equation*}
\mathfrak{\mathfrak{O}_i} := \left\{ x \in \R^3 \ \Big/ 
  \ \dist(x, \supp \omega_0) \leq \frac{i}{3} \dist(\mathcal{C}_0, \overline{\mathfrak{O}_0})\right\}
  \ \text{ for } \ i=1,2.
\end{equation*}
Given $T>0$, we set
\begin{gather} 
\label{Eq:DefB}
{\mathcal B}:=\bigg\{ 
(p,\eta) \in C^{0}([0,T];\, \R^6 \times C^{\lambda+1,r}(\R^3;\R^{3})) \ \big/  \\
%
%
\label{Eq:DefB2}
\forall t \in [0,T], \ \forall x \in \R^3 \setminus \overline{\mathfrak{O}_2}, \ \ \eta(t,x) =x, \\
\label{Eq:DefB3}
\| p \|_{C^{0}([0,T];\, \R^6)} \leq C_\star (1 + |p_0| + \| \omega_0 \|_{C^{\lambda,r}}) , \\
\label{Eq:DefB4}
\text{ and }  \quad
 \| \eta - \Id \|_{C^{0}([0,T];\, C^{\lambda+1,r}(\overline{\mathfrak{O}_0}))} 
  \leq \min \left(\frac{1}{2}, \frac{1}{3} \dist\left(\overline{\mathfrak{O}_0}, \mathcal{C}_0\right) \right)
    \bigg\},
\end{gather}
where the constant $C_\star$ depends only on $D_0$ and $R_0$ and will be defined later. 
\par
Note in particular that for any $(p,\eta) \in \mathcal{B}$ and any $t \in [0,T]$,
$\eta(t,\cdot)$ is a $C^1$-diffeomorphism which moreover satisfies that
$\eta(t, \, \overline{\mathfrak{O}_0}) \subset \overline{\mathfrak{O}_1}$. \par
We endow ${\mathcal B}$ with the 
$C^{0}([0,T];\, \R^6) \times L^\infty([0,T];C^{\lambda+1,r}(\overline{\mathfrak{O}_2};\,\R^{3}))$-distance, for which it is complete.

Now we define ${\mathcal T}:{\mathcal B} \rightarrow {\mathcal B}$ as follows. 
Given $(p,\eta) \in {\mathcal B}$, we first introduce 
$\omega: [0,T] \times \R^3 \rightarrow \R^{3}$ by
\begin{equation} \label{DefOmega}
\omega(t,x) = (\nabla \eta)(t,\eta^{-1}(t,x)) \cdot \omega_{0}(\eta^{-1}(t,x)) 
\text{ for } x \in \R^3.
\end{equation}
Next, we define $u: [0,T] \times \R^3 \rightarrow \R^{3}$ by \eqref{Eq:UnifiedDecomp},
that is,
\begin{equation} \label{cheapBS}
\left. \begin{array}{c}
u =  \mu   H^*  + K_{ \R^3} [\omega] \ \text{ if } \ \varepsilon =0, \\
u =  \mu   H^\varepsilon  + K_{ \R^3} [\omega]  + \uref^\varepsilon[\omega] 
    + \sum_{i=1}^{6} p_i \nabla \Phi_i^\varepsilon
\ \text{ if } \ \varepsilon >0,
\end{array} \right. 
\end{equation}
where $\uref^\varepsilon$ is defined by \eqref{a1}-\eqref{a4}. \par

Next, we set 
\begin{equation*}
u_{\mathcal{S}} (t,x) := \ell(t) + \Omega (t) \wedge x, 
\ \text{ where } \ p =: ( \ell , \Omega ) \in \R^3 \times \R^3 . 
\end{equation*}
Now we introduce a cutoff function $\Lambda \in C^\infty(\R^3;[0,1])$ such that 
$\Lambda=1$ on $\overline{\mathfrak{O}_1}$ 
and $\Lambda=0$ on $\R^3 \setminus \overline{\mathfrak{O}_2}$. 
We note in passing that, classically constructing $\Lambda$ based on distances to sets 
and convolution with a compactly supported approximation of unity,
the $C^k$ norms of $\Lambda$ can be made dependent only on $k$ and 
$\dist(\mathcal{C}_0,\overline{\mathfrak{O}_0})$,  
but not on the particular geometry of $\overline{\mathfrak{O}_0}$. \ \par
Then we define the ``updated flow'' $\tilde{\eta}(t,x)$ as 
\begin{equation} \label{up-flot}
\tilde{\eta}(t,x) - x 
  = \int_{0}^{t} \Big[ \Lambda(\cdot) (u(s,\cdot) - u_{\mathcal{S}}(s,\cdot) ) \Big] 
  \circ {\eta} (s,x) \dd s.
\end{equation}
Note that, despite the fact that $u$ is merely defined in $\mathcal{F}_0^\varepsilon$
(for $\varepsilon>0$) or singular on $\mathcal{C}_0$ (for $\varepsilon=0$),
the support of $\Lambda$ makes $\Lambda(\cdot) (u(\cdot) - u_{\mathcal{S}}(\cdot) )$
regularly defined in $\R^3$. \par
%
%
%
Next we define the ``updated solid velocity'' $\tilde p$ as the solution of 
\begin{equation} \label{Eq:EDOtildeP}
{\mathcal M}_{g}  \tilde{p}\,'  + \langle  \Gamma_g , \tilde p , \tilde p \rangle
=  \mu B\tilde p + {\mathcal D} [ \tilde{p}, u, \omega ] ,
\end{equation}
with initial data $p_0$, using the same unified notation as in Section \ref{Sec:MainProof}. \par
Let us prove that this solution is globally defined in $[0,T]$. 
The following estimates are valid for $t \in [0,\tilde{T})$ for constants independent of $t$,
where $\tilde{T}$ is the maximal time of existence for $\tilde{p}$. 
First, by \eqref{Eq:DefB4} and \eqref{DefOmega}, we can estimate $\omega$ as
\begin{equation} \label{Eq:EstOmega0}
\| \omega(t,\cdot) \|_{C^{\lambda,r}(\R^3)} \leq C \| \omega_0 \|_{C^{\lambda,r}(\overline{\mathfrak{O}_0})} ,
\end{equation}
and see that 
\begin{equation} \label{Eq:SuppOmega}
\supp \omega(t,\cdot) \subset \overline{\mathfrak{O}_1} \text{ for all } t \text{ in } [0,T].
\end{equation}
By \eqref{cheapBS}, Propositions~\ref{est pot}, \ref{est ref} and \ref{est vor filament}, 
the elliptic regularity for $K_{\R^3}$ 
and the remoteness of $\overline{\mathfrak{O}_2}$ from $\mathcal{C}_0$, we can estimate $u$ as 
\begin{equation} \label{EstU0}
\| u(t,\cdot) \|_{C^{\lambda+1,r}(\overline{\mathfrak{O}_2})} 
 \leq C \left( 1 + |p(t)| + \|\omega(t,\cdot)\|_{C^{\lambda,r}( \overline{\mathfrak{O}_0} )} \right) 
\leq C  \left( 1 + |p_0| + \|\omega_{0}\|_{C^{\lambda,r}( \overline{\mathfrak{O}_0} )} \right),
\end{equation}
where we used \eqref{Eq:DefB3}. \par
%
%
Now, going back to \eqref{Eq:EDOtildeP}, on a time interval where $\tilde{p}$ is defined, 
we can perform an energy estimate.
Relying on the skew-symmetry of $B$ and on \eqref{AnnulationGamma2}, 
we observe that it suffices to consider the last term.
We write:
\begin{equation*} 
\frac{1}{2}\left|\frac{\mathrm{d}}{\mathrm{d}t}\tilde{p}\cdot(\mathcal{M}_g+\mathcal{M}_a)\tilde{p}\right|=\left| \tilde{p} \cdot {\mathcal D} [ \tilde{p}, u, \omega ] \right| 
\leq (1+R_0) |\overline{\mathfrak{O}_1}| \, (|\tilde{p}| + \|\omega(t,\cdot)\|_\infty) 
   \| u(t,\cdot) \|_{L^\infty(\overline{\mathfrak{O}_2})} |\tilde{p}|,
\end{equation*}
where $R_0$ appears due to the cases $i=4,5,6$ in  $\mathcal{D}_i$ and we can control the size of the supports with $R_0$ through \eqref{Eq:DefB4}. 
We emphasize that this is true for both the $\mathcal{D}$ in \eqref{defD} (case $\varepsilon>0$) 
and the $\mathcal{D}$ in \eqref{defD*} (case $\varepsilon=0$.) 
Using this estimate, the positive-definiteness of $\mathcal{M}_g$ and Gronwall's lemma, we see that $\tilde{p}$
is indeed defined globally (up to time $T$), and satisfies
\begin{equation} \label{Eq:EsttildeP}
\| \tilde{p} \|_{C([0,T];\, \R^6)}
\leq C \left( 1 + |p_0| + \| \omega_0 \|_{C^{\lambda,r}(\overline{\mathfrak{O}_0})} \right).
\end{equation}
The constant $C$ above depends only on $D_0$ and $R_0$ and gives the constant $C_\star$ 
in \eqref{Eq:DefB}. \par
Now, we define on the time interval $[0,T]$
\begin{equation} \label{DefT}
{\mathcal T}(p,\eta):=(\tilde p,\tilde{\eta}).
\end{equation}
{$\bullet$\ }{\it Proof of $\mathcal{T}(\mathcal{B}) \subset \mathcal{B}$.}
Let us prove that for proper $T>0$ of the form \eqref{Eq:EstTexistence},
${\mathcal T}$ maps ${\mathcal B}$ into $ {\mathcal B}$. 
That $\tilde{\eta}$ satisfies the relation \eqref{Eq:DefB2} comes from 
the support of $\Lambda$, 
and that $\tilde{p}$ satisfies \eqref{Eq:DefB3} comes from the above analysis.

Concerning \eqref{Eq:DefB4}, we first gather the previous estimates that are valid on $[0,T]$:
\begin{gather}
\label{Eq:EstOmega}
\| \omega \|_{L^{\infty}(0,T;\, C^{\lambda,r}(\R^3))} 
  \leq C \| \omega_0 \|_{C^{\lambda,r}(\overline{\mathfrak{O}_0})} , \\
\label{EstU}
\| u \|_{L^{\infty}(0,T;\, C^{\lambda+1,r}(\overline{\mathfrak{O}_2}))} 
\leq C  \left( 1 + |p_0| + \|\omega_{0}\|_{C^{\lambda,r}( \overline{\mathfrak{O}_0} )} \right).
\end{gather}
Also, for some geometric constant, one has
\begin{equation}
\| u_\mathcal{S} \|_{C^{\lambda+1,r}(\overline{\mathfrak{O}_2}) } \leq C \| {p} \|_{C^{0}([0,T];\,\R^{6})}.\label{us smooth}
\end{equation}
%
%
Now, by \eqref{up-flot} and Lemma \ref{BBLemmaA.2}, we see that 
\begin{align*}
& \|\tilde{\eta} - \Id \|_{ C^{0}([0,T];\, C^{\lambda+1,r}(\overline{\mathfrak{O}_2})) } \\
& \hspace{0.5cm}\leq C T
\left( 1 + \| {\eta} -\Id \|_{C^{0}([0,T]; \, C^{\lambda +1,r}(\overline{\mathfrak{O}_2})) }^{\lambda +1} \right) 
\Big( \| u \|_{L^{\infty}(0,T;\, C^{\lambda+1,r}(\overline{\mathfrak{O}_2}))} + \| p \|_{C^{0}([0,T];\R^{6})} \Big) \\
& \hspace{0.5cm}\leq C T
\left( 1 + \| {\eta} -\Id \|_{C^{0}([0,T]; \, C^{\lambda +1,r}(\overline{\mathfrak{O}_2})) }^{\lambda +1} \right) 
\Big( 1 + |p_0| + \|\omega_{0}\|_{C^{\lambda,r}( \overline{\mathfrak{O}_0} )} \Big) .
\end{align*}
Hence, if $T>0$ is such that
\begin{equation*}
CT \left(1 + \left(\frac{1}{2}\right)^{\lambda+1} \right) 
( 1 + |p_0| + \|\omega_{0}\|_{C^{\lambda,r}( \overline{\mathfrak{O}_0} )} ) \leq 1/2, 
\end{equation*}
then ${\mathcal T}$ maps ${\mathcal B}$ into itself. \par
\medskip
\noindent

{\bf $\bullet$\ }{\it Proof that $\mathcal{T}$ is a contraction.}
Now, let us prove that ${\mathcal T}$ is contractive for small $T>0$. 
Given $(p_1 ,\eta_{1})$ and $ (p_1 ,\eta_{2})$ in ${\mathcal B}$, 
we let $\omega_{1}$, $\omega_{2}$, $u_{1}$, $u_{2}$, etc. be the various objects 
associated to $\eta_{1}$ and $\eta_{2}$ in the construction of ${\mathcal T}$. 
We also define
\begin{equation} \label{DefMathcalU}
{\mathcal U}_{i}(t,x) = [\Lambda \, u_{i}] (t,{\eta}_{i}(t,x)) 
  - [\Lambda\,  u_{\mathcal{S},i}] (t,{\eta}_{i}(t,x)) 
\text{ for } (t,x) \in [0,T] \times \R^3 \text{ and } i=1,2,
\end{equation}
so that we have
\begin{equation} \label{DiffFlows}
\tilde{\eta}_{1}(t,x) - \tilde{\eta}_{2}(t,x) = \int_{0}^{t} 
  \left[{\mathcal U}_{1}(s,x) - {\mathcal U}_{2}(s,x)\right] \, \mathrm{d}s.
\end{equation}
%
%
We begin with an estimate of $\| \omega_1 - \omega_2\|_{L^1(\R^3)}$. 
We rely on \eqref{DefOmega}, and see that 
\begin{equation}
\| \omega_1{\circ\eta_1} - \omega_2{\circ\eta_2}\|_{C^{\lambda,r}(\R^3)} 
\lesssim \| \omega_0 \|_{C^{\lambda,r}(\R^3)} 
\| \eta_{1} - \eta_{2} \|_{C^{\lambda+1,r}(\overline{\mathfrak{O}_2})}.\label{Holest}
\end{equation}
Using \eqref{Eq:SuppOmega}, which is satisfied by both $\omega_1$ and $\omega_2$, and Lemma \ref{Holem}, 
we find for some geometric constant $C>0$:
\begin{equation} \label{Eq:DiffOmegaL1}
\| \omega_1 - \omega_2\|_{L^1(\R^3)} \leq C \| \omega_0 \|_{C^{\lambda,r}(\R^3)}  
                  \| \eta_{1} - \eta_{2} \|_{C^{\lambda+1,r}(\overline{\mathfrak{O}_2})}.
\end{equation}
Now, let us prove that for some geometric constant $C>0$, we have
\begin{multline*}
\mkern-18mu\| {\mathcal U}_{1} - {\mathcal U}_{2} \|
              _{L^{\infty}([0,T];\, C^{\lambda+1,r}( \overline{\mathfrak{O}_2} ))} \\
\leq C  \left( 1 + \|\omega_{0}\|_{C^{\lambda,r}( \overline{\mathfrak{O}_0} )} 
      + \| p \|_{C^{0}([0,T];\R^{6})} \right)
    \| \eta_{1} -\eta_{2}\|_{L^{\infty}([0,T];C^{\lambda+1,r}(\overline{\mathfrak{O}_2}))}+C\norm{p_1-p_2}_{C([0,T];\R^6)}\mkern-4mu.
\end{multline*}
For $t \in [0,T]$, we have, omitting the dependence on $t$ to simplify the notations, 
using the smoothness of $\Lambda$,
\begin{equation} \label{Eq:1EstU}
\| {\mathcal U}_{1} - {\mathcal U}_{2} \|_{C^{\lambda+1,r}(\overline{\mathfrak{O}_2})}
\leq
\| [\Lambda u_{1}] \circ \eta_{1} - [\Lambda u_{2}] \circ \eta_{2} 
                \|_{C^{\lambda+1,r}(\overline{\mathfrak{O}_2})} 
+ \| [\Lambda u_{\mathcal{S},1}] \circ \eta_{1} - [\Lambda u_{\mathcal{S},2}] \circ \eta_{2}   \|_{C^{\lambda+1,r}(\overline{\mathfrak{O}_2})}.
\end{equation}
For the second term, we simply estimate, using \eqref{Eq:EsttildeP} and Lemma \ref{Holem} as well as \eqref{us smooth}
\begin{align*}
 &\| [\Lambda u_{\mathcal{S},1}] \circ \eta_{1} - [\Lambda u_{\mathcal{S},2}] \circ \eta_{2} \|
                      _{C^{\lambda+1,r}(\overline{\mathfrak{O}_2})} \\
& \hspace{1cm} \leq 
\| [\Lambda u_{\mathcal{S},1}] \circ \eta_{1} - [\Lambda u_{\mathcal{S},2}] \circ \eta_{1} \|
                      _{C^{\lambda+1,r}(\overline{\mathfrak{O}_2})} 
+ \ \| [\Lambda u_{\mathcal{S},2}] \circ \eta_{1} - [\Lambda u_{\mathcal{S},2}] \circ \eta_{2} \|
                      _{C^{\lambda+1,r}(\overline{\mathfrak{O}_2})} \\
& \hspace{1cm} \leq C \Big[ |p_1 - p_2|
+\big( 1 + \| \omega_{0}\|_{C^{\lambda,r}(\overline{\mathfrak{O}_2})} 
+ \| p \|_{C^{0}([0,T];\, \R^{6})} \big) 
  \| \eta_{1} - \eta_{2} \|_{C^{\lambda+1,r}(\R^3)} \Big].
\end{align*}
For the first term on the right-hand side of \eqref{Eq:1EstU}, we rely on the 
decomposition \eqref{Eq:UnifiedDecomp}.
We write
\begin{eqnarray}
\nonumber 
\| [\Lambda u_{1}] \circ \eta_{1} - [\Lambda u_{2}] \circ \eta_{2} \|
_{C^{\lambda+1,r}(\overline{\mathfrak{O}_2})} 
&\leq&
|\mu|\, \| H^\varepsilon \circ \eta_{1} - H^\varepsilon \circ \eta_{2} \|
                    _{C^{\lambda+1,r}(\overline{\mathfrak{O}_2})} \\
\nonumber 
&+&
\| \left\{ \Lambda K_{\R^3}[\omega_{1}]\right\} \circ \eta_{1} 
  - \left\{ \Lambda K_{\R^3}[\omega_{2}]\right\} \circ \eta_{2} \|
                    _{C^{\lambda+1,r}(\overline{\mathfrak{O}_2})} \\
\nonumber 
&+&
\left\| \left\{ \Lambda \uref^\eps[\omega_{1}]\right\} \circ \eta_{1} 
  - \left\{ \Lambda \uref^\eps[\omega_{2}]\right\} \circ \eta_{2} \right\|
            _{C^{\lambda+1,r}(\overline{\mathfrak{O}_2})} \\
\label{Eq:MajLag}
&+&
\left\| \sum_{k=1}^{6} p^1_k (\Lambda \nabla \Phi_k^\varepsilon) \circ \eta_{1} 
  - \sum_{k=1}^{6} p^2_k (\Lambda \nabla \Phi_k^\varepsilon) \circ \eta_{2}  \right\|
                    _{C^{\lambda+1,r}(\overline{\mathfrak{O}_2})}
\end{eqnarray}
where $\uref^\eps$ and $\Phi_k^\eps$ are understood as $0$ if $\eps=0$.
%
%
%

Concerning the first term in the right-hand side of \eqref{Eq:MajLag}, 
we use the positive distance between $\mathcal{S}_0^\varepsilon$ and $\mathfrak{O}_2$ and
Proposition~\ref{est vor filament} to deduce
\begin{equation*}
\| H^\varepsilon \circ \eta_{1} - H^\varepsilon \circ \eta_{2} \|
                          _{C^{\lambda+1,r}(\overline{\mathfrak{O}_2})}
\lesssim \| \eta_{1} - \eta_{2} \|_{C^{\lambda+1,r}(\R^3)}.
\end{equation*}
Reasoning in the same way for the last term in the right-hand side of \eqref{Eq:MajLag},
and using Proposition~\ref{est pot}, we find
\begin{equation*}
\left\| \sum_{k=1}^{6} p^1_k \nabla \Phi_k^\varepsilon \circ \eta_{1} 
  - \sum_{k=1}^{6} p^2_k \nabla \Phi_k^\varepsilon \circ \eta_{2}  \right\|
                          _{C^{\lambda+1,r}(\overline{\mathfrak{O}_2})}
\lesssim |p_1 - p_2| + \| \eta_{1} - \eta_{2} \|_{C^{\lambda+1,r}(\R^3)}.
\end{equation*}
We now focus on the second term in \eqref{Eq:MajLag}.
Relying on \eqref{Eq:DefB4}, the smoothness of $\Lambda$, elliptic regularity estimates 
and Lemma~\ref{BBLemmaA.2}, we deduce:
\begin{align}
\nonumber
\| [\Lambda K_{\R^3}[\omega_{1}]] \circ \eta_1 & - [\Lambda K_{\R^3}[\omega_{2}]] \circ \eta_2 \|_{C^{\lambda+1,r}(\overline{\mathfrak{O}_2})} \\
\nonumber 
&\leq C \|  K_{\R^3}[\omega_{1}] \circ \eta_1 \circ \eta_2^{-1} - K_{\R^3}[\omega_{2}] \|_{C^{\lambda+1,r}(\overline{\mathfrak{O}_2})} \\
\label{nononumb}
&\leq
C \, \Big( \| \curl( K_{\R^3}[\omega_{1}] \circ \eta_{1} \circ \eta_{2}^{-1}) 
            - \curl ( K_{\R^3}[\omega_{2}]) \|_{C^{\lambda,r}(\overline{\mathfrak{O}_2})} \\
\nonumber 
& \  \hspace{2cm} 
+ \ \| \div( K_{\R^3}[\omega_{1}] \circ \eta_{1} \circ \eta_{2}^{-1}) 
      - \div ( K_{\R^3}[\omega_{2}]) \|_{C^{\lambda,r}(\overline{\mathfrak{O}_2})}  \Big) 
\nonumber
\end{align}
Concerning the first term in the right-hand side of \eqref{nononumb}, using \eqref{DefOmega}, \eqref{Eq:EstOmega}, \eqref{Holest},
Lemmas~\ref{BBLemmaA.2} and \ref{BBLemma4}, we see that 
\begin{align*}
\| \curl(K_{\R^3}[\omega_{1}] \circ & \,\eta_{1} \circ \eta_{2}^{-1}) 
        - \curl (K_{\R^3}[\omega_{2}]) \|_{C^{\lambda,r}(\overline{\mathfrak{O}_2})} \\
\leq & \ 
\| \curl \left( K_{\R^3}[\omega_{1}] \circ \eta_{1} \circ \eta_{2}^{-1} \right) 
- (\curl K_{\R^3}[\omega_{1}]) \circ \eta_{1} \circ \eta_{2}^{-1} 
                   \|_{C^{\lambda,r}(\overline{\mathfrak{O}_2})} \\
& \hspace{2cm} +
\|  (\curl K_{\R^3}[\omega_{1}]) \circ \eta_{1} \circ \eta_{2}^{-1} 
  - (\curl K_{\R^3}[\omega_{2}]) \|_{C^{\lambda,r}(\overline{\mathfrak{O}_2})} \\
\leq & \ 
C \, \| \eta_{1} \circ \eta_{2}^{-1}-\Id \|_{C^{\lambda+1,r}(\overline{\mathfrak{O}_2})} 
  \, \left\| K_{\R^3}[\omega_{1}](t) \right\|_{ C^{\lambda+1,r}(\overline{\mathfrak{O}_2})} \\
& \hspace{2cm} +
C \left\|  (\curl K_{\R^3}[\omega_{1}]) \circ \eta_{1} 
      - (\curl K_{\R^3}[\omega_{2}]) \circ \eta_{2} \right\|_{C^{\lambda,r}(\overline{\mathfrak{O}_2})} \\
\leq & \ 
C \left( \|\omega_{0} \|_{ C^{\lambda,r}( \overline{\mathfrak{O}_2}) )} 
      + \| \omega_{1}(t) \|_{ C^{\lambda,r}(\overline{\mathfrak{O}_2})} \right)
    \| \eta_{1} - \eta_{2} \|_{C^{\lambda+1,r}(\overline{\mathfrak{O}_2})} \\
\leq & \ 
C  \|\omega_{0} \|_{ C^{\lambda,r}( \overline{\mathfrak{O}_2}) )} 
    \| \eta_{1} - \eta_{2} \|_{C^{\lambda+1,r}(\overline{\mathfrak{O}_2})}.
\end{align*}
In the same way, the second term in the right-hand side of \eqref{nononumb} is treated as
\begin{align}
\nonumber 
\| \div(K_{\R^3}[\omega_{1}] \circ \eta_{1} \circ \eta_{2}^{-1}) & - \div (K_{\R^3}[\omega_{2}]) \|_{C^{\lambda,r}(\R^3)} \\
\nonumber 
& = \ \| \div(K_{\R^3}[\omega_{1}] \circ \eta_{1} \circ \eta_{2}^{-1}) 
              \|_{C^{\lambda,r}(\overline{\mathfrak{O}_2})} \\
\label{Est136}
& \leq \ C \, \| \eta_{1} - \eta_{2} \|_{C^{\lambda+1,r}(\overline{\mathfrak{O}_2})} 
    \| \omega_0 \|_{ C^{\lambda,r}(\overline{\mathfrak{O}_2})}.
\end{align}
It remains to estimate the third term in \eqref{Eq:MajLag}.
We write
\begin{align*}
 \| \left\{ \Lambda \uref[\omega_{1}]\right\} \circ \eta_{1} 
& - \left\{ \Lambda \uref[\omega_{2}]\right\} \circ \eta_{2} \|
                          _{C^{\lambda+1,r}(\overline{\mathfrak{O}_2})} \\
& \leq 
C \|  \uref[\omega_{1}] \circ \eta_{1} -  \uref[\omega_{2}] \circ \eta_{2} \|
                          _{C^{\lambda+1,r}(\overline{\mathfrak{O}_2})}
+ C \, \| \eta_{1} - \eta_{2} \|_{C^{\lambda+1,r}(\overline{\mathfrak{O}_2})} \\
& \leq C \|  \uref[\omega_{1} - \omega_2 ] \circ \eta_{1} \|
                        _{C^{\lambda+1,r}(\overline{\mathfrak{O}_2})}
+  C \|  \uref[\omega_{2}] \circ \eta_{1} -  \uref[\omega_{2}] \circ \eta_{2} \|
                         _{C^{\lambda+1,r}(\overline{\mathfrak{O}_2})} \\
& \ \ 
+ C \, \| \eta_{1} - \eta_{2} \|_{C^{\lambda+1,r}(\overline{\mathfrak{O}_2})} .
\end{align*}
Now we rely on the remoteness of $\mathfrak{O}_2$ from $\mathcal{S}_0^\varepsilon$
and on Proposition~\ref{est ref}, together with Lemma \ref{BBLemmaA.2} to deduce
\begin{align*}
\|  \uref[\omega_{1} - \omega_2 ] \circ \eta_{1} \|
        _{C^{\lambda+1,r}(\overline{\mathfrak{O}_2})}
  & \lesssim \| \eta_1 \| _{C^{\lambda+1,r}(\overline{\mathfrak{O}_2})}
  \, \| \omega_1 - \omega_2 \|_{L^1} \\
& \lesssim (1 + \|\omega_0\|_{C^{\lambda,r}(\overline{\mathfrak{O}_2}) })
   \| \eta_{1} - \eta_{2} \|_{C^{\lambda+1,r}(\overline{\mathfrak{O}_2})},
\end{align*}
where we used \eqref{Eq:DiffOmegaL1} and have in the same way, using the Proposition \eqref{est ref} with $m=\lambda+3$ and Lemma \ref{Holem}:
\begin{eqnarray*}
\|  \uref[\omega_{2}] \circ \eta_{1} -  \uref[\omega_{2}] \circ \eta_{2} \|
                    _{C^{\lambda+1,r}(\overline{\mathfrak{O}_2})} 
&\lesssim& \|  \uref[\omega_{2}] \|_{C^{\lambda+2,r}(\overline{\mathfrak{O}_2})}
 \| \eta_{1} - \eta_{2} \|_{C^{\lambda+1,r}(\overline{\mathfrak{O}_2})} \\
&\lesssim& \|  \omega_{2} \|_{L^1(\mathfrak{O}_2)}
   \| \eta_{1} - \eta_{2} \|_{C^{\lambda+1,r}(\overline{\mathfrak{O}_2})} \\
&\lesssim& \|\omega_0\|_{C^{\lambda,r}(\overline{\mathfrak{O}_2}) }
  \| \eta_{1} - \eta_{2} \|_{C^{\lambda+1,r}(\overline{\mathfrak{O}_2})}.
\end{eqnarray*}
Hence having completely estimated the right-hand side of \eqref{Eq:1EstU} we finally get
\begin{equation*}
\| {\mathcal U}_{1} - {\mathcal U}_{2} \|_{C^{\lambda+1,r}(\overline{\mathfrak{O}_2})}
\leq
C \left( 1 + |p_0| + \|\omega_{0}\|_{C^{\lambda,r}( \overline{\mathfrak{O}_0} )} \right)
\| \eta_{1} - \eta_{2} \|_{C^{\lambda+1,r}(\overline{\mathfrak{O}_2})}.
\end{equation*}
Going back to \eqref{DiffFlows}, it follows that for some small $T>0$ 
of the form \eqref{Eq:EstTexistence}, we have
\begin{equation} \label{Eq:ContractionU}
\| \tilde{\eta}_{1} - \tilde{\eta}_{2} \|_{C^{\lambda+1,r}(\overline{\mathfrak{O}_2})} 
\leq \frac{1}{4}
 \Big(\| \eta_{1} - \eta_{2} \|_{C^{\lambda+1,r}(\R^3)} + |p_1-p_2|\Big).
\end{equation}
Concerning $\tilde{p}_1$ and $\tilde{p}_2$ which satisfy \eqref{Eq:EDOtildeP} with
$(u_1, \omega_1)$ and $(u_2, \omega_2)$ respectively, we take the difference of the two equations
and use the fact that they both satisfy \eqref{Eq:EsttildeP}. 
We easily arrive at
\begin{equation} \label{tildePpreGronwall}
| {\mathcal M}_{g}  (\tilde{p}_1 - \tilde{p}_2)'| 
\leq C \left( 1 + \| u_{0}\|_{C^{\lambda+1,r}(\overline{\mathfrak{O}_2})} 
+ \| p \|_{C^{0}([0,T];\, \R^{6})} \right) 
  \big( |\tilde{p}_1 - \tilde{p}_2| + \| \omega_1 - \omega_2\|_{L^1(\R^3)}\big).
\end{equation}
Using \eqref{Eq:DiffOmegaL1} and the fact that $\tilde{p}_1$ and $\tilde{p}_2$ have 
the same initial data and Gronwall's lemma, 
we deduce that for some geometric constant $C'>0$, we have
\begin{equation*}
|\tilde{p}_1 - \tilde{p}_2| 
\leq C' T \left( 1 + \| u_{0}\|_{C^{\lambda+1,r}(\overline{\mathfrak{O}_2})} 
                  + \| p \|_{C^{0}([0,T];\, \R^{6})} \right).
\end{equation*}
With \eqref{Eq:ContractionU}, we deduce that for small enough $T>0$ 
of the form \eqref{Eq:EstTexistence},
the operator ${\mathcal T}$ is contractive. \par
\ \par
\noindent
$\bullet$ {\it Conclusion.}
To conclude, it remains to show that such a fixed-point is indeed a solution satisfying 
the requirements. Let $(p,\eta)$ be such a fixed point, and associate the various
quantities $u$, $\omega$, etc., defined together with $\mathcal{T}$. \par
First, we see that for all $ t \in [0,T]$, the flow map ${\eta}(t,\cdot)$ is 
a volume-preserving diffeomorphism from  $\overline{\mathfrak{O}_0}$ on its image in $\R^3$,
due to the fact that both $u$ and $u_{\mathcal{S}}$ are divergence-free, 
the fact that $\Lambda=1$ on $\overline{\mathfrak{O}_1}$, \eqref{Eq:SuppOmega} and 
the Liouville theorem.

Next it follows from \eqref{DefOmega}-- \eqref{up-flot}, 
the fact that $\Lambda=1$ on $\overline{\mathfrak{O}_1}$,
\eqref{Eq:SuppOmega} and the chain rule that $\omega$ satisfies
\eqref{omega-eqeps} (for $\varepsilon>0$) or \eqref{omega-eq*} (for $\varepsilon=0$).
Due to \eqref{Eq:EDOtildeP}, $p$ satisfies \eqref{eql1-WV} (resp. \eqref{eqeps-WV}).
Hence  a fixed point of ${\mathcal T}$ 
in ${\mathcal B}$ gives a solution to the vorticity equation.
Conversely, one can check that a solution gives a fixed point to $\mathcal{T}$. \par
Concerning the regularities, the regularity $\eta \in W^{1,\infty}(C^{\lambda+1,r} (\R^3))$
and the one of $\omega$ follow directly from the construction.
The regularity of  $u$ (for $\varepsilon>0$) and $u - \mu H$ (for $\varepsilon=0$)
is a consequence of the one of $\omega$ and of Schauder elliptic regularity 
estimates (cf.\ \cite[Ch.\ 6]{GilbargTrudinger}). 
Note that $p \in C^1 ( [0,T];\R^6 )$ follows from \eqref{3}.
%
%
%
%
%
%
The estimates \eqref{Eq:UniformEstimates}, given a fixed size $M>0$, on the time
interval $[0,\underline{T}]$, are consequences of \eqref{Eq:EstOmega}, \eqref{EstU0}, 
\eqref{Eq:EsttildeP} and \eqref{DefOmega}.
%

To show the estimate on the time derivatives in \eqref{Eq:UniformEstimates}, we show uniform estimates for $\norm{\de_t^l \omega^\eps}_{C^{\lambda-l,\lambda}}$ and $\de_t^l p^\eps$ by induction in $l\leq \lambda+1$, which then yields the boundedness of these norms by smooth approximation, since smooth solutions are easily seen to be $C^\infty$ in time. The base case $l=0$ is already covered by the previous arguments. \\
Regarding the induction step $l\rightarrow l+1$, we see from the fact that $B$ and $\mathcal{D}$ are polynomial in $p^\eps,u^\eps,\omega^\eps$ but have no other time dependence that \begin{align*}
\left|\frac{\mathrm{d}^{l+1}}{\mathrm{d}t^{l+1}}p^\eps\right|&=\left|\frac{\mathrm{d}^{l}}{\mathrm{d}t^l}\left((\mathcal{M}_g+\mathcal{M}_a^\eps)^{-1}\big(\langle \Gamma_a^\eps+\Gamma_g,p^\eps,p^\eps\rangle-\mu B p-\mathcal{D}[p,u^\eps,\omega^\eps]\big)\right)\right|\\
&\leq C(l,\,R_0,D_0) \left(1+\sum_{k=0}^l \left|\frac{\mathrm{d}^{k}}{\mathrm{d}t^{k+1}}p^\eps\right|+\norm{\de_t^k u^\eps}_{L^\infty(\supp\omega^\eps)}+\norm{\de_t^k \omega^\eps}_{L^\infty(\R^3)}\right)^3,
\end{align*}
where we have also used that $\mathcal{D}$ and $B$ are controlled through $R_0$, $D_0$ and the $L^\infty$-norms by elliptic regularity in the case of $B$ and by definition in case of $\mathcal{D}$.
We can further estimate, using the definition of $u$ and the estimates in the Propositions \ref{est pot}, \ref{est ref} and Corollary \eqref{Cor:DecayH}  
\begin{align}
\norm{\de_t^k u^\eps}_{L^\infty(\supp\omega^\eps)}&\leq \norm{\de_t^k u^\eps}_{C^{0,r}(\supp\omega^\eps)}\nonumber\\
&\lesssim \norm{H^\eps}_{C^{0,r}(\supp\omega^\eps)}+\norm{K_{\R^3}[\de_t^k\omega]}_{C^{0,r}(\supp\omega^\eps)}+\norm{\uref^\eps[\de_t^k\omega]}_{C^{0,r}(\supp\omega^\eps)}\nonumber\\&\quad+C(R_0,D_0)\left|\frac{\mathrm{d}^{k}}{\mathrm{d}t^k}p^\eps\right|\nonumber\\
&\leq C(R_0,D_0)\left(1+\left|\frac{\mathrm{d}^{k}}{\mathrm{d}t^k}p^\eps\right|+\norm{\de_t^k \omega^\eps}_{L^\infty(\R^3)}\right),\label{est dtu}
\end{align}
for $k\leq l$.
Similarly, we have from the definition \begin{align}
\norm{\de_t^k u_{\mathcal{S}}^\eps}_{C^m(\supp\omega^\eps)}\leq C(k,m,R_0,D_0)\left|\frac{\mathrm{d}^{k}}{\mathrm{d}t^k}p^\eps\right|\label{est dtus}
\end{align}
for every $m$.
Combining \eqref{est dtu} and \eqref{est dtus}, we conclude that \begin{align*}
\MoveEqLeft\norm{\de_t^{l+1}\omega^\eps}_{C^{\lambda-(l+1),r}(\R^3)}\leq \norm{\de_t^l\nabla\times ((u^\eps-u_{\mathcal{S}}^\eps)\times \omega^\eps) }_{C^{\lambda-(l+1),r}(\R^3)}\\
&=\norm{\de_t^l((u^\eps-u_{\mathcal{S}}^\eps)\times \omega^\eps)}_{C^{\lambda-l,r}(\R^3)}\leq C(l,R_0,D_0)\left(1+\sum_{k=0}^l\left|\frac{\mathrm{d}^{k}}{\mathrm{d}t^k}p^\eps\right|+\norm{\de_t^{k}\omega^\eps}_{C^{\lambda-l,r}(\R^3)}\right)^2.
\end{align*}
We hence see the statement about the time regularities in the theorem by strong induction in $l$.

\end{proof}
%
%
%
%

\bigskip \ \par \noindent {\bf Acknowledgements.}  
D.M. was supported by the European Research Council (ERC) under the European Union’s Horizon 2020 research and innovation programme through the grant agreement 862342. He further wishes to thank the University of Bordeaux for its hospitality and the MM PhD Outgoing Programme of the Universität Münster for funding the visit.

F.S. was supported by the project ANR-23-CE40-0014-01 BOURGEONS 
 sponsored by  the French National Research Agency (ANR)  and the project ANR-24-CE92-0028-01 SUSPENSIONS, jointly sponsored by the ANR and the German Research Foundation DFG.

%
\bigskip


  \bibliography{dfo}
    \bibliographystyle{plain}

\begin{flushleft}
\medskip

Olivier Glass, 
Universit\'e Paris-Dauphine PSL,
CEREMADE, UMR CNRS 7534, 
Place du Mar\'echal de Lattre de Tassigny, 75775 Paris Cedex 16, France \par\nopagebreak 
\noindent  \textit{E-mail address:} \texttt{glass@ceremade.dauphine.fr}
\end{flushleft}
\medskip

\begin{flushleft}

David Meyer, Instituto de Ciencias Matemáticas, Calle Nicol\'as Cabrera 13-15, 28049 Madrid, Spain \par\nopagebreak 

\noindent  \textit{E-mail address:} \texttt{david.meyer@icmat.es}

\end{flushleft}
\medskip

\begin{flushleft}
Franck Sueur, Department of Mathematics, 
Maison du nombre, 6 avenue de la Fonte, 
University of Luxembourg,  
L-4364 Esch-sur-Alzette, Luxembourg \par\nopagebreak 

\noindent  \textit{E-mail address:} \texttt{Franck.Sueur@uni.lu}

\end{flushleft}

\end{document}